\documentclass[12pt]{amsart}

\usepackage{amsmath,amsthm,amscd,amsfonts,amssymb,epic,eepic}
\usepackage[pagebackref,colorlinks=true,linkcolor=blue,citecolor=blue]{hyperref}

\newtheorem{thm}{Theorem}[section]
\newtheorem{cor}[thm]{Corollary}
\newtheorem{conj}[thm]{Conjecture}
\newtheorem{lem}[thm]{Lemma}

\newtheorem{prop}[thm]{Proposition}
\theoremstyle{remark}
\newtheorem*{rem}{Remark}

\newcounter{remarkscounter}
\newenvironment{remarks}
{\medskip\noindent{\it
Remarks.}\begin{list}{{\rm(\arabic{remarkscounter})}
}{\usecounter{remarkscounter}

\setlength{\labelsep}{\fill} \setlength{\leftmargin}{0pt}
\setlength{\itemindent}{\fill}
\setlength{\labelwidth}{\fill}\setlength{\topsep}{0pt}
\setlength{\listparindent}{0pt}}} {\end{list}}

\numberwithin{equation}{section}

\newcommand{\A}{\mathbb{A}}

\newcommand{\GL}{\mathrm{GL}}
\newcommand{\SL}{\mathrm{SL}}
\newcommand{\ZZ}{\mathbb{Z}}

\newcommand{\Gal}{\mathrm{Gal}}

\newcommand{\QQ}{\mathbb{Q}}

\newcommand{\lto}{\longrightarrow}

\newcommand{\OO}{\mathcal{O}}
\newcommand{\CC}{\mathbb{C}}
\newcommand{\RR}{\mathbb{R}}

\newcommand{\mm}{\mathfrak{m}}
\newcommand{\nn}{\mathfrak{n}}

\newcommand{\N}{\mathrm{N}}

\newcommand{\quash}[1]{}

\theoremstyle{definition}
\newtheorem{defn}[thm]{Definition}

\renewcommand{\bar}{\overline}

\numberwithin{equation}{subsection}

\renewcommand{\hat}{\widehat}

\newcommand{\p}{{\sf p}}

\newcommand{\Spec}{\mathrm{Spec}}

\begin{document}
\title{An approach to nonsolvable base change and descent}
\author{Jayce R.~ Getz}
\address{Department of Mathematics and Statistics\\
McGill University\\
Montreal, QC, H3A 2K6}
\email{jgetz@math.mcgill.ca}
\subjclass[2000]{Primary 11F70}

\begin{abstract}
We present a collection of conjectural trace identities and explain why they are equivalent to base change and descent of automorphic representations of $\GL_n(\A_F)$ along nonsolvable extensions (under some simplifying hypotheses).  The case $n=2$ is treated in more detail and applications towards the Artin conjecture for icosahedral Galois representations are given.
\end{abstract}

\maketitle

\tableofcontents

\section{Introduction}

Let $F$ be a number field and let $v$ be a nonarchimedian place of $F$.
 By the local Langlands correspondence, now a theorem due to Harris and Taylor building on work of Henniart, there is a bijection
\begin{align} \label{loc-Langl}
\left(\varphi_v:W_{F_v}' \to \GL_n(\CC)\right) \longmapsto \pi(\varphi_v)
\end{align}
between equivalence classes of Frobenius semisimple representations $\varphi_v$ of the local Weil-Deligne group $W_{F_v}'$ and isomorphism classes of irreducible admissible representations of $\GL_n(F_v)$.  The bijection is characterized uniquely by certain compatibilities involving $\varepsilon$-factors and $L$-functions which are stated precisely in \cite{HT} (see also \cite{PreuveHenn}).  The corresponding statement for $v$ archimedian was proven some time ago by Langlands \cite{LanglandsArch}.
We write $\varphi_v(\pi_v)$ for any representation attached to $\pi_v$ and call it the \textbf{Langlands parameter} or \textbf{$L$-parameter} of $\pi_v$; it is unique up to equivalence of representations.

Now let $E/F$ be an extension of number fields, let $v$ be a place of $F$ and let $w|v$ be a place of $E$.  We say that an admissible irreducible representation $\Pi_w$ of $\GL_n(E_w)$ is a \textbf{base change} of $\pi_v$ and write $\pi_{vEw}:=\Pi_w$ if
$$
\varphi(\pi_v)|_{W_{E_w}'}\cong \varphi(\Pi_w).
$$
In this case we also say that $\Pi_w$ \textbf{descends} to $\pi_v$.
We say that an isobaric\footnote{For generalities on isobaric automorphic representations see  \cite{LanglEinM} and \cite{JSII}.} automorphic representation $\Pi$ of $\GL_n(\A_E)$ is a \textbf{base change} (resp.~\textbf{weak base change}) of an isobaric automorphic representation $\pi$ of $\GL_n(\A_F)$ if $\Pi_w=\pi_{vE}$ for all (resp.~almost all) places $v$ of $F$ and all places $w|v$ of $E$.  If $\Pi$ is a (weak) base change of $\pi$, then we also say $\Pi$ descends (weakly) to $\pi$.  We write $\pi_E$ for a weak base change of $\pi$, if it exists; it is uniquely determined up to isomorphism by the strong multiplicity one theorem \cite[Theorem 4.4]{JSII}.  If $\Pi$ is a weak base change of $\pi$, we say that the base change is compatible at a place $v$ of $F$ if $\Pi_w$ is a base change of $\pi_v$ for all $w|v$.

If $E/F$ is a prime degree cyclic extension, then the work of Langlands \cite{Langlands} for $n=2$ and Arthur-Clozel \cite{AC} for $n$ arbitrary implies that a base change always exists.  The fibers and image of the base change are also described in these works.  Given that any finite degree Galois extension $E/F$ contains a family of subextensions $E=E_0 \geq E_1 \geq \cdots \geq E_n=F$ where $E_i/E_{i+1}$ is Galois with simple Galois group, to complete the theory of base change it is necessary to understand base change and descent with respect to Galois extensions with simple nonabelian Galois group.  In this paper we introduce a family of conjectural trace identities that are essentially equivalent to proving base change and descent in this setting.  The (conjectural) trace identity is based on combining two fundamental paradigms pioneered by Langlands, the second still in its infancy:
\begin{itemize}
\item Comparison of trace formulae, and
\item Beyond endoscopy.
\end{itemize}
The point of this paper is to provide
some evidence that proving the conjectural trace identities unconditionally is a viable strategy for proving nonsolvable base change.

\subsection{Test functions}
Fix an integer $n \geq 1$, a Galois extension of number fields $E/F$, an automorphism $\tau \in \Gal(E/F)$, and a set of places $S_0$ of $E$ containing the infinite places and the places where $E/F$ is ramified.
Let $w \not \in S_0$ and let $v$ be the place of $F$ below $w$.
Let
$$
A=\begin{pmatrix} t_{1w} & & \\ & \ddots & \\ & & t_{nw}\end{pmatrix}\quad \textrm{ and } \quad A^{\tau}=\begin{pmatrix} t_{1w^{\tau}} & & \\ & \ddots & \\& & t_{nw^{\tau}}\end{pmatrix}.
$$
We view these as matrices in $\prod_{w|v} \GL_n(\CC[t_{1w}^{\pm 1},\cdots,t_{nw}^{\pm 1}])$
For $j \in \ZZ_{> 0}$ let $\mathrm{Sym}^j:\GL_n \to \GL_{\binom{n+j-1}{j}}$ be the $j$th symmetric power representation, where $\binom{m}{j}$ is the $m$-choose-$j$ binomial coefficient.

For a prime power ideal $\varpi_w^j$ of $\OO_{E}$ define a test function
\begin{align}
f(\varpi_w^j):=\mathcal{S}^{-1}(\mathrm{tr}(\mathrm{Sym}^j(A \otimes (A^{\tau})^{-1}))) \in C_c^{\infty}(\GL_n(E \otimes_F F_v)//\GL_n(\OO_{E} \otimes_{\OO_F}\OO_{F_v}))
\end{align}
where $\mathcal{S}$ is the Satake isomorphism (see \S \ref{ssec-uha}).  We denote by $f(\OO_{E_w})$ the characteristic function of $\GL_n(\OO_E \otimes_{\OO_F} \OO_{F_v})$ and regard the $f(\varpi_w^j)$ as elements of $C_c^{\infty}(\GL_n(\A_E^{\infty})//\GL_n(\widehat{\OO}_E))$.  Define $f(\nn) \in C_c^{\infty}(\GL_{n}(\A_{E}^{\infty})//\GL_{n}(\widehat{\OO}_E))$ in general by declaring that $f$ is multiplicative, that is, if $\nn+\mm=\OO_E$ we set
$$
f(\nn\mm):=f(\nn)*f(\mm)
$$
where the asterisk denotes convolution in the Hecke algebra.  If $\nn$ is coprime to $S_0$, we often view $f(\nn)$ as an element of $C_c^{\infty}(\GL_n(\A_E^{S_0})//\GL_n(\widehat{\OO}_E^{S_0}))$.

Assume that $\Pi$ is an isobaric automorphic representation of $\GL_n(\A_E)$ unramified outside of $S_0$.  Define $\Pi^{\tau}$ by $\Pi^{\tau}(g):=\Pi(g^{\tau})$.  The purpose of defining the operators $f(\mm)$ is the following equality:
$$
\sum_{\mm \subset \OO_{E}^{S_0}} \frac{\mathrm{tr}(\Pi^{S_0})(f(\mm))}{|\N_{F/\QQ}(\mm)|^{s}}=L^{S_0}(s,\Pi \times \Pi^{\tau})
$$
This follows from \eqref{RS-descr} and the fact that, in the notation of loc.~cit., $A(\Pi^{\tau}_w)=A(\Pi_{w^{\tau}})$.

Let $\phi \in C_c^{\infty}(0,\infty)$ be nonnegative.  Thus $\widetilde{\phi}(1) >0$, where
$$
\widetilde{\phi}(s):=\int_{0}^{\infty}\phi(s)x^{s-1}dx
$$
is the Mellin transform of $\phi$.
We introduce the following test function, a modification of that considered by Sarnak in \cite{Sarnak}:
\begin{align} \label{Sig-func}
\Sigma^{S_0}(X):=\Sigma_{\phi}^{S_0}(X):=\sum_{\mm \subset \OO_E^{S_0}} \phi(X/|\N_{E/\QQ}(\mm)|)f(\mm).
\end{align}

\subsection{Conjectural trace identities}
Assume that $E/F$ is a Galois extension.
For convenience, let
\begin{align}
\Pi_n(F):&=\{\textrm{Isom.~classes of isobaric automorphic representations of }\GL_n(\A_F)\}\\
\Pi_n^0(F):&=\{\textrm{Isom.~classes of cuspidal automorphic representations of }\GL_n(\A_F)\} \nonumber
\\ \Pi_n^{\mathrm{prim}}(E/F):&=\{ \textrm{Isom.~classes of $E$-primitive automorphic representations of }\GL_n(\A_F)\}.
\nonumber
\end{align}
The formal definition of an $E$-primitive automorphic representation is postponed until \S \ref{ssec-primitive}.  If we knew Langlands functoriality we could characterize them easily
as those representations that are cuspidal and not automorphically induced from an automorphic representation of a subfield $E \geq K > F$.
We note that there is a natural action of $\Gal(E/F)$ on  $\Pi_n(E)$ that preserves $\Pi_n^0(E)$; we write $\Pi_n(E)^{\Gal(E/F)}$ for those representations that are isomorphic to their $\Gal(E/F)$-conjugates and $\Pi_n^0(E)^{\Gal(E/F)}=\Pi_n^0(E) \cap \Pi_n(E)^{\Gal(E/F)}$.

Let $S$ be a finite set of places of $F$ including all infinite places and let $S'$, $S_0$ be the set of places of $F'$, $E$ lying above $S$.  Assume that $h \in C_c^{\infty}(\GL_n(\A_{F'}))$, $\Phi \in C_c^{\infty}(\GL_n(\A_F))$ are transfers of each other in the sense of \S \ref{ssec-transfers} below and that they are unramified outside of $S'$ and $S$, that is, invariant under right and left multiplication by $\GL_n(\widehat{\OO}_{F'}^{S'})$ and $\GL_n(\widehat{\OO}_F^S)$, respectively.
For the purposes of the following theorems, if $G$ is a finite group let $G^{\mathrm{ab}}$ be the maximal abelian quotient of $G$.

Assume for the remainder of this subsection that $\Gal(E/F)$ is the universal perfect central extension of a finite simple nonabelian group.  Let $E \geq F' \geq F$ be a subfield such that $\Gal(E/F')$ is solvable and $H^2(\Gal(E/F'),\CC^{\times})=0$.  Moreover let $\tau \in \Gal(E/F)$ be an element such that
$$
\Gal(E/F)=\langle \tau,\Gal(E/F')\rangle.
$$

\begin{rem}
In \S \ref{ssec-upce} we discuss these assumptions, the upshot being that they are no real loss of generality.
\end{rem}

Our first main theorem is the following:
\begin{thm}  \label{main-thm-1}
Consider the following hypotheses:
\begin{itemize}
\item $\Gal(E/F')$ is solvable, $H^2(\Gal(E/F'),\CC^{\times})=0$ and $[E:F']$ is coprime to $n$.
  \item For all divisors $m|n$ there is no irreducible nontrivial representation
  $$
  \Gal(E/F) \lto \GL_m(\CC),
  $$
 \item The case of Langlands functoriality explicated in Conjectures \ref{conj-1} below is true for $E/F$, and
 \item The case of Langlands functoriality explicated in Conjecture \ref{conj-solv} is true for $E/F'$.
\end{itemize}
If these hypotheses are valid and $h$ and $\Phi$ are transfers of each other then the limits
\begin{align} \label{11}
\lim_{X \to \infty}|\Gal(E/F')^{\mathrm{ab}}|^{-1}X^{-1}
\sum_{\pi' \textrm{ $E$-primitive}} \mathrm{tr}(\pi')(h^1b_{E/F'}(\Sigma_{\phi}^{S_0}(X)))
\end{align}
and
\begin{align} \label{12}
\lim_{X \to \infty} X^{-1}\sum_{\pi} \mathrm{tr}(\pi)(
\Phi^1b_{E/F}(\Sigma_{\phi}^{S_0}(X)))
\end{align}
converge absolutely and are equal.  Here the first sum is over a set of representatives for the equivalence classes of $E$-primitive cuspidal automorphic representations of $A_{\GL_{nF'}} \backslash \GL_n(\A_{F'})$ and the second sum is over a set of representatives for the equivalence classes of cuspidal automorphic representations of $A_{\GL_{nF}} \backslash \GL_n(\A_{F})$.
\end{thm}
Here
\begin{align*}
h^1(g):&=\int_{A_{\GL_{nF'}}}h(ag)da'\\
\Phi^1:&=\int_{A_{\GL_{nF}}}\Phi(ag)da
\end{align*}
where the $da'$ and $da$ are the Haar measures on $A_{\GL_{nF'}}$ and $A_{\GL_{nF}}$, respectively, used in the definition of the transfer.  For the definition of $A_{\GL_{nF'}}$ and $A_{\GL_{nF}}$ we refer to \S \ref{HC-subgroup} and for the definition of $b_{E/F}$ and $b_{E/F'}$ we refer to \eqref{bEF}.

\begin{remarks}
\item If we fix a positive integer $n$, then for all but finitely many finite simple groups $G$ with universal perfect central extensions $\widetilde{G}$ any representation $\widetilde{G} \to \GL_{n}$ will be trivial.  This follows from \cite[Theorem 1.1]{LS}, for example (the author does not known if it was known earlier).
Thus the second hypothesis in Theorem \ref{main-thm-1} holds for almost all groups (if we fix $n$).  In particular, if $n=2$, then the only finite simple nonabelian group admitting a projective representation of degree $2$ is $A_5$ by a well-known theorem of Klein.  Thus when $n=2$ and $\Gal(E/F)$ is the universal perfect central extension of a finite simple group other than $A_5$ the first hypothesis of Theorem \ref{main-thm-1} holds.

\item Conjecture \ref{conj-1} and its analogues conjectures \ref{conj-2}, \ref{conj-32} and \ref{conj-33} below each amount to a statement that certain (conjectural) functorial transfers of automorphic representations exist and have certain properties.  To motivate these conjectures, we state and prove the properties of $L$-parameters to which they correspond in propositions \ref{prop-bij-EF'}, \ref{prop-A5-EF} and lemmas \ref{lem-A5-EF}, \ref{lem-A5-EF3} respectively.  The facts about $L$-parameters we use are not terribly difficult to prove given basic facts from finite group theory, but they are neither obvious nor well-known, and one of the motivations for this paper is to record them.

\item Conjecture \ref{conj-solv} is a conjecture characterizing the image and fibers of solvable base change.  Rajan \cite{Rajan3} has explained how it can be proved (in principle) using a method of Lapid and Rogawski \cite{LR} together with the work of Arthur and Clozel \cite{AC}.  It is a theorem when $n=2$ \cite[Theorem 1]{Rajan3} or when $\Gal(E/F)$ is cyclic of prime degree \cite[Chapter 3, Theorems 4.2 and 5.1]{AC}.
\end{remarks}

The following weak converse of Theorem \ref{main-thm-1} is true:

\begin{thm} \label{main-thm-1-conv}
Assume Conjecture \ref{conj-transf} on transfers of test functions and assume that $F$ is totally complex.  If \eqref{11} and \eqref{12} converge absolutely and are equal for all $h$ unramified outside of $S'$ with transfer $\Phi$ unramified outside of $S$, then every
cuspidal automorphic representation $\Pi$ of $\GL_n(\A_E)$ satisfying $\Pi^{\sigma} \cong \Pi$ for all $ \sigma \in \Gal(E/F)$ admits a unique weak descent to $\GL_n(\A_F)$.  Conversely, if $\pi$ is a cuspidal automorphic representation of $\GL_n(\A_F)$ such that
\begin{align} \label{non-zero-Sig}
\lim_{X \to \infty}X^{-1}\mathrm{tr}(\pi)(b_{E/F}(\Sigma_{\phi}^{S_0}(X))) \neq 0
\end{align}
then $\pi$ admits a weak base change to $\GL_n(\A_E)$.  The weak base change of a given cuspidal automorphic representation $\pi$ of $\GL_n(\A_F)$ is unique.  The base change is compatible for the infinite places of $F$ and the finite places $v$ of $F$ where $E/F$ and $\pi_v$ are unramified or where $\pi_v$ is a twist of the Steinberg representation by a quasi-character.
\end{thm}
Here, as usual, $\Pi^{\sigma}:=\Pi \circ \sigma$ is the representation acting on the space of $\Pi$ via
$$
\Pi^{\sigma}(g):=\Pi(\sigma(g)).
$$

\begin{rem}
Conjecture \ref{conj-transf} is roughly the statement that there are ``enough''  $h$ and  $\Phi$ that are transfers of each other.  If one assumes that $\Pi$ (resp.~$\pi$) and $E/F$ are everywhere unramified, then one can drop the assumption that Conjecture \ref{conj-transf} is valid.
\end{rem}

We conjecture that \eqref{non-zero-Sig} is always nonzero:

\begin{conj} \label{conj-nonzero} Let $\pi$ be a cuspidal automorphic representation of $A_{\GL_{nF}} \backslash \GL_n(\A_F)$; then
\begin{align*}
\lim_{X \to \infty}X^{-1}\mathrm{tr}(\pi)(b_{E/F}(\Sigma_{\phi}^{S_0}(X))) \neq 0
\end{align*}
\end{conj}
This conjecture is true for all $\pi$ that admit a base change to an isobaric automorphic representation of $\GL_n(\A_E)$ by an application of Rankin-Selberg theory (compare Proposition \ref{Perron-prop} and \eqref{ord-pole}), however, assuming this would be somewhat circular for our purposes.  The author is hopeful that Conjecture \ref{conj-nonzero} can be proven independently of the existence of the base change.  Indeed, the Chebatarev density theorem is proven despite the fact that the Artin conjecture is still a conjecture.  The smoothed sum in Conjecture \ref{conj-nonzero} is analogous to some sums that can be evaluated using the Chebatarev density theorem; in some sense the Chebatarev density theorem is the case where $\pi$ is the trivial representation of $\GL_1(\A_F)$.

Isolating primitive representations is not a trivial task.  For example, the main focus of \cite{Venk} is the isolation of cuspidal representations that are not primitive when $n=2$.  Therefore it seems desirable to have a trace identity similar to that of Theorem \ref{main-thm-1} that involves sums over all cuspidal representations.  This is readily accomplished under additional assumptions on $\Gal(E/F)$ using the following lemma:
\begin{lem} \label{lem-prim}Let $L/K$ be a Galois extension of number fields.
Suppose that there is no proper subgroup $H \leq \Gal(L/K)$ such that $[\Gal(L/K):H]|n$.  Then
$$
\Pi_n^{\mathrm{prim}}(L/K)=\Pi_n^0(K).
$$\qed
\end{lem}
The proof is immediate from the definition of $L$-primitive automorphic representations in \S \ref{ssec-primitive}.

\subsection{Icosahedral extensions}

We now consider the case of the smallest simple nonabelian group $A_5$ in more detail.  We begin by setting notation for specific subsets of $\Pi^0_n(F)$ and $\Pi^0_n(E)$.

Let $E/F$ be a Galois extension, and let
$$
\rho:W_F' \lto {}^L\GL_{nF}
$$
be an $L$-parameter trivial on $W_E'$; thus $\rho$ can essentially be identified with the Galois representation $\rho_0:\Gal(E/F) \to \GL_{n}(\CC)$ obtained by composing $\rho$ with the projection ${}^L \GL_{nF} \to \GL_n(\CC)$.   For every quasi-character $\chi:F^{\times} \backslash \A_F^{\times} \cong (W_{F}')^{\mathrm{ab}} \to \CC^{\times} $ we can then form the $L$-parameter
$$
\rho \otimes \chi:W_{F}' \lto \GL_n(\CC).
$$
We say that a cuspidal automorphic representation $\pi$ of $\GL_n(\A_F)$ is \textbf{associated} to $\rho \otimes \chi$ if $\pi_v$ is the representation attached to the $L$-parameter $(\rho \otimes \chi)_v$:
$$
\pi_v=\pi((\rho \otimes \chi)_v)
$$
for almost all places $v$ of $F$ (see \eqref{loc-Langl} above).  If $\pi_v=\pi((\rho \otimes \chi)_v)$ for all places $v$, then we write $\pi=\pi(\rho \otimes \chi)$.  In this case we also say that $\pi$ and $\rho \otimes \chi$ are \textbf{strongly associated}.
More generally, if $\pi$ is a cuspidal automorphic representation of $\GL_n(\A_F)$ such that $\pi$ is associated to $\rho \otimes \chi$ for some $\chi$ we say that $\pi$ is of \textbf{$\rho$-type}.  If $\pi$ is associated to $\rho \otimes \chi$ for some $\rho$ and $\chi$ we say that $\pi$ is of \textbf{Galois type}.

Assume for the remainder of this section that $\Gal(E/F) \cong \widetilde{A}_5$, the universal perfect central extension of the alternating group $A_5$ on $5$ letters.  One can formulate analogues of theorems \ref{main-thm-1} and \ref{main-thm-1-conv} in this setting.
  For this purpose, fix an embedding $A_4 \hookrightarrow A_5$, and let
$\widetilde{A}_4 \leq \widetilde{A}_5$  be the preimage of $A_4$ under the surjection $ \widetilde{A}_5 \to A_5$.  Thus $\widetilde{A}_4$ is a nonsplit double cover of $A_4$.

\begin{thm}  \label{main-thm-2}
Let $n=2$, let $F' = E^{\widetilde{A}_4}$, and let $\tau \in \Gal(E/F)$ be any element of order $5$.  Let $h \in C_c^{\infty}(\GL_2(\A_{F'}))$ be unramified outside of $S'$ and have transfer $\Phi \in C_c^{\infty}(\GL_2(\A_F))$  unramified outside of $S$.  Assume the case of Langlands functoriality explicated in Conjecture \ref{conj-2} for $E/F$.  Then the limits
\begin{align} \label{A21}
2\lim_{X \to \infty}\left(\frac{d^{3}}{ds^{3}}(\widetilde{\phi}(s)X^s)|_{s=1}\right)^{-1}|\Gal(E/F')^{\mathrm{ab}}|^{-1}
\sum_{\pi'} \mathrm{tr}(\pi')(h^1b_{E/F'}(\Sigma_{\phi}^{S_0}(X)))
\end{align}
and
\begin{align} \label{A22}
 \lim_{X \to \infty} \left(\frac{d^{3}}{ds^{3}}(\widetilde{\phi}(s)X^s)|_{s=1}\right)^{-1} \sum_{\pi } \mathrm{tr}(\pi)(\Phi^1b_{E/F}(\Sigma_{\phi}^{S_0}(X)))
\end{align}
converge absolutely and are equal.
Similarly, again assuming Conjecture \ref{conj-2} below, the limits
\begin{align} \label{B21}
\lim_{X \to \infty}X^{-1}|\Gal(E/F')^{\mathrm{ab}}|^{-1}
\sum_{\substack{\pi' \textrm{ not of $\rho$-type for $\rho$ trivial on $W_E'$}}} \mathrm{tr}(\pi')(h^1b_{E/F'}(\Sigma_{\phi}^{S_0}(X)))
\end{align}
and
\begin{align} \label{B22}
 \lim_{X \to \infty} X^{-1} \sum_{\substack{\pi \textrm{ not of $\rho$-type for $\rho$ trivial on $W_E'$}}} \mathrm{tr}(\pi)(\Phi^1b_{E/F}(\Sigma^{S_0}_{\phi}(X)))
\end{align}
converge absolutely and are equal.
In both cases the first sum is over a set of representatives for the equivalence classes of cuspidal automorphic representations of $A_{\GL_{2F'}} \backslash \GL_2(\A_{F'})$ and the second sum is over a set of representatives for the equivalence classes of cuspidal automorphic representations of $A_{\GL_{2F}} \backslash \GL_2(\A_{F})$.
\end{thm}

Again, a converse statement is true:

\begin{thm} \label{main-thm-2-conv}
Assume Conjecture \ref{conj-transf} on transfers of test functions, assume that $F$ is totally complex, and assume that the limits \eqref{B21} and \eqref{B22} converge absolutely for all test functions  $h$ unramified outside of $S'$  with transfer $\Phi$ unramified outside of $S$.  Under these assumptions every cuspidal automorphic representation $\Pi$ of $\GL_2(\A_E)$ that is isomorphic to its
$\Gal(E/F)$-conjugates is a weak base change of a unique cuspidal automorphic representation of $\GL_2(\A_F)$.
Conversely, if $\pi$ is a cuspidal automorphic representation of $\GL_2(\A_F)$ such that
\begin{align*}
\lim_{X \to \infty}X^{-1}\mathrm{tr}(\pi)(b_{E/F}(\Sigma_{\phi}^{S_0}(X))) \neq 0
\end{align*}
then $\pi$ admits a unique weak base change to $\GL_2(\A_F)$.  If $\pi$ is a cuspidal automorphic representation of $\GL_2(\A_F)$ that is not of $\rho$-type for $\rho$ trivial on $W_{F}'$, then $\pi_E$ is cuspidal.  The base change is compatible at the infinite places of $F$ and the finite places $v$ of $F$ where $E/F$ and $\pi_v$ are unramified or $\pi_v$ is a twist of the Steinberg representation by a quasi-character.
\end{thm}

\subsection{On the Artin conjecture for icosahedral representations}

As in the last subsection we assume that $\Gal(E/F) \cong \widetilde{A}_5$.  Fix an embedding $\ZZ/2 \times \ZZ/2 \hookrightarrow A_5$ and let $Q \hookrightarrow \widetilde{A}_5$ be the inverse image of $\ZZ/2 \times \ZZ/2$ under the quotient $\widetilde{A}_5 \to A_5$.  For the purposes of the following theorem, let $S_1$ be a subset of the places of $F$ disjoint from $S$  and let $S'_1$ (resp.~$S_{10}$) be the set of places of $F'$ (resp.~$E$) above $S_1$.  Moreover
 let $h^{S'_1} \in C_c^{\infty}(\GL_2(\A_{F'}^{S'})$ and $\Phi^{S_1} \in C_c^{\infty}(\GL_2(\A_F^{S_1}))$ be transfers of each other unramified outside of $S'$ and $S$, respectively, and let $h_{S_1'} \in C_c^{\infty}(\GL_2(F'_{S'_1})//\GL_2(\OO_{F'_{S'_1}}))$.

\begin{thm}  \label{main-thm-3} Consider the following hypotheses:
\begin{itemize}
\item One has $n=2$ and $F'=E^Q$, and the case of Langlands functoriality explicated in Conjecture \ref{conj-32} is true for $E/F$.
\item One has $n=3$, $F'=E^{\widetilde{A}_4}$, the case of Langlands functoriality explicated in Conjecture \ref{conj-33} is true for $E/F$, and Conjecture \ref{conj-solv} is true for $E/F'$
\end{itemize}
Under these assumptions the limits
\begin{align} \label{31}
2\lim_{X \to \infty}\left(\frac{d^{n^2-1}}{ds^{n^2-1}}(\widetilde{\phi}(s)X^s)\big|_{s=1}\right)^{-1}
\sum_{\pi' } \mathrm{tr}(\pi')((h^{S'_1})^{1}h_{S'_1}b_{E/F'}(\Sigma_{\phi}^{S_0}(X)))
\end{align}
and
\begin{align} \label{32}
 \lim_{X \to \infty} \left(\frac{d^{n^2-1}}{ds^{n^2-1}}(\widetilde{\phi}(s)X^s)\big|_{s=1}\right)^{-1} \sum_{\pi } \mathrm{tr}(\pi)((\Phi^{S_1})^1b_{F'/F}(h_{S'_1})b_{E/F}(\Sigma_{\phi}^{S_0}(X)))
\end{align}
converge absolutely and are equal.
Here the first sum is over equivalence classes of cuspidal automorphic representations of $A_{\GL_{nF'}} \backslash \GL_n(\A_{F'})$ and the second sum is over equivalence classes of cuspidal automorphic representations of $A_{\GL_{nF}} \backslash \GL_n(\A_{F})$.
\end{thm}

\begin{remarks}
\item One can always find $\tau \in \Gal(E/F)$ such that $\langle \tau, \Gal(E/F') \rangle=\Gal(E/F)$ (this follows from Theorem \ref{thm-GK}, for example, or by an elementary argument).

\item The fact that this theorem involves more general test functions than those in theorems \ref{main-thm-1} and \ref{main-thm-2} is important for applications to the Artin conjecture (see Theorem \ref{main-thm-3-conv}).
\end{remarks}

Let $\rho_2:W_F' \to {}^L\GL_{2F}$ be an irreducible $L$-parameter trivial on $W_E'$ (i.e. an irreducible Galois representation $\rho_2:\Gal(E/F) \to \GL_2(\CC)$).  Its character takes values in $\QQ(\sqrt{5})$ and if $\langle \xi \rangle =\Gal(\QQ(\sqrt{5})/\QQ)$ then $\xi \circ \rho_2$ is another irreducible $L$-parameter that is not equivalent to the first (see \S \ref{appendix}).
A partial converse of Theorem \ref{main-thm-3} above is the following:

\begin{thm} \label{main-thm-3-conv} Assume Conjecture \ref{conj-transf} and that \eqref{31} and \eqref{32} converge and are equal for all test functions as in Theorem \ref{main-thm-3} for $n \in \{2,3\}$.  Assume moreover that $F$ is totally complex.  Then there is a pair of nonisomorphic cuspidal automorphic representations $\pi_1,\pi_2$ of $\GL_2(\A_F)$ such that
$$
\pi_{1} \boxplus \pi_2 \cong \pi((\rho_{2} \oplus \xi \circ \rho_2)).
$$
\end{thm}
Here the $\boxplus$ denotes the isobaric sum \cite{LanglEinM} \cite{JSII}.

\begin{rem}  It should be true that, upon reindexing if necessary, $\pi_{1} \cong \pi(\rho_{2})$.  However, the author does not know how to prove this at the moment.
\end{rem}

As a corollary of this theorem and work of Kim and Shahidi, we have the following:

\begin{cor} \label{cor-artin-cases}
Under the hypotheses of Theorem \ref{main-thm-3-conv}, if $\rho:\Gal(E/F)\to \GL_n(\CC)$ is an irreducible Galois representation of degree strictly greater than $3$, then there is an automorphic representation $\pi$ of $\GL_n(\A_F)$ such that $\pi=\pi(\rho)$.
\end{cor}

The point of the theorems above is that the sums \eqref{11}, \eqref{12} and their analogues in the other theorems can be rewritten in terms of orbital integrals using either the trace formula (compare \cite{LanglBeyond}, \cite{FLN}) or the relative trace formula (specifically the Bruggeman-Kuznetsov formula, compare \cite{Sarnak}, \cite{Venk})\footnote{We note that when applying the relative trace formula the distributions $h \mapsto \mathrm{tr}(\pi')(h^1)$ will be replaced by Bessel distributions defined using Whittaker functionals. Thus the definition of $\Sigma_{\phi}^{S_0}(X)$ has to be modified to be useful in a relative trace formula approach.}.  One then can hope to compare these limits of orbital integrals and prove nonsolvable base change and descent.  The author is actively working on this comparison.  He hopes that the idea of comparing limiting forms of trace formulae that underlies theorems \ref{main-thm-1-conv}, \ref{main-thm-2-conv}, \ref{main-thm-3-conv} will be useful to others working ``beyond endoscopy.''

To end the introduction we outline the sections of this paper.
Section \ref{sec-notat} states notation and conventions; it can be safely
skipped and later referred to if the reader encounters unfamiliar notation.
In \S \ref{sec-tf}, we review unramified base change, introduce
a notion of transfer for test functions, and prove the existence of the
transfer in certain cases.  Section  \ref{sec-limit-cusp} introduces the
smoothed test functions used in the statement and proof of our main theorems
and develops their basic properties using Rankin-Selberg theory.
Perhaps the most important result is that the trace of these test
functions over the cuspidal spectrum is well-defined and picks
out the representations of interest (see Corollary \ref{cor-aut-trace}).
The behavior of $L$-parameters under restriction along an
extension of number fields is considered in \S \ref{sec-rest-desc};
this is used to motivate the conjectures appearing in our main theorems above,
which are also stated precisely in \S \ref{sec-rest-desc}.
Section \ref{sec-proofs} contains the proofs of the theorems stated above and the proof of Corollary \ref{cor-artin-cases}.
Finally, in \S \ref{sec-groups} we explain why the group-theoretic
assumptions made in theorems \ref{main-thm-1} and \ref{main-thm-2}
are essentially no loss of generality.

\section{General notation} \label{sec-notat}

\subsection{Ad\`eles}
The ad\`eles of a number field $F$ will be denoted by $\A_F$.  We write $\widehat{\OO}_{F}:=\prod_{v \textrm{ finite}}\OO_{F_v}$. For a
set of places $S$ of $F$ we write $\A_{F,S}:=\A_F \cap \prod_{v \in
S}F_v$ and $\A^S_F:=\A_F \cap \prod_{v \not \in S}F_v$.
If $S$ is finite we sometimes write $F_S:=\A_{FS}$.  The set of
infinite places of $F$ will be denoted by $\infty$. Thus
$\A_{\QQ,\infty}=\RR$ and $\A_{\QQ}^{\infty}:=\prod_{\substack{p \in
\ZZ_{>0}
\\p \textrm{ prime}}}\QQ_p$. For an affine $F$-variety $G$ and a
subset $W \leq G(\A_F)$ the notation $W_{S}$ (resp. $W^S$) will
denote the projection of $W$ to $G(\A_{F,S})$ (resp. $G(\A^S_F)$).
If $W$ is replaced by an element of $G(\A_F)$, or if $G$ is an
algebraic group and $W$ is replaced by a character of $G(\A_F)$ or a
Haar measure on $G$, the same notation will be in force; e.g. if
$\gamma \in G(\A_F)$ then $\gamma_v$ is the projection of $\gamma$ to $G(F_v)$.

If $w,v$ are places of $E,F$ with $w|v$ we let $e(E_w/F_v)$ (resp.~$f(E_w/F_v)$) the ramification degree (resp.~inertial degree) of $E_w/F_v$.

\subsection{Restriction of scalars}

Let $A \to B$ be a morphism of $\ZZ$-algebras and let $X \to \mathrm{Spec}(B)$ be a $\mathrm{Spec}(B)$-scheme.  We denote by
$$
\mathrm{R}_{B/A}(X) \to \Spec(A)
$$
the Weil restriction of scalars of $X$.  We will only use this functor in cases
where the representability of $\mathrm{R}_{B/A}(X)$ by a scheme is well-known.   If $X \to \Spec(A)$, we often abbreviate
$$
\mathrm{R}_{B/A}(X):=\mathrm{R}_{B/A}(X_B).
$$

\subsection{Characters}

If $G$ is a group we let $G^{\wedge}$ be the group of abelian characters of $G$.  Characters are always assumed to be unitary.  A general homomorphism $G \to \CC^{\times}$ will be called a quasi-character.  If $E/F$ is a Galois extension of number fields, we often identify
$$
\Gal(E/F)^{\wedge}=F^{\times} \backslash \A_F^{\times}/\N_{E/F}(\A_E^{\times})
$$
using class field theory.

\subsection{Harish-Chandra subgroups} \label{HC-subgroup}
Let $G$ be a connected reductive group over a number field $F$.  We write $A_G \leq Z_G(F \otimes_{\QQ} \RR)$ for the connected component of the real points of the largest $\QQ$-split torus in the center of $\mathrm{R}_{F/\QQ}G$.  Here when we say ``connected component'' we mean in the real topology.  Write
$X^*$ for the group of $F$-rational characters of $G$ and set
$\mathfrak{a}_G:=\mathrm{Hom}(X^*,\RR)$. There is a morphism
\begin{align*}
HC_G:G(\A_F) \lto \mathfrak{a}_G
\end{align*}
defined by
\begin{align}
\langle HC_G(x),\chi\rangle =|\log(x^{\chi})|
\end{align}
for $x \in G(\A_F)$ and $\chi \in X^*$.  We write
\begin{align}
G(\A_F)^1:=\ker(HC_G).
\end{align}
and refer to it as the Harish-Chandra subgroup of $G(\A_F)$.  Note that $G(F) \leq G(\A_F)^1$ and
$G(\A_F)$ is the direct product of $A_G$ and $G(\A_F)^1$.
We say that $\pi$ is an \textbf{automorphic representation of $A_{GF} \backslash G(\A_F)$} if it is an automorphic representation of $G(\A_F)$ trivial on $A_{GF}$ (and therefore unitary).

\subsection{Local fields}
\label{ssec-loc-fields}
A uniformizer for a local field equipped with a discrete valuation will be denoted by $\varpi$. If $F$ is a global field and $v$ is a non-archimedian place of $F$ then we will write $\varpi_v$ for a choice of uniformizer of $F_v$.  The number of elements in the residue field of $F_v$ will be denoted by $q_v$, and we write
$$
| \cdot|_v:F_v \lto \RR_{\geq 0}
$$
 for the $v$-adic norm, normalized so $|\varpi_v|_v=q_v^{-1}$.  For an infinite place $v$, we normalize
$$
|a|_v:=\begin{cases} |a| &\textrm{(the usual absolute value) if }v \textrm{ is real}\\
a\bar{a} & \textrm{(the square of the usual absolute value) if }v \textrm{ is complex.} \end{cases}
$$

\subsection{Field extensions} \label{ssec-fe}

In this paper we will often deal with a tower of field extensions $E \geq F' \geq F$.  When in this setting we have attempted to adhere to the following notational scheme:

\begin{center}
\begin{tabular}{ l | c |c |c | c}
 & Place & Set of places & Test function & Representation \\
\hline
$E$ & $w$ & $S_0$ & $f$ & $\Pi$\\
$F'$ & $v'$ & $S'$ & $h$ & $\pi'$\\
$F$ & $v$ & $S$ & $\Phi$ & $\pi$
\end{tabular}
\end{center}
Thus, e.g. $w$ will be a place of $E$ above the place $v$ of $F$ and $h$ will denote an element of $C_c^{\infty}(\GL_2(F_{v'}'))$ for some place $v'$ of $F'$.

\section{Test functions}
\label{sec-tf}
In this section we recall
basic results on test functions that are used later in the paper.  In \S \ref{ssec-uha}
we set notation for unramified Hecke algebras and the Satake isomorphism.  In \S \ref{ssec-bcformal} we recall
the usual base change map on unramified Hecke algebras and in \S \ref{ssec-transfers} we define a notion of transfer.

\subsection{Unramified Hecke algebras} \label{ssec-uha}

For each positive integer $n$ let
$T_n \leq \GL_n$ be the standard diagonal maximal torus and let
\begin{align}
X_*(T_n) \cong \ZZ^n=\{\lambda:=(\lambda_1, \cdots, \lambda_n)\}
\end{align}
be the group of rational cocharacters.

Let $F_v$ be the completion of a global field $F$ at some non-archimedian place $v$ and let $\varpi_v$ be a uniformizer of $F_v$.  We write
$$
\mathbf{1}_{\lambda}:=\mathrm{ch}_{\GL_n(\OO_{F_v}) \varpi_v^{\lambda} \GL_n(\OO_{F_v})} \in C_c^{\infty}(\GL_n(F_v)//\GL_n(\OO_{F_v}))
$$
for the characteristic function of the double coset
$$
\GL_{n}(\OO_{F_v})\varpi_v^{\lambda}\GL_n(\OO_{F_v}) \in \GL_n(\OO_{F_v}) \backslash \GL_n(F_v) /\GL_n(\OO_{F_v}).
$$

Let $\hat{T}_n \leq \hat{\GL}_n$ denote the dual torus in the (complex) connected dual group.  We let
$$
\mathcal{S}:C_c^{\infty}(\GL_n(F_v)//\GL_n(\OO_{F_v})) \lto \CC[X^*(\widehat{T}_n)]^{W(T_n,\GL_n)} =\CC[t_1^{\pm 1}, \dots,t_n^{\pm 1}]^{S_n}
$$
denote the Satake isomorphism, normalized in the usual manner (see, e.g. \cite[\S 4.1]{Laumon}).  Here $W(T_n,\GL_n)$ is the Weyl group of $T_n$ in $\GL_n$; it is well-known that $W(T_n,\GL_n) \cong S_n$, the symmetric group on $n$ letters.

Let $E/F$ be a field extension and let $v$ be a finite place of $F$.
For the purpose of setting notation, we recall that the Satake isomorphism for $\mathrm{R}_{E/F}\GL_n(F_v)$ induces an isomorphism
$$
\mathcal{S}:C_c^{\infty}(\mathrm{R}_{E/F}\GL_n(F_v)//\mathrm{R}_{\OO_{E}/\OO_F}\GL_n(\OO_{F_v})) \tilde{\lto} \otimes_{w|v} \CC[t_{1w}^{\pm 1},\dots,t_{nw}^{\pm 1}]^{S_n}.
$$
Here the product is over the places $w$ of $E$ dividing $v$ and the $w$ factor of
$$
C_c^{\infty}(\mathrm{R}_{E/F}\GL_n(F_v)//\mathrm{R}_{\OO_{E}/\OO_F}\GL_n(\OO_{F_v})) \cong \prod_{w|v}C_c^{\infty}(\GL_n(E_{w})//\GL_n(\OO_{E_w})),
$$
is sent to the $w$ factor of $\otimes_{w|v} \CC[t_{1w}^{\pm 1},\dots,t_{nw}^{\pm 1}]^{S_n}$.

For a place $w|v$, write
$$
\mathbf{1}_{\lambda w} \in C_c^{\infty}(\mathrm{R}_{E/F}\GL_n(F_v)//\mathrm{R}_{\OO_{E}/\OO_F}\GL_n(\OO_{F_v}))
$$
for the product of $\mathbf{1}_{\lambda} \in C_c^{\infty}(\GL_n(E_{w})//\GL_n(\OO_{E_{w}}))$ with
$$
\prod_{\substack{w'\neq w\\ w' |v}}\mathbf{1}_{(0,\dots,0)w'} \in \prod_{\substack{w' \neq w\\w' |v}}C_c^{\infty}\left(\GL_n(E_{w'})//\GL_n(\OO_{E_{w'}})\right).
$$
Thus $\mathcal{S}(\mathbf{1}_{\lambda w})=p(t_{1w},\dots,t_{nw})$ for some polynomial $p \in \CC[x_1^{\pm 1},\dots,x_n^{\pm 1}]^{S_n}$.

\subsection{Base change for unramified Hecke algebras} \label{ssec-bcformal}
Let $E/F$ be an extension of global fields.
For any subfield $E \geq k \geq F$ we have a base change map
$$
b_{E/k}:{}^L\mathrm{R}_{k/F}\GL_n \lto {}^L\mathrm{R}_{E/F}\GL_n;
$$
it is given by the diagonal embedding on connected components:
$$
({}^L\mathrm{R}_{k/F}\GL_n)^{\circ} \cong \GL_n(\CC)^{[k:F]} \lto ({}^L\mathrm{R}_{E/F}\GL_n)^{\circ} \cong \GL_n(\CC)^{[E:F]}
$$
and the identity on the Weil-Deligne group.
Suppose that $E/F$ is unramified at a finite place $v$ of $F$.
We recall that via the Satake isomorphism the base change map $b_{E/k}$ defines an algebra homomorphism
\begin{align} \label{bEF}
b_{E/k}:C_c^{\infty}(\mathrm{R}_{E/F}\GL_n(F_v)//\mathrm{R}_{\OO_E/\OO_F}\GL_n(\OO_{F_v})) \lto C_c^{\infty}(\mathrm{R}_{k/F}\GL_n(F_v)//\mathrm{R}_{\OO_k/\OO_F}\GL_n(\OO_{F_v})).
\end{align}
In terms of Satake transforms, this map is given explicitly by
$$
b_{E/k}\left(\prod_{w|v}\mathcal{S}(f_w)(t_{1w},\dots,t_{nw})\right)=\prod_{v'|v}\prod_{w|v'}\mathcal{S}(f_w)(t_{1v'}^{i_{v'}},\dots,t_{nv'}^{i_{v'}})
$$
where the product over $v'|v$ is over the places of $k$ dividing $v$ and $i_{v'}$ is the inertial degree of $v'$ in the extension $E/k$ \cite[Chapter 1, \S 4.2]{AC}.  It satisfies the obvious compatibility condition
$$
b_{E/F}=b_{k/F} \circ b_{E/k}.
$$

Let $\pi_v$ be an irreducible admissible unramified representation of $\GL_n(F_v)$ and let $w$ be a place of $E$ above $v$. There exists an irreducible admissible
representation
$$
b_{E/F}(\pi_v)=\pi_{vEw}
$$
of $\GL_n(E_{w})$, unique up to equivalence of representations, such that
\begin{align} \label{BC-map}
\mathrm{tr}(b_{E/F}(\pi_v)(f))=\mathrm{tr}(\pi_v(b_{E/F}(f)))
\end{align}
for all
$f \in C_c^{\infty}(\GL_n(E_w)//\GL_n(\OO_{E_w}))$.  It is called the \textbf{base change} of $\pi_v$, and is
an unramified irreducible admissible representation of $\GL_n(E_w)$.  Explicitly, if $\pi_v=I(\chi)$ is the irreducible unramified constituent of the unitary induction of
an unramified character $\chi:T_n(F_v) \lto \CC^{\times}$, then $\pi_{vEw} \cong I(\chi \circ \N_{E_w/F_v})$, where
$$
\N_{E_w/F_v}:T_n(E_w) \lto T_n(F_v)
$$
is the norm map induced by the norm map $\N_{E_w/F_v}:E_w \to F_v$.  Here $I(\chi \circ \N_{E_w/F_v})$ is the irreducible unramified constituent of the unitary induction of $\chi \circ \N_{E_w/F_v}$.  The fact that
\begin{align} \label{BC-map2}
\mathrm{tr}(b_{E/F}(I(\chi \circ \N_{E_w/F_v})(f))=\mathrm{tr}(I(\chi)(b_{E/F}(f)))
\end{align}
is readily verified using the well-known formulas for the trace of an (unramified) Hecke operator in $C_c^{\infty}(\GL_n(F_v)//\GL_n(\OO_{F_v}))$ acting on a spherical representation in terms of Satake parameters (see \cite[Theorem 7.5.6]{Laumon}).

\subsection{Transfers} \label{ssec-transfers}
Let $E/F$ be a field extension.
As indicated at the beginning of this paper, the local Langlands correspondence for $\GL_n$ implies that the local base change transfer exists.  Thus any irreducible admissible representation $\pi_v$ of $\GL_n(F_v)$ admits a base change $\Pi_w=\pi_{vEw}$ to $\GL_n(E_w)$ for any place $w|v$; this representation is uniquely determined up to isomorphism by the requirement that
$$
\varphi(\pi_v)|_{W_{E_w}'} \cong \varphi(\Pi_w)
$$
where $\varphi(\cdot)$ is the $L$-parameter of $(\cdot)$.
\begin{rem}
There is a representation-theoretic definition, due to Shintani, of a base change of an admissible irreducible representation of $\GL_n(F_v)$ along a cyclic extension \cite[Chapter 1, Definition 6.1]{AC}.  One can iterate this definition along cyclic subextensions of a general solvable extension of local fields to arrive at a representation-theoretic definition of the base change of an irreducible admissible representation.  For unramified representations of non-archimedian local fields, one uses descent \cite[Lemma 7.5.7]{Laumon} to verify that the two definitions are compatible.  Similarly, it is easy to see that they are compatible for abelian twists of the Steinberg representation using compatibility of the local Langlands correspondence with twists, the fact that the Steinberg representation has a very simple $L$-parameter (namely the representation of \cite[(4.1.4)]{Tate}) and \cite[Chapter 1, Lemma 6.12]{AC}.  However, the author does not know of any reference for their compatibility in general.  It probably follows from the compatibility of the local Langlands correspondence with $L$-functions and $\varepsilon$-factors of pairs together with the local results of \cite[Chapter 1, \S 6]{AC}, but we have not attempted to check this.
\end{rem}

Now assume that $E \geq F' \geq F$ is a subfield.  Let $S$ be a set of places of $F$ and let $S'$ (resp.~$S_0$) be the set of places of $F'$ (resp.~$E$) lying above places in $S$.

\begin{defn} \label{defn-transf}
Two functions $h_{S'} \in C_c^{\infty}(\GL_n(\A_{F'S'}))$ and $\Phi_S \in C_c^{\infty}(\GL_n(\A_{FS}))$ are said to be \textbf{transfers} of each other if there is a function $f_{S_0} \in C_c^{\infty}(\GL_n(\A_{ES_0}))$ such that for all irreducible generic unitary representations $\pi_{S}$ of $\GL_n(F_S)$ one has
$$
\prod_{w \in S_0}\mathrm{tr}(\pi_{vEw})(f_w)=\prod_{v' \in S'} \mathrm{tr}(\pi_{vF'v'})(h_{v'})=\prod_{v \in S}\mathrm{tr}(\pi_{v})(\Phi_v)
$$
\end{defn}

We immediately state one conjecture suggested by this definition:

\begin{conj} \label{conj-transf} Let $S$ be a finite set of places of $F$ containing the infinite places.  If $\Pi^{\sigma} \cong \Pi$ for all $\sigma \in \Gal(E/F)$, then there exists an $f_{S_0} \in C_c^{\infty}(\GL_n(E_{S_0}))$ and $h_{S'} \in C_c^{\infty}(\GL_n(F'_{S'}))$ that admits a transfer $\Phi_{S} \in C_c^{\infty}(\GL_n(F_S))$ of positive type such that the identity of Definition \ref{defn-transf} holds for all irreducible generic unitary representations $\pi_S$ of $\GL_n(F_S)$ and additionally
$$
\mathrm{tr}(\Pi_{S_0})(f_{S_0}) \neq 0.
$$
\end{conj}
Here we say that $\Phi_S$ is of positive type if $\mathrm{tr}(\pi_S)(\Phi_S) \geq 0$ for all irreducible generic unitary admissible representations $\pi_S$ of $\GL_n(F_S)$.

\begin{rem}
Understanding which $h_{S'}$ and $\Phi_S$ are transfers of each other seems  subtle.   One na\"ive guess is that those $\Phi_S$ that are supported on ``norms'' of elements of $\GL_n(E_{S_0})$ should be transfers.  However, there does not appear to be a good notion of the norm map from conjugacy classes in $\GL_n(E_{S_0})$ to conjugacy classes in $\GL_n(F_S)$ and this guess makes little sense without a notion of norm.   One also hopes to be able to make use of the theory of cyclic base change, but in that setting one is interested in twisted traces, not traces.
For these reasons, the author is somewhat hesitant in stating Conjecture \ref{conj-transf}.  The author wouldn't be surprised if the conjectured equality holds only up some transfer factor.
\end{rem}

In one case the existence of transfers  is clear:

\begin{lem} \label{lem-unr-transf}  Assume that $S$ is a set of finite places and that $E/F$ is unramified at all places in $S$.  Let
$$
f_{S_0} \in C_c^{\infty}(\GL_n(E_{S_0})//\GL_n(\OO_{E_{S_0}})).
$$
Then $b_{E/F'}(f_{S_0})$ and $b_{E/F}(f_{S_0})$ are transfers of each other.
\end{lem}

\begin{proof}
If $\pi_S$ is unramified then the identity in the definition of transfers follows from the discussion in \S \ref{ssec-bcformal}.  If for some $v \in S$ the representation $\pi_v$ is ramified then $\pi_{vE}$ and $\pi_{vF'}$ are both ramified since the extension $E/F$ is unramified and the local Langlands correspondence sends (un)ramified representations to (un)ramified representations.  Thus in this case for $f_{S_0}$ as in the lemma one has
$$
\prod_{w \in S_0}\mathrm{tr}(\pi_{vEw})(f_w)=\prod_{v' \in S'} \mathrm{tr}(\pi_{vF'v'})(h_{v'})=\prod_{v \in S}\mathrm{tr}(\pi_{v})(\Phi_v)=0.
$$

\end{proof}

Another case where transfers exist is the following:

\begin{lem} \label{lem-archi-transf}
Let $v$ be a complex place of $F$ and let $S=\{v\}$.  Moreover let $f_{S_0}=\otimes_{w|v}f_w \in C_c^{\infty}(\GL_n(E_{S_0}))$. Then there exist functions $h_{S'} \in C_c^{\infty}(\GL_n(F'_{S'}))$ and $\Phi_v \in C_c^{\infty}(\GL_n(F_v))$ that are transfers of each other such that the identity of Definition \ref{defn-transf} holds.
\end{lem}

\begin{proof}
We use descent to prove the lemma.  Let $B \leq \GL_n$ be the Borel subgroup of lower triangular matrices.  We let $B=MN$ be its Levi decomposition; thus $M \leq \GL_n$ is the maximal torus of diagonal matrices.
 Let $\pi_v$ be an irreducible unitary generic representation of $\GL_n(F_v)$.
It is necessarily isomorphic to the unitary induction $\mathrm{Ind}_{B(F_v)}^{\GL_n(F_v)}(\chi)$ of a quasi-character $\chi:M(F_v) \to \CC^{\times}$ (where $\chi$ is extended to a representation of $MN(F_v)$ by letting $\chi$ act trivially on $N(F_v)$).  Here  we have used our assumption that $\pi_v$ is generic to conclude that the corresponding induced representation is irreducible (see the comments following \cite[Lemma 2.5]{JacquetRS} or \cite{Vogan}).  Let
$\Phi_v^{(B(F_v))} \in C_c^{\infty}(M(F_v))$ denote the constant term of $\Phi_v$ along $B(F_v)$.  Then one has
$$
\mathrm{tr}(\pi_v)(\Phi_v)=\mathrm{tr}(\chi)(\Phi_v^{(B(F_v))})
$$
(see \cite[(10.23)]{Knapp}).

Now let
\begin{align*}
N_{E/F'}:M(E \otimes_FF_v) &\lto M(F' \otimes_FF_v)\\
\N_{E/F} M(E \otimes_F F_v) &\lto M(F_v)
\end{align*}
be the norm maps.  Let $f_{S_0}=\otimes_vf_v \in C_c^{\infty}(\GL_n(E_{S_0}))$.
From the comments above, we see that if $h_{S'} =\otimes_{v'|v}h_v\in C_c^{\infty}(\GL_n(F'_{S'}))$ and $\Phi_v \in C_c^{\infty}(\GL_n(F_v))$ are any functions such that
\begin{align*}
\N_{E/F'*}(f^{(B(E_{S_0}))})&=h_{S'}^{(B(F'_{S'}))}\\
\N_{E/F*}(f^{(B(E_{S_0}))})&=\Phi_v^{(B(F_v))}
\end{align*}
then $h_{S'}$ and $\Phi_v$ are transfers of each other and the identity of Definition \ref{defn-transf} holds.

We note that
$$
\N_{E/F'*}(f^{(B(E_{S_0}))})_{v'} \in C_c^{\infty}(M(F'_{v'}))^{W(M,\GL_n)}
$$
for each $v'|v$ and
$$
\N_{E/F*}(f^{(B(E_{S_0}))})_v \in C_c^{\infty}(M(F_v))^{W(M,\GL_n)}.
$$
Thus to complete the proof it suffices to observe that the map
\begin{align*}
C_c^{\infty}(\GL_n(F_v)) &\lto C_c^{\infty}(M(F_v))^{W(M,\GL_n)}\\
\Phi_v &\longmapsto \Phi_v^{(B(F_v))} \nonumber
\end{align*}
is surjective by definition of the constant term \cite[(10.22)]{Knapp} and the Iwasawa decomposition of $\GL_n(F_v)$.
\end{proof}

Finally we consider transfers of Euler-Poincar\'e functions.  Let $v$ be a finite place of $F$ and let
$$
f_{EP} \in C_c^{\infty}(\mathrm{SL}_n(F_v)) \quad \textrm{ and } \quad f'_{EP} \in C_c^{\infty}(\mathrm{SL}_n(F' \otimes_FF_v))
$$
be the Euler-Poincar\'e functions on $\mathrm{SL}_n(F_v)$ and $\mathrm{SL}_n(F' \otimes F_v)$, respectively, defined with respect to the Haar measures on $\SL_n(F_v)$ and $\mathrm{SL}_n(F' \otimes F_v)$ giving $\mathrm{SL}_n(\OO_{F_v})$ and $\mathrm{SL}_n(\OO_{F'} \otimes_{\OO_{F}} \OO_{F_v})$ volume one, respectively \cite[\S 2]{KottTama}.  Moreover fix functions $\nu \in C_c^{\infty}(\GL_n(F_v)/\mathrm{SL}_n(F_v))$ and $\nu' \in C_c^{\infty}(\GL_n(F' \otimes_FF_v)/\mathrm{SL}_n(F' \otimes F_v))$.  We thereby obtain functions
$$
\nu f_{EP} \in C_c^{\infty}(\GL_n(F_v)) \quad \textrm{ and } \quad \nu'f'_{EP} \in C_c^{\infty}(\GL_n(F' \otimes_F F_v))
$$
We refer to any function in $C_c^{\infty}(\GL_n(F_v))$ (resp.~$C_c^{\infty}(\GL_n(F' \otimes_F F_v))$) of the form $\nu f_{EP}$ (resp.~$\nu'f'_{EP}$) as a \textbf{truncated Euler-Poincar\'e function}.

For the purposes of stating a lemma on transfers, fix measures on $F_v^{\times}$, $(F' \otimes_F F_{v})^{\times}$, and $(E \otimes_F F_v)^{\times}$ giving their maximal compact open subgroups (e.g. $\OO_{F_v}^{\times}$) volume one.  These measures induce Haar measures on $F_v^{\times} \cong \GL_n(F_v)/\mathrm{SL}_n(F_v)$, etc.
Assume for the purposes of the following lemma that $\nu$ and $\nu'$ are chosen so that there is a function $\nu_0 \in C_c^{\infty}(\GL_n(E \otimes_F F_v)/\mathrm{SL}_n(E \otimes_F F_v))$ such that
$$
\N_{E\otimes_{F}F_v/F' \otimes_F F_v*}\nu_0=\nu' \quad
\textrm{ and } \quad \N_{E\otimes_{F}F_v/F_v*}\nu_0=\nu.
$$
\begin{lem} \label{lem-EP}
Let $v$ be a finite place of $F$ that is unramified in $E/F$ and let $S=\{v\}$.  The Euler Poincar\'e functions
$$
(-1)^{r_1}\nu'f'_{EP} \quad  \textrm{ and } \quad (-1)^{r_2} \nu f_{EP}
$$
are transfers of each other for some integers $r_1,r_2$ that depend on $v$ and $E/F$.
\end{lem}

\begin{proof}
Recall that if $\mathrm{St}_v$ (resp.~$\mathrm{St}_{v'}$, $\mathrm{St}_{w}$) denotes the Steinberg representation of $\GL_n(F_v)$ (resp.~$\GL_n(F'_{v'})$, $\GL_n(E_w)$), then $\mathrm{St}_{vF'v'}=\mathrm{St}_{v'}$ and $\mathrm{St}_{vEw}=\mathrm{St}_w$ for all places $v'|v$ of $F'$ and $w|v$ of $E$ \cite[Chapter 1, Lemma 6.12]{AC}.

Let $f_{0EP} \in C_c^{\infty}(\mathrm{SL}_n(E \otimes_FF_v))$ be the Euler-Poincare function constructed with respect to the Haar measure giving measure one to $\mathrm{SL}_n(\OO_E \otimes_{\OO_F} \OO_{F_v})$.
We let $f_{S}=\nu_0f_{0EP}$ in the definition of transfer.
The statement of the lemma follows from the comments in the first paragraph of this proof and \cite[Theorem 2'(b)]{KottTama}.  This theorem of Casselman states in particular that the only generic unitary irreducible admissible representations of a semisimple $p$-adic group that have nonzero trace when applied to an Euler-Poincar\'e function are the Steinberg representations and gives a formula for the trace in this case.  We note that in the case at hand one has $q(\SL_{nF_v})=n-1$ in the notation of loc.~cit.
\end{proof}

\section{Limiting forms of the cuspidal spectrum}
 \label{sec-limit-cusp}

In the statement of our main theorems we considered limits built out of sums of the trace of a test function acting on the cuspidal spectrum.  In this section we prove that these limits converge absolutely.  Let $E/F$ be a Galois extension of number fields, let $E \geq F' \geq F$ be a solvable subfield, and let $\tau \in \Gal(E/F)$.  Assume that $S'$ is a finite set of places of $F'$ including the infinite places and that $S_0$ is the set of places of $E$ above them.  Let $\phi \in C_c^{\infty}(0,\infty)$ and let
$$
\widetilde{\phi}(s):=\int_{0}^{\infty}x^s\phi(x)\frac{dx}{x}
$$
be its Mellin transform.  The first result of this section is the following proposition:

\begin{prop} \label{prop-testsums}
Let $\pi'$ be a cuspidal unitary automorphic representation of $A_{\GL_{nF'}} \backslash \GL_n(\A_{F'})$ and assume that
$\pi'$ admits a weak base change $\Pi:=\pi_E'$ to $\GL_n(\A_E)$ that satisfies $\Pi_w=(\pi_{v'}')_E$ for all $w|v'$ such that $v' \not \in S'$.  Fix $\varepsilon>0$.
One has
\begin{align*}
&\mathrm{tr}(\pi'^{S'})(b_{E/F'}\Sigma^{S_0}_{\phi}(X))=\mathrm{Res}_{s=1}\left(\widetilde{\phi}(s)X^sL(s,\Pi \times \Pi^{\tau\vee S_0})\right)+O_{E,F',n,S_0,\varepsilon}(C_{\Pi \times \Pi^{\tau \vee}}(0)^{\tfrac{1}{2}+\varepsilon})
\end{align*}
for all sufficiently large $X \geq 1$.  Here ``sufficiently large'' depends only on $E,F',n,S_0,\varepsilon$.
\end{prop}
Here the complex number $C_{\Pi \times \Pi^{\tau \vee}}(s)$  is the analytic conductor of $L(s,\Pi \times \Pi^{\tau \vee})$ normalized as in \S \ref{ssec-analytic-cond} below and
the $L$-function is the Rankin-Selberg $L$-function (see, e.g. \cite{Cog1}).  The point of Proposition \ref{prop-testsums} is the fact that
$\lim_{X \to \infty} X^{-1} \mathrm{tr}(\pi^{S'})(\Sigma_{\phi}^{S_0}(X)) \neq 0$ if and only if
$\mathrm{Hom}_{I}(\Pi, \Pi^{\tau}) \neq 0$ (see \S \ref{ssec-analytic-cond}).  This is the keystone of the approach to nonsolvable base change and descent exposed in this paper.  We will prove Proposition \ref{prop-testsums} in the following subsection.

Let
\begin{align} \label{cusp-subspace}
L^2_0(\GL_n(F')A_{\GL_{nF'}} \backslash \GL_n(\A_{F'})) \leq L^2(\GL_n(F')A_{\GL_{nF'}} \backslash \GL_n(\A_{F'}))
\end{align}
be the cuspidal subspace.  For $h\in C_c^{\infty}(\GL_n(\A_{F'})$ unramified outside of $S'$, let
$$
h^1(g):=\int_{A_{\GL_{nF'}}}h(ag)da \in C_c^{\infty}(A_{\GL_{nF'}} \backslash \GL_n(\A_{F'}))
$$
and let $R(h^1)$ be the corresponding operator on $L^2(\GL_n(F') A_{\GL_{nF'}} \backslash \GL_n(\A_{F'}))$.  Finally let $R_0(h^1)$ be its restriction to $L^2_0(\GL_n(F') A_{\GL_{nF'}} \backslash \GL_n(\A_{F'}))$.  This restriction is well-known to be of trace class \cite[Theorem 9.1]{Donnely}.

The following corollary implies that the limits and sums in \eqref{11}, \eqref{A21}, \eqref{B21}, and \eqref{31} of theorems \ref{main-thm-1}, \ref{main-thm-2} and \ref{main-thm-3} can be switched:

\begin{cor} \label{cor-aut-trace} One has
\begin{align*}
\mathrm{tr}(R_0(h^1b_{E/F'}(\Sigma^{S'}_{\phi}(X))))&+o_{E,F',n,S_0}(X)\\&=\sum_{\pi'} \mathrm{tr}(\pi')(h^1)
\mathrm{Res}_{s=1}\left( \widetilde{\phi}(s)X^sL(s,(\pi'_E \times \pi'^{\tau \vee}_E)^{S_0})\right)
\end{align*}
where the sum is over any subset of the set of equivalence classes of cuspidal automorphic representations $\pi'$ of $A_{\GL_{nF'}} \backslash \GL_n(\A_{F'})$.  The sum on the right converges absolutely.
\end{cor}

\subsection{Rankin-Selberg $L$-functions and their analytic conductors} \label{ssec-analytic-cond}
For our later use, in this subsection we will consider a setting slightly more general than that relevant for the proof of Proposition \ref{prop-testsums}.
Let $n_1,n_2$ be positive integers and let $\Pi_1,\Pi_2$ be isobaric automorphic representations of $A_{\GL_{n_1E}} \backslash \GL_{n_1}(\A_E)$, $A_{\GL_{n_2E}} \backslash \GL_{n_2}(\A_E)$, respectively\footnote{For generalities on isobaric automorphic representations see \cite{LanglEinM} and \cite{JSII}.}.
We always assume that
$$
\Pi_i\cong \Pi_{i1} \boxplus \cdots \boxplus \Pi_{im_i}
$$
for $i \in\{1,2\}$ where the $\Pi_{ij}$ are cuspidal automorphic representations of $A_{\GL_{d_jE}} \backslash \GL_{d_j}(\A_E)$ for some set of integers $d_j$ such that $\sum_{j=1}^{m_i}d_j=n_i$.  We let $L(s,\Pi_1 \times \Pi_2)$ be the Rankin-Selberg $L$-function \cite{Cog1}; it satisfies
$$
L(s,\Pi_1 \times \Pi_2)=\prod_{r=1}^{m_1}\prod_{s=1}^{m_2} L(s,\Pi_{1r} \times \Pi_{2s}).
$$
and is holomorphic in the complex plane apart from possible poles at $s \in \{0,1\}$ \cite[Theorem 4.2]{Cog1}.
One sets
\begin{align} \label{hom-I}
\mathrm{Hom}_I(\Pi_1,\Pi_2)=\oplus_{j_1,j_2:\Pi_{1j_1} \cong \Pi_{2j_2}^{\vee}} \CC.
\end{align}
Lest the notation mislead the reader, we note that $\Pi_1$ and $\Pi_2$ are irreducible \cite[\S 2]{LanglEinM}, so the space of ``honest'' morphisms of automorphic representations between them is at most one-dimensional.  A fundamental result due to Jacquet and Shalika is that
\begin{align} \label{ord-pole}
-\mathrm{ord}_{s=1}L(s,\Pi_1 \times \Pi_2)=\dim(\mathrm{Hom}_I(\Pi_1,\Pi_2))
\end{align}
\cite[Theorem 4.2]{Cog1}.

There is a set of complex numbers $\{\mu_{\Pi_{1iv} \times \Pi_{2jv}}\}_{\substack{1 \leq i \leq n_2\\ 1 \leq j \leq n_2}} \subset \CC$, an integer $N_{\Pi_1 \times \Pi_2}$,
and a complex number $\epsilon_{\Pi_1 \times \Pi_2}$ such that if we set
$$
\Lambda(s,\Pi_1 \times \Pi_2):=L(s,(\Pi_1 \times \Pi_2)^{\infty})\prod_{w|\infty} \prod_{i=1}^{n_1}\prod_{j=1}^{n_2}\Gamma_{w}(s+\mu_{\Pi_{1wi} \times \Pi_{2wj}})
$$
then
\begin{align} \label{FE}
\Lambda(s,\Pi_1 \times \Pi_2)=\epsilon_{\Pi_1 \times \Pi_2} N_{\Pi_1 \times \Pi_2}^{\tfrac{1}{2}(1-s)}\Lambda(1-s,\Pi_1^{\vee} \times \Pi_2^{\vee}).
\end{align}
Here
$$
\Gamma_{v}(s):=\begin{cases}\pi^{-s/2}\Gamma(s/2) \textrm{ if }w \textrm{ is real}\\2(2\pi)^{-s}\Gamma(s) \textrm{ if }w \textrm{ is complex}. \end{cases}
$$
For a proof of this statement, combine \cite[Proposition 3.5]{Cog1}, the appendix of \cite{JacquetRS} and \cite[Theorem 4.1]{Cog1} (\cite[\S 3]{Tate} is also useful).  In these references one can also find the fact that, at least after reindexing, one has
\begin{align} \label{conj}
\bar{\mu_{\Pi_{1wi}^{\vee} \times \Pi_{2wj}^{\vee}}}=\mu_{\Pi_{1wi} \times \Pi_{2wj}}.
\end{align}

One sets
$$
\gamma(\Pi_1 \times \Pi_2,s):=\prod_{w|\infty}\prod_{i=1}^{n_1}\prod_{j=1}^{n_2} \Gamma_{w}(s+\mu_{\Pi_{1wi} \times \Pi_{2wj}})
$$
Following Iwaniec-Sarnak \cite{IS} (with a slight modification), for $s \in \CC$ the \textbf{analytic conductor} is defined to be
\begin{align}
C_{\Pi_1 \times \Pi_2}(s):=N_{\Pi_1 \times \Pi_2} \prod_{w|\infty}\prod_{i=1}^{n_1}\prod_{j=1}^{n_2} \left|\frac{1+\mu_{\Pi_{1wi} \times \Pi_{2wj}}+s}{2\pi}\right|_w.
\end{align}
We recall that $N_{\Pi_1 \times \Pi_2}=\prod_{w \nmid \infty}N_{\Pi_{1w} \times \Pi_{2w}}$ for some integers numbers $N_{\Pi_{1w} \times \Pi_{2w}}$ \cite[Proposition 3.5]{Cog1}.  If $S_0$ is a set of places of $E$  set
\begin{align*}
C_{\Pi_{1S_0} \times \Pi_{2S_0}}(s):&=\left(\prod_{\textrm{infinite }w \in S_0} \prod_{i=1}^{n_1}\prod_{j=1}^{n_2}\left|\frac{1+\mu_{\Pi_{1iw} \times \Pi_{2jw}}+s}{2\pi}\right|_w\right)
\prod_{\textrm{ finite }w \in S_0} N_{\Pi_{1w} \times \Pi_{2w}}.
\end{align*}

For the purpose of stating a proposition, let $\lambda(m) \in \CC$ be the unique complex numbers such that
$$
L(s,(\Pi_1 \times \Pi_2)^{S_0})=\sum_{m=1}^{\infty}\frac{\lambda(m)}{m^s}
$$
for $\mathrm{Re}(s) \gg 0$.
With all this notation set, we have the following proposition:

\begin{prop} \label{Perron-prop}  Assume that $\Pi_1 \times \Pi_2$ is automorphic of the form
$$
\Pi_1 \times \Pi_2 \cong \Pi_1 \boxplus \cdots \boxplus \Pi_m
$$
where the $\Pi_i$ are cuspidal automorphic representations of $A_{\GL_{n_iE}} \backslash \GL_{n_i}(\A_E)$.  Fix $\varepsilon>0$.
For $X \in \RR_{>0}$ sufficiently large one has
$$
\sum_{m=1}^{\infty}\lambda(m)\phi(m/X)=\mathrm{Res}_{s=1}\left( \widetilde{\phi}(s)X^sL(s,\Pi_1 \times \Pi_2^{S_0})\right) +O_{E,n,S_0}(C_{\Pi_1 \times \Pi_2}(0)^{\tfrac{1}{2}+\varepsilon}).
$$
Here ``sufficiently large'' depends only on $E,n,S_0$.
\end{prop}

Before beginning the proof of Proposition \ref{Perron-prop} we record one consequence of known bounds towards the Ramanujan-Petersson conjecture.
Let $r(E,n) \in \RR_{\geq 0}$ be a nonnegative real number such that for all finite places $w$ of $E$ one has
\begin{align} \label{rp-bound}
|L(s,\Pi_{1w} \times \Pi_{2w})| \leq (1-q_w^{-\mathrm{Re}(s)+r(E,n)})^{-n}
\end{align}
for $\mathrm{Re}(s)>r(E,n)$ and
\begin{align} \label{rbound}
|L(s,\Pi_{1w} \times \Pi_{2w})^{-1}| \leq (1+q_w^{-\mathrm{Re}(s)+r(E,n)})^n.
\end{align}
for all $s$.  This real number exists (and can be taken to be independent of the choice of cuspidal automorphic representations $\Pi_1$ and $\Pi_2$ of $A_{\GL_{nE}} \backslash \GL_n(\A_E)$) by \cite[Theorem 2]{LRS} in the unramified case and \cite[Proposition 3.3]{MS} in the general case.  In particular we may take
\begin{align} \label{LRS-bound}
r(E,n) < \frac{1}{2}-\frac{1}{n^2+1}.
\end{align}
If the Ramanujan-Petersson conjecture were known then we could take $r(E,n)=0$.

\begin{proof}[Proof of Proposition \ref{Perron-prop}]
The proof is a standard application of the inverse Mellin transform entirely analogous to the proof of \cite[Theorem 3.2]{Booker}.  We only make a few comments on how to adapt the proof of \cite[Theorem 3.2]{Booker}.  First, the assumption of the Ramanujan conjecture in \cite[Theorem 3.2]{Booker} can be replaced by the known bounds toward it that are recorded in \eqref{LRS-bound} above.  Second, the bounds on the gamma factors in terms of the analytic conductor are proven in detail in  \cite{Moreno}.  Finally, we recall that if $L(s,\Pi_1 \times \Pi_2)$ has a pole in the half-plane $\mathrm{Re}(s)>0$ then it is located at $s=1$ and is of order equal to $-\dim \mathrm{Hom}_{I}(\Pi_1,\Pi_2)$ by a result of Jacquet, Piatetskii-Shapiro and Shalika \cite[Theorem 4.2]{Cog1}; this accounts for the main term in the expression above.
\end{proof}

We now prepare for the proof of Proposition \ref{prop-testsums}.
If $w$ is a finite place of $E$ and $\Pi_w$ is an unramified representation of $\GL_n(E_w)$ we denote by $A(\Pi_w)$ the Langlands class of $\Pi_w$; it is a semisimple conjugacy class in $({}^L\GL_{nE})^{\circ}=\GL_n(\CC)$, the neutral component of ${}^L\GL_{nE}$.
For $\mathrm{Re}(s)>1$ we have that
\begin{align} \label{RS-descr}
L(s,(\Pi_1 \times \Pi_2)^{S_0}) =
\prod_{w \not \in S_0} \sum_{n \geq 1}\mathrm{tr}(\mathrm{Sym}^n(A(\Pi_{1w}) \otimes A(\Pi_{2w})))q_w^{-ns}
\end{align}
and the sum on the right converges absolutely \cite[Theorem 5.3 and proof of Proposition 2.3]{JS}.

\begin{proof}[Proof of Proposition \ref{prop-testsums}]
Let $v'$ be a place of $F'$ where $\pi'$ is unramified and let $a_1,\cdots,a_n$ be the Satake parameters of $\pi'_v$ (i.e. the eigenvalues of $A(\pi'_{v'})$).
We recall that
$$
\mathrm{tr}(\pi'_{v'})(h)=\mathcal{S}(h)(a_1,\dots,a_n)
$$
for all $h \in C_c^{\infty}(\GL_n(F'_{v'})//\GL_n(\OO_{F_{v'}'}))$ \cite[Theorem 7.5.6]{Laumon}.  This together with Proposition \ref{Perron-prop}, \eqref{RS-descr}, and the description of unramified base change recalled in \S \ref{ssec-bcformal}  implies the proposition.
\end{proof}

\subsection{Proof of Corollary \ref{cor-aut-trace}}
\label{ssec-cor-trace}

In this subsection our ultimate goal is to prove Corollary \ref{cor-aut-trace}.  In order to prove this corollary we first establish the following two lemmas:

\begin{lem} \label{lem-cond} Let $v'$ be an infinite place of $F'$, let $h_{v'} \in C_c^{\infty}(\GL_n(F_{v'}'))$ and let $N \in \ZZ$.  If $A$ is a countable set of inequivalent irreducible generic admissible representations of $\GL_n(F_{v'}')$, then for $\pi'_{v'} \in A$
$$
\mathrm{tr}(\pi'_{v'})(h_{v'})C_{\pi'_{v'}}(0)^N \to 0 \textrm{ as } C_{\pi'_{v'}}(0) \to \infty.
$$
\end{lem}

\begin{lem} \label{lem-Casimir} Fix a positive integer $n$.  There is an integer $N>0$ depending on $n$ and a polynomial $P$ of degree $N$ in $n$ variables such that the Casimir eigenvalue of an irreducible generic admissible representation $\pi_v$ of $\GL_n(F_v)$ is bounded by $|P(|\mu_{1\pi_v}|,\dots,|\mu_{n\pi_v}|)|$.
\end{lem}

Thus the trace and Casimir eigenvalue of the $\pi_v$ are controlled by the analytic conductor.  This is certainly well-known, but the author was unable to locate these results in the literature.  Moreover, the proof of Lemma \ref{lem-cond} is more interesting than one would expect.

We begin by recalling some notions that will allow us to use decent.
Let $v$ be an archimedian place of $F$; we fix an embedding $\RR \hookrightarrow F_v$ (which is an isomorphism if $v$ is real).  Let $\mathfrak{h} \leq \mathrm{R}_{F_v/\RR}\mathfrak{gl}_n$ be the Cartan subalgebra of diagonal matrices.  For a Lie algebra $\mathfrak{g}$ over $\RR$, write $\mathfrak{g}_{\CC}:=\mathfrak{g} \otimes_{\RR} \CC$.
 Without loss of generality we assume that the set of positive roots of $\mathfrak{h}_{\CC}$ inside $\mathrm{R}_{F_v/\RR}\mathfrak{gl}_{n\CC}$ is defined using the Borel subgroup $B \leq \GL_n$ of \textbf{lower} triangular matrices (this is to be consistent with \cite{JacquetRS}).  Thus standard parabolic subgroups are parabolic subgroups containing $B$.  If $Q=MN$ is the Levi decomposition of a standard parabolic subgroup then
\begin{align} \label{Levi-decomp}
M =\prod_jM_j \cong \prod_{j} \mathrm{R}_{F_v/\RR}\GL_{n_j}
\end{align}
 where $M_j \cong \mathrm{R}_{F_v/\RR}\GL_{n_j}$ and $\sum_jn_j=n$.  If $Q$ is cuspidal then $n_j \in \{1,2\}$ if $v$ is real and all $n_j=1$ if $v$ is complex.  We let $\mm:=\mathrm{Lie}(M)$ and $\mm_j:=\mathrm{Lie}(M_j)$.
 Moreover we let $\mathfrak{h}_j:=\mathfrak{h} \cap \mm_j$; thus $\mathfrak{h}_j$ is isomorphic to the Cartan subalgebra of diagonal matrices in $\mathrm{R}_{F_v/\RR}\GL_{n_j}$.

Let $\pi_v$ be an irreducible admissible representation of $\GL_n(F_v)$.
Thus there is a cuspidal standard parabolic subgroup $Q \leq \mathrm{Res}_{F_v/\RR}\GL_n$ with Levi decomposition $Q=MN$ and an irreducible admissible representation $\pi_M$ of $M(F_v)$
 such that
$$
\pi_v \cong J(\pi_M)
$$
where $J(\pi_M)$ is the Langlands quotient of the induced representation $\mathrm{Ind}_{Q(F_v)}^{\GL_n(F_v)}(\pi_M)$ \cite[Theorem 14.92]{Knapp}. Moreover $\pi_M$ can be taken to be a twist of a discrete series or limit of discrete series.  Here we are viewing $\pi_M$ as a representation of $Q(F_v)$ by letting it act trivially on $N(F_v)$.  If $\pi_v$ is generic, then $\mathrm{Ind}_{Q(F_v)}^{\GL_n(F_v)}(\pi_M)$ is irreducible and hence
$$
\pi_v \cong J(\pi_M) \cong \mathrm{Ind}_{Q(F_v)}^{\GL_n(F_v)}(\pi_M)
$$
(see the comments after \cite[Lemma 2.5]{JacquetRS} or \cite{Vogan}).
We decompose
$$
\pi_M \cong \otimes_j \pi_j
$$
where each $\pi_j$ is an admissible irreducible representation of $M_j(F_v)$.
We note that, essentially by definition,
\begin{align} \label{L-prod}
L(s,\pi_v)=\prod_jL(s,\pi_j)=\prod_j\prod_{i=1}^{n_j}\Gamma_v(s+\mu_{i\pi_{j}})
\end{align}
(compare \cite[Appendix]{JacquetRS}).

We are now in a position to prove Lemma \ref{lem-cond}:

\begin{proof}[Proof of Lemma \ref{lem-cond}]
Let $f \in C_c^{\infty}(\GL_n(F_v))$ and let $f^{(Q)} \in C_c^{\infty}(M(F_v))$ be the constant term of $f$ along $Q$ (see \cite[(10.22)]{Knapp} for notation).  Using the natural isomorphism $C_c^{\infty}(M(F_v))=\prod_{j}C_c^{\infty}(M_j(F_v))$ we decompose $f^{(Q)}=\prod_jf_j$.  One then has
\begin{align} \label{const}
\mathrm{tr}(\pi_v)(f)=\mathrm{tr}(\pi_M)(f^{(Q)})=\prod_j\mathrm{tr}(\pi_j)(f_j)
\end{align}
(see \cite[(10.23)]{Knapp}).
  Combining \eqref{const} and \eqref{L-prod}, we see that the lemma will follow if we establish it in the special cases $n \in \{1,2\}$ for $v$ real and $n=1$ for $v$ complex.  Moreover when $n=2$ we can assume that $\pi_v$ is a twist by a quasi-character of a discrete series or limit of discrete series representation.  We henceforth place ourselves in this situation.

Assume for the moment that $n=1$ and $v$ is real.  Then $\pi_v$ is a quasi-character of $\RR^{\times}$ and hence it is of the form
$$
\pi_v(t)=|t|_v^{u}\mathrm{sgn}^{k}(t)
$$
for $t \in \RR^{\times}$ and some $u \in \CC$ and $k \in \{0,1\}$.  In this case we have $\mu_{1\pi_v}=u+k$ by \cite[Appendix]{JacquetRS}.  Similarly, if $n=1$ and $v$ is complex, then $\pi_v$ is a quasi-character of $\CC^{\times}$ and hence is of the form
$$
\pi_v(z):=z^{m}(z\bar{z})^{-\tfrac{m}{2}}|z|_v^{u}
$$
for $z \in \CC^{\times}$ and some $m \in \ZZ$ and $u \in \CC$.  In this case we have $\mu_{1\pi_v}=\tfrac{m}{2}+u$ by \cite[Appendix]{JacquetRS}.  In either case, as a function of $\mu_{1\pi_v}$  the trace $\mathrm{tr}(\pi_v)(f)$ is easily seen to be rapidly decreasing since the Fourier transform of a compactly supported smooth function on $\RR^{\times}$ or $\CC^{\times}$ is rapidly decreasing.  The lemma follows in these cases.

We are left with the case where $n=2$ and $v$ is real; thus $\pi_v$ is a twist of a discrete series or nondegenerate limit of discrete series representation by a quasi-character.  For $m \in \ZZ$ let
$$
\Omega_{m}(z):=z^{m}(z\bar{z})^{-\tfrac{m}{2}} \quad \textrm{ and } \quad \sigma_m:=\mathrm{Ind}_{\CC^{\times}}^{W_{F_v}}(\Omega_m).
$$
The $L$-parameter $\varphi(\pi_v):W_{F_v} \to {}^L\GL_{2F_v}$ attached to $\pi_v$ is of the form $\sigma_m \otimes \chi$ for some $m \in \ZZ_{\geq 0}$ and some one-dimensional representation $\chi:W_{F_v} \to \RR^{\times}$.
The discrete series (or limit of discrete series) representation $\pi(\sigma_m)$ will be denoted by $D_{m+1}$; it is in the discrete series if and only if $m>0$ (see \cite[Appendix]{JacquetRS}).  The representation $D_{m+1}$ is usually referred to as the discrete series of weight $m+1$ if $m>0$ and the limit of discrete series if $m=0$.
Recall that
any one-dimensional representation of $W_{F_v}$ can be regarded (canonically) as a character $\RR^{\times} \to \RR^{\times}$; this applies in particular to $\chi$.  We note that
$$
\sigma_m \otimes \mathrm{sgn} \cong \sigma_m
$$
since $\mathrm{sgn}$ can be identified with the nontrivial character of $W_{\RR}/W_{\CC}$ by class field theory.  Since $\pi$ is assumed to be unitary, we assume without loss of generality that $\chi=|\cdot|_v^{it}$ for some real number $t$.  With this in mind, the duplication formula implies that we may take $\mu_{1\pi_v}=\tfrac{m}{2}+it$ and $\mu_{2\pi_v}=\tfrac{m}{2}+1+it$ (compare \cite[Appendix]{JacquetRS}).

We compute the trace $\mathrm{tr}(\pi_v)(f)$ for $f \in C_c^{\infty}(\GL_2(F_v))=C_c^{\infty}(\GL_2(\RR))$.  First, define
\begin{align*}
f_t:\mathrm{SL}_2(\RR) &\lto \CC\\
g &\longmapsto \int_{\RR^{\times}}|z|_v^{it}\int_{\mathrm{SO}_2(\RR)}f(k^{-1}zgk)dzdk
\end{align*}
where we normalize the measure $dk$ so that $\mathrm{meas}_{dk}(\mathrm{SO}_2(\RR))=1$ and $dz$ is some choice of Haar measure.  Thus $f_t \in C_c^{\infty}(\SL_2(\RR))$.   One has (with appropriate choices of measures)
\begin{align*}
\mathrm{tr}(\pi_v)(f)=\int_{\mathrm{SL}_2(\RR)} \Theta_{m+1}(g)f_t(g)dg=:\Theta_{m+1}(f_t)
\end{align*}
where $dg$ is the Haar measure on $\mathrm{SL}_2(\RR)$ giving $\mathrm{SO}_2(\RR)$ measure one and $\Theta_{m+1}$ is the character of $D_{m+1}|_{\mathrm{SL}_2(\RR)}$.  By Fourier theory on $C_{\GL_{2}}(\RR) \cong \RR^{\times}$, one sees that to prove the lemma it suffices to prove that as $m \to \infty$
$$
|\Theta_{m+1}(f_t)| |m|^N \to 0
$$
for all $N \in \ZZ$.  For this it suffices to show that for all $f \in C_c^{\infty}(\mathrm{SL}_2(\RR))$ such that $f(kxk^{-1})=f(x)$ one has
\begin{align} \label{to-show-induct}
|\Theta_{m+1}(f)||m|^N \to 0
\end{align}
for all $N \in \ZZ$.  This is what we will show.

Let
$$
M(F_v)=\left\{ a_t:  a_t=\begin{pmatrix}e^t & 0 \\ 0 & e^{-t} \end{pmatrix},t \in \RR\right\}
$$
 and
$$
T(F_v)=\left\{k_{\theta}: k_{\theta}=\begin{pmatrix} \cos \theta & \sin \theta\\ -\sin \theta  & \cos \theta \end{pmatrix}, 0 < \theta \leq 2\pi\right\}
$$
By \cite[(11.37)]{Knapp}\footnote{Knapp denotes  $M$ by $T$ and $T$ by $B$.} and the discussion below it, for $f \in C_c^{\infty}(\mathrm{SL}_2(\RR))$ satisfying $f(kxk^{-1})=f(x)$ for $k \in \mathrm{SL}_2(\RR)$ we have
\begin{align} \label{first-formula}
m\Theta_{m+1}(f)=&-\frac{1}{2\pi i} \int_{0}^{2 \pi}(e^{im\theta}+e^{-im\theta})\frac{d}{d\theta}F^T_f(\theta)d\theta\\& \nonumber +\frac{1}{2}\int_{-\infty}^{\infty} e^{-m|t|}(\mathrm{sgn}(t)) \frac{d}{dt} F^M_f(a_t)dt\\ \nonumber
&+\frac{1}{2}(-1)^m\int_{-\infty}^{\infty}e^{-m|t|}(\mathrm{sgn}(t))\frac{d}{dt}F_f^M(-a_t)dt
\end{align}
for $m>0$.    The $m=0$ term is unimportant for our purposes as we are interested in the behavior as $m \to \infty$.  Here $dt$ and $d\theta$ are the usual Lesbesgue measures and
\begin{align*}
F^T_f(k_{\theta})&=(e^{i\theta}-e^{-i\theta})\int_{\GL_2(F_v)/T(F_v)}f(xk_{\theta}x^{-1})d \dot{x}\\
F^M_f(\pm a_t)&=\pm|e^t-e^{-t}|\int_{\GL_2(F_v)/M(F_v)} f(xa_tx^{-1}) d\dot{x}
\end{align*}
for suitably chosen Haar measures (that are independent of $\pi_v$) \cite[(10.9a-b)]{Knapp}. We note that the functions $F^M_f$ are smooth \cite[Proposition 11.8]{Knapp} and for integers $k \geq 0$ the odd-order derivative $\frac{d^{2k+1}}{dt^{2k+1}}F^T_f(\theta)$ is continuous (see the remarks after \cite[Proposition 11.9]{Knapp}).  Moreover $F^M_f(a_t)$ vanishes outside a bounded subset of $M(\RR)$ \cite[Proposition 11.7]{Knapp}.

We claim that for $m>0$ and $k \geq 1$ one has
{\allowdisplaybreaks
\begin{align} \label{claim-induct}
m^{2k+1}\Theta_{m+1}(f)=&-\frac{1}{2\pi i^{2k+1}} \int_{0}^{2 \pi}(e^{im \theta}+e^{-im \theta})\frac{d^{2k+1}}{d\theta^{2k+1}}F_f^T(\theta) d\theta\\&+\frac{1}{2}\int_{-\infty}^{\infty}e^{-m|t|}(\mathrm{sgn}(t))\frac{d^{2k+1}}{dt^{2k+1}}F^M_f(a_t)dt \nonumber \\
&+\frac{1}{2}(-1)^m\int_{-\infty}^{\infty}e^{-m|t|} (\mathrm{sgn}(t))\frac{d^{2k+1}}{dt^{2k+1}}F^M_f(-a_t)dt. \nonumber
\end{align}}
Assuming \eqref{claim-induct}, an application of the Riemann-Lesbesgue lemma implies \eqref{to-show-induct} which in turn implies the lemma.  Thus proving \eqref{claim-induct} will complete the proof of the lemma.

Proceeding by induction, assume \eqref{claim-induct} is true for $k-1>0$.  Applying integration by parts we obtain that \eqref{claim-induct} is equal to
{\allowdisplaybreaks\begin{align*}
&-(mi)^{-1}\left(-\frac{1}{2 \pi i^{2k-1}}\right)\int_{0}^{2 \pi}( e^{im\theta}-e^{-im\theta})\frac{d^{2k}}{d\theta^{2k}}F_f^T(\theta)d\theta\\
&-\frac{1}{2m}\int_{-\infty}^{\infty} e^{-m|t|}(\mathrm{sgn}(t)) \frac{d^{2k}}{dt^{2k}} F^M_f(a_t)dt\\
&-\frac{1}{2m}(-1)^m\int_{-\infty}^{\infty}e^{-m|t|}(\mathrm{sgn}(t))\frac{d^{2k}}{dt^{2k}}F_f^M(-a_t)dt\\
&+-\frac{1}{-2m}e^{-m|0^+|}\left( \frac{d^{2k-1}}{dt^{2k-1}}F_f^M(a_0^+) +(-1)^m\frac{d^{2k-1}}{dt^{2k-1}}F_f^M(-a_0^+)\right)\\&+-\frac{1}{-2m}e^{-m|0^-|}\left( \frac{d^{2k-1}}{dt^{2k-1}}F_f^M(a_0^-) +(-1)^m\frac{d^{2k-1}}{dt^{2k-1}}F_f^M(-a_0^-)\right)\\
&=-m^{-1}\left(-\frac{1}{2 \pi i^{2k}}\right)\int_{0}^{2 \pi}( e^{im\theta}-e^{-im\theta})\frac{d^{2k}}{d\theta^{2k}}F_f^T(\theta)d\theta\\
&-\frac{1}{2m}\int_{-\infty}^{\infty} e^{-m|t|}(\mathrm{sgn}(t)) \frac{d^{2k}}{dt^{2k}} F^M_f(a_t)dt\\
&-\frac{1}{2m}(-1)^m\int_{-\infty}^{\infty}e^{-m|t|}(\mathrm{sgn}(t))\frac{d^{2k}}{dt^{2k}}F_f^M(-a_t)dt
\\&+\frac{1}{m}\left( \frac{d^{2k-1}}{dt^{2k-1}}F_f^M(a_0)+(-1)^m\frac{d^{2k-1}}{dt^{2k-1}}F_f^M(-a_0)\right)
\end{align*}}

\noindent where the $\pm$ denote values as $t \to 0^{\pm}$ (this is purely for emphasis, as $F_f^M$ is smooth).  We note that the extra terms occur because of the singularity of the sign function $\mathrm{sgn}(t)$ at $t=0$. Since $F_f^M(\pm a_t)$ is even as a function of $t$, the last terms in the expression above vanish. Thus the quantity above is equal to
\begin{align*}
&-m^{-1}\left(-\frac{1}{2 \pi i^{2k}}\right)\int_{0}^{2 \pi}( e^{im\theta}-e^{-im\theta})\frac{d^{2k}}{d\theta^{2k}}F_f^T(\theta)d\theta\\
&-\frac{1}{2m}\int_{-\infty}^{\infty} e^{-m|t|}(\mathrm{sgn}(t)) \frac{d^{2k}}{dt^{2k}} F^M_f(a_t)dt\\
&-\frac{1}{2m}(-1)^m\int_{-\infty}^{\infty}e^{-m|t|}(\mathrm{sgn}(t))\frac{d^{2k}}{dt^{2k}}F_f^M(-a_t)dt.
\end{align*}
Keeping in mind that $\frac{d^{2k}}{d\theta^{2k}}F_f^T(\theta)$ has jump discontinuities at $0$ and $\pi$ (see the remark after \cite[Proposition 11.9]{Knapp}), we now apply integration by parts again to see that this expression is equal to
\begin{align*}
&m^{-2}\left(-\frac{1}{2 \pi i^{2k+1}}\right)\int_{0}^{2 \pi}( e^{im\theta}+e^{-im\theta})\frac{d^{2k+1}}{d\theta^{2k+1}}F_f^T(\theta)d\theta\\
&-m^{-2}\left(-\frac{1}{ 2\pi i^{2k+1}}\right)(2)\left( \frac{d^{2k}}{d\theta^{2k}}F_f^T(0^-)-(-1)^m\frac{d^{2k}}{d\theta^{2k}}F_f^T(\pi^+)+
(-1)^m\frac{d^{2k}}{d\theta^{2k}}F_f^T(\pi^-)-\frac{d^{2k}}{d\theta^{2k}}F_f^T(0^+)\right)\\
&+\frac{1}{2m^2}\int_{-\infty}^{\infty} e^{-m|t|}(\mathrm{sgn}(t)) \frac{d^{2k+1}}{dt^{2k+1}} F^M_f(a_t)dt\\
&+\frac{1}{2m^2}(-1)^m\int_{-\infty}^{\infty}e^{-m|t|}(\mathrm{sgn}(t))\frac{d^{2k+1}}{dt^{2k+1}}F_f^M(-a_t)dt\\
&+\frac{1}{m^2}\frac{d^{2k}}{dt^{2k}}F_f^T(a_0)+\frac{1}{m^2}(-1)^m\frac{d^{2k}}{dt^{2k}}F_f^T(-a_0).
\end{align*}
The second and last lines of the expression above cancel by the jump relations \cite[(11.45a), (11.45b)]{Knapp}.  Thus the above is equal to
\begin{align*}
&m^{-2}\left(-\frac{1}{2 \pi i^{2k+1}}\right)\int_{0}^{2 \pi}( e^{im\theta}+e^{-im\theta})\frac{d^{2k+1}}{d\theta^{2k+1}}F_f^T(\theta)d\theta\\
&+\frac{1}{2m^2}\int_{-\infty}^{\infty} e^{-m|t|}(\mathrm{sgn}(t)) \frac{d^{2k+1}}{dt^{2k+1}} F^M_f(a_t)dt\\
&+\frac{1}{2m^2}(-1)^m\int_{-\infty}^{\infty}e^{-m|t|}(\mathrm{sgn}(t))\frac{d^{2k+1}}{dt^{2k+1}}F_f^M(-a_t)dt
\end{align*}
which completes the induction, proving \eqref{claim-induct} and hence the lemma.
\end{proof}

\begin{rem} The jump relations which appear in this proof play a role in Langlands' adelization of the trace formula and his hope that it will be amenable to Poisson summation \cite{FLN} \cite{LSing}.
\end{rem}

For the proof of Lemma \ref{lem-Casimir}, it is convenient to summarize some of the information obtained in the previous lemma in the following table:
\begin{center}
\begin{tabular}{ l | c |c }
 $\pi_{v}$& $v$ &   $(\mu_{i\pi_{v}})$\\
\hline
$t \mapsto \mathrm{sgn}(t)^k|t|_v^{u}$ & real & $k+u$\\
$D_{m+1} \otimes |\cdot|_v^u$ & real &$(m/2+u,m/2+1+u)$\\
$z \mapsto z^{m}(z\bar{z})^{-\tfrac{m}{2}}|z|_v^{u}$ & complex & $\tfrac{m}{2}+u$
\end{tabular}
\end{center}

We now prove Lemma \ref{lem-Casimir}:

\begin{proof}[Proof of Lemma \ref{lem-Casimir}]
The Harish-Chandra isomorphism \cite[\S VIII.5]{Knapp} factors as
 $$
 \gamma:Z(\mathrm{R}_{F_v/\RR}\mathfrak{gl}_{n \CC}) \tilde{\lto} Z(\mm_{\CC}) \tilde{\lto} U(\mathfrak{h}_{\CC})
 $$
 where the second map is the Harish-Chandra isomorphism
 $$
 \gamma_M:Z(\mm_{\CC}) \tilde{\lto} U(\mathfrak{h}_{\CC}).
 $$
 We also have Harish-Chandra isomorphisms
$$
\gamma_j:Z(\mm_{j\CC}) \tilde{\lto} U(\mathfrak{h}_{j\CC}).
$$

The infinitisimal character of $\pi_v$ (resp.~$\pi_M$) is of the form $\Lambda(\pi_v) \circ \gamma$ (resp.~$\Lambda(\pi_{M}) \circ \gamma_M$) for some $\Lambda(\pi_v)$ (resp.~$\Lambda(\pi_M)$) in $ \mathfrak{h}_{\CC}^{\wedge}$ \cite[\S VIII.6]{Knapp}.  Similarly the infinitisimal character of $\pi_j$ is of the form $\Lambda(\pi_j) \circ \gamma_j$ for some $\Lambda(\pi_j) \in \mathfrak{h}_j^{\wedge}$.  Moreover
\begin{align} \label{factor-Lambda}
\Lambda(\pi_v)=\Lambda(\pi_M)=\sum_j \Lambda(\pi_j)
\end{align}
\cite[Proposition 8.22]{Knapp}.

If $C \in Z(\mathrm{R}_{F_v/\RR}\mathfrak{gl}_{n\CC})$ is the Casimir operator, the eigenvalue of $C$ acting on the space of $\pi_v$  is $\Lambda(\pi_v)( \gamma(C))$.  For each $j$ let $\{\Lambda_{j,\alpha}\} \subset \mathfrak{h}_{j\CC}^{\wedge}$ be a basis, and write
$$
\Lambda(\pi_v):=\sum_{j} \sum_{\alpha} a_{j,\alpha}(\pi_j) \Lambda_{j,\alpha}
$$
for some $a_{j,\alpha}(\pi_v) \in \CC$.
We note that $\gamma(C)$ does not depend on $\pi_v$.  Therefore in order to prove the lemma it suffices to exhibit a basis as above such that the $a_{j,\alpha}(\pi_j)$ are bounded by a polynomial in the $|\mu_{i\pi_v}|$ for $1 \leq i \leq n$.  In view of \eqref{L-prod} and \eqref{factor-Lambda} it follows that in order to prove the lemma it suffices to prove this statement in the special case where $n \in \{1,2\}$ for $v$ real and the special case $n=1$ for $v$ complex.  In the $n=1$ cases this comes down to unraveling definitions.  For $n=2$ we can assume that $\pi_v$ is a twist by a quasi-character of a discrete series or limit of discrete series representation.  In this case we refer to \cite[Chapter VIII, \S 16, Problems 1]{Knapp}.
\end{proof}

We now prove Corollary \ref{cor-aut-trace}:

\begin{proof}[Proof of Corollary \ref{cor-aut-trace}]
In view of Proposition \ref{prop-testsums} in order to prove the corollary it suffices to show that the contribution of the terms in Proposition \ref{prop-testsums} that depend on the automorphic representation does not grow too fast when we sum over all automorphic representations of $A_{\GL_{nF'}} \backslash \GL_n(\A_{F'})$ fixed by a given compact open subgroup of $\GL_n(\A_{F'}^{\infty})$.
More precisely, it suffices to show that
\begin{align} \label{obound}
\sum_{\pi'}\mathrm{tr}(\pi')(h)(C_{\pi'_E \times \pi'^{\tau\vee}_E}(0)^{\tfrac{1}{2}+\varepsilon})=o(X)
\end{align}
 where the sum is over any subset of the set of equivalence classes of cuspidal automorphic representations $\pi'$ of $A_{\GL_{nF'}} \backslash \GL_{n}(\A_{F'})$ and $\pi_E'$ is the base change of $\pi'$ to $\GL_n(\A_E)$.

The basic properties of cyclic base change (i.e. the relationship between the $L$-function of an admissible representation and its base change) together with the recollections on local $L$-factors collected in \S \ref{ssec-analytic-cond} above
imply that
\begin{align*}
C_{\pi'_{E\infty} \times \pi'^{\vee}_{E\infty}}(s) \leq C_{\pi'_{\infty}}(s)^N
\end{align*}
for sufficiently large $N \geq 0$ depending on $E/F'$ and $n$.
Using Lemma \ref{lem-cond} and the Weyl law for cusp forms \cite[Theorem 9.1]{Donnely}, we see that in order to prove \eqref{obound} it suffices to show that the Casimir eigenvalues (and hence the Laplacian eigenvalues) of a cuspidal automorphic representation $\pi'$ contributing to \eqref{obound} are a polynomial function of the absolute value of the parameters
$\mu_{i\pi'_{v'}}$ for archimedian $v'$.  This is the content of Lemma \ref{lem-Casimir}.
\end{proof}

\section{Restriction and descent of $L$-parameters}\label{sec-rest-desc}

The goal of this section is to prove some properties of $L$-parameters under restriction along an extension of number fields and then formulate the conjectures in automorphic representation theory to which these properties correspond.  Criteria for parameters to descend are given in \S \ref{ssec-desc}.  In \S \ref{ssec-primitive} we define a notion of $E$-primitive parameters and automorphic representations and then use it in \S \ref{ssec-rest-param} to give
an explicit description of the fibers and image of the restriction map (see Proposition \ref{prop-bij-EF'}).  In \S \ref{ssec-icosa-gp} a complement to Proposition \ref{prop-bij-EF'}, namely Proposition \ref{prop-A5-EF}, is given.  More specifically, Proposition \ref{prop-A5-EF} deals with the case of field extensions with Galois group isomorphic to $\widetilde{A}_5$.   Propositions \ref{prop-bij-EF'} and \ref{prop-A5-EF} are meant as motivation for conjectures \ref{conj-1} and \ref{conj-2} below respectively; these are the conjectures that appear in the statements of our first two main theorems.  Finally, in \S \ref{ssec-artin-conj} we prove lemmas on restriction of $L$-parameters along subfields of a $\widetilde{A}_5$-extension that motivate conjectures \ref{conj-32} and \ref{conj-33}, the conjectures assumed in Theorem \ref{main-thm-3}.

\subsection{Parameters and base change}
\label{ssec-param-bc}
In this subsection we recall the base change map (or restriction map) on $L$-parameters which conjecturally corresponds to base change of automorphic representations of $\GL_n$.
Let $E/F$ be an extension of number fields, let $W_E$ (resp.~$W_F$) denote the Weil group of $E$ (resp.~$F$) and let
$$
W_E'=W_E \times \mathrm{SU}(2)
$$ (resp.~$W_F':=W_F \times \mathrm{SU}(2)$) denote the Weil-Deligne group of $E$ (resp.~$F$), where $\mathrm{SU}(2)$ is the real compact Lie group of unitary matrices of determinant one\footnote{There are competing definitions of representations of the Weil-Deligne group that are all equivalent, see \cite[\S 2.1]{GR}.  To pass from $\mathrm{SL}_2(\CC)$ to $\mathrm{SU}(2)$ one uses the unitary trick.}.
We will be using the notion of an $L$-parameter
$$
\varphi:W_F' \lto{}^L\GL_{nF}=W_{F}' \times \GL_n(\CC)
$$
extensively\footnote{$L$-parameters are defined in \cite[\S 8]{Borel} where they are called ``admissible homomorphisms.''}.
Part of the definition of an $L$-parameter is the stipulation that the induced map
$$
W_{F}' \lto W_F'
$$
given by projection onto the first factor of ${}^L\GL_{nF}=W_{F}' \times \GL_n(\CC)$ is the identity.
Thus $\varphi$ is determined by the representation $W_{F}' \to\GL_n(\CC)$ defined by
projection onto the second factor of ${}^L\GL_{nF}$:
\begin{align} \label{proj-second}
\begin{CD}
 W_F' @>{\phi}>> {}^L\GL_{nF}@>>> {}^L\GL_{nF}^{\circ}=\GL_n(\CC).
\end{CD}
\end{align}
Thus one can safely think of $L$-parameters as representations $W_{F}' \to \GL_n(\CC)$ satisfying certain additional properties.
We say that an $L$-parameter $\phi:W_F' \to {}^L\GL_{n}$ is irreducible if the representation \eqref{proj-second} is irreducible.

For convenience, we denote by
\begin{align} \label{L-param}
\Phi_n(F):&=\{\textrm{Equivalence classes of }L\textrm{-parameters }\varphi:W_{F}' \to {}^L\GL_{nF}\}\\
\nonumber \Phi_n^0(F):&=\{\textrm{Equivalence classes of irreducible }L\textrm{-parameters }\varphi:W_{F}' \to {}^L\GL_{nF}\}.
\end{align}
If $E/F$ is Galois then there is a natural action of $\Gal(E/F)$ on the set of $L$-parameters from $W_E'$ given by
$$
\phi^{\sigma}(g)=\phi(\sigma^{-1}g\sigma).
$$
This induces an action of $\Gal(E/F)$ on $\Phi_n(E)$ which preserves $\Phi_n^0(E)$; we denote the invariants under this action by $\Phi_n(E)^{\Gal(E/F)}$ (resp.~$\Phi_n^0(E)^{\Gal(E/F)}$).

As noted above in \S \ref{ssec-bcformal}, we have a base change $L$-map
$$
b_{E/F}:{}^L\GL_{nF} \lto {}^L\mathrm{R}_{E/F}\GL_{nF}
$$
given by the diagonal embedding on connected components and the identity on the $W_F'$-factors.  For each $L$-parameter $\varphi:W_F' \to {}^L\GL_{nF}$ the composition $b_{E/F} \circ \varphi :W_F' \to {}^L\GL_{nF}$ is another $L$-parameter (compare \cite[\S 15.3]{Borel}).  One can view this construction in an equivalent manner as follows: an $L$-parameter $\phi:W_F' \to {}^L\mathrm{R}_{E/F}\GL_{nF}$ can be identified canonically with an $L$-parameter $\phi:W_E' \to {}^L\GL_{nE}$. From this viewpoint, the base change map simply associates to a parameter $\varphi:W_F' \to {}^L\GL_{nF}$ its restriction $b_{E/F}\circ \varphi:=\varphi|_{W_E'}$ (compare \cite[\S 15.3]{Borel}).
Thus $b_{E/F}$ induces a map
\begin{align} \label{L-restr}
b_{E/F}:\Phi_n^0(F) &\lto \Phi_n(E)\\
\varphi &\longmapsto \varphi|_{W_E'} \nonumber
\end{align}
which has image in $\Phi_n(E)^{\Gal(E/F)}$ if $E/F$ is Galois.

According to Langlands functoriality, there should be a corresponding transfer of $L$-packets of automorphic representations.
In fact, since $L$-packets are singletons in the case at hand, we should obtain an honest map from the set of equivalence classes of automorphic representations of $\GL_n(\A_F)$ to the set of equivalence classes of automorphic representations of $\mathrm{R}_{E/F}\GL_n(\A_F)=\GL_n(\A_E)$.
Thus we should expect
a map
\begin{align*}
b_{E/F}:\Pi_n(F) &\stackrel{?}{\lto} \Pi_n(E)\\
\pi &\stackrel{?}{\longmapsto} \pi_E
\end{align*}
which has image in $\Pi_n(E)^{\Gal(E/F)}$ if $E/F$ is Galois.  Moreover this map should share certain properties of \eqref{L-restr}.  Making this precise in general seems to require the introduction of the conjectural Langlands group.  Rather than take this route, we will simply prove properties of the restriction map on $L$-parameters (specifically propositions \ref{prop-bij-EF'} and \ref{prop-A5-EF} and lemmas \ref{lem-A5-EF} and \ref{lem-A5-EF3}) and then state the specific conjectures in automorphic representation theory (specifically conjectures \ref{conj-1}, \ref{conj-2}, \ref{conj-32} and \ref{conj-33}) that they suggest.

\subsection{Descent of parameters}
\label{ssec-desc}

Our goal in this subsection is to prove the following lemma:

\begin{lem} \label{lem-bc-param}  Let $E/F$ be a Galois extension of number fields.
If $\phi$ is irreducible and $\phi^{\sigma} \cong \phi$ for all $\sigma \in \Gal(E/F)$, then there is an $L$-parameter
$$
\varphi:W_F' \lto {}^L\GL_{nF}
$$
such that $b_{E/F} \circ \varphi \otimes \chi=\phi$, where $\chi:W_F' \lto {}^L\GL_{1F}$ is a quasi-character invariant under $\Gal(E/F)$.  If $H^2(\Gal(E/F),\CC^{\times})=0$, then $\chi$ can be taken to be trivial.
\end{lem}

Before we begin the proof we set a little notation.  Let $\varphi$ and $\phi$ be $L$-parameters as above.  We let
\begin{align} \label{naughts}
\varphi_0:W_F' &\lto ({}^L\GL_{nF})^{\circ}=\GL_n(\CC)\\
\nonumber \phi_0:W_E' &\lto ({}^L\GL_{nE})^{\circ}=\GL_n(\CC)
\end{align}
be the homomorphisms defined by composing $\varphi$ (resp.~$\phi$) with the projection ${}^L\GL_{nF} \to ({}^L\GL_{nF})^{\circ}$ (resp. ${}^L\GL_{nE} \to ({}^L\GL_{nE})^{\circ}$).

\begin{proof}

By assumption, for every $\sigma \in \Gal(E/F)$ we are given a $c(\sigma) \in \GL_n(\CC)$ such that
$$
\phi_0(\sigma \zeta \sigma^{-1})=c(\sigma)\phi_0(\zeta)c(\sigma)^{-1}.
$$
Since $\phi_0$ is irreducible, Schur's lemma implies that $c(\sigma)c(\tau)=\lambda(\sigma,\tau)c(\sigma\tau)$ for some $\lambda(\sigma, \tau) \in \CC^{\times}$.
In other words, the projective representation
$$
P\phi_0:W_E' \lto ({}^L\GL_{nE})^{\circ} \lto \mathrm{PGL}_n(\CC)
$$
obtained by composing $\phi_0$ with the canonical projection
can be extended to a (continuous) projective representation
$$
\psi:W_F'\lto \mathrm{PGL}_n(\CC).
$$
This extension has the property that $\psi(w)$ is semisimple for all $w \in W_F$.

By \cite[\S 8]{Rajan2},  there is an $L$-parameter
$$
\varphi:W_F' \lto {}^L\GL_{nF}
$$
such that $\varphi$ is a lift of $\psi$.  Let $P(b_{E/F}(\varphi)_0)$ denote
 the composite of $b_{E/F}(\varphi)_0=\varphi_0|_{W_E'}$ with the projection ${}^L\GL_{nE} \to \mathrm{PGL}_n(\CC)$.  We have that
$$
P(b_{E/F}(\varphi)_0) \cong P \phi_0.
$$
It follows that $b_{E/F}(\varphi)_0 \cong \phi_0 \otimes \chi$ for some character $\chi:W_E' \to \CC^{\times}$ invariant under $\Gal(E/F)$.

To complete the proof of the lemma, we need to show that if $H^2(\Gal(E/F),\CC^{\times})=0$, then any character $\chi:W_E' \to \CC^{\times}={}^L\GL_{1F}$ that is invariant under $\Gal(E/F)$ is the restriction of a character of $W_F'$.
 Viewing $\CC^{\times}$ as a trivial $W_F'$ and $\Gal(\bar{F}/F$)-module, we have an inflation-restriction exact sequence \cite[\S 3]{Rajan2}
\begin{align} \label{inf-res}
\begin{CD}
H^1(W_F',\CC^{\times}) @>{\mathrm{res}}>> H^1(W_E',\CC^{\times})^{W_E'/W_F'} @>>> H^2(W_F'/W_E',\CC^{\times})
\end{CD}
\end{align}
coming from the Hochschild-Serre spectral sequence.  Here $H$ denotes the Moore cohomology groups.  We note that for $i \geq 1$, $G$ discrete the Moore cohomology group $H^i(G,M)$ is equal to the usual continuous group cohomology group \cite[\S 3]{Rajan2}.  Since $W_F'/W_E' \cong \Gal(E/F)$, this completes the proof of the lemma.
\end{proof}

One would like to construct functorial transfers of automorphic representations corresponding to the base change map on $L$-parameters recalled above.
The $n=1$ case is trivial, as we now explain:  Given a quasi-character
$$
\mu:\GL_1(\A_F) \lto \CC^{\times}
$$
trivial on $\GL_1(F)$ its base change $b_{E/F}(\mu)$ is given by
$$
b_{E/F}(\mu):=\mu \circ \N_{E/F}:\GL_1(\A_E) \lto \CC^{\times}
$$
where $\N_{E/F}$ is the norm map.  We have the following lemma characterizing the image of the base change:
\begin{lem} \label{lem-image-bc} Suppose that $E/F$ is Galois and $H^2(\Gal(E/F),\CC^{\times})=0$.  If $\eta:\GL_1(\A_E) \to \CC^{\times}$ is a quasi-character trivial on $\GL_1(E)$ satisfying $\eta^{\sigma} =\eta$ for all $\sigma \in \Gal(E/F)$ then $\eta=\chi \circ \N_{E/F}$ for some quasi-character $\chi:\GL_1(\A_F) \to \CC^{\times}$ trivial on $\GL_1(F)$.
\end{lem}
\begin{proof}
We have a commutative diagram
\begin{align} \label{nice-diag}
\begin{CD}
 E^{\times} \backslash \A_E^{\times} @>{r_E}>> W_E'/W'^c_E\\
@V{\N_{E/F}}VV @VVV\\
F^{\times} \backslash \A_F^{\times} @>{r_F}>> W_F'/W'^c_F
\end{CD}
\end{align}
where $(\cdot)^c$ denotes the closure of the commutator subgroup of $(\cdot)$ and the right homomorphism is induced by the inclusion $W_E' \leq W_F'$ \cite[(1.2.2)]{Tate}.  As we proved in Lemma \ref{lem-bc-param}, any quasi-character of $W_E'/W'^{c}_E$ that is invariant under $\Gal(E/F)$ is the restriction of a quasi-character of $W_F'/W'^{c}_F$.
   Translating this to the left hand side of \eqref{nice-diag}, this implies that any quasi-character of $\A_E^{\times}$ trivial on $E^{\times}$ that is invariant under $\Gal(E/F)$
   is the composition of a quasi-character of $\A_F^{\times}$ trivial on $F^{\times}$ with the norm map.
\end{proof}

\subsection{Primitive parameters and automorphic representations}
\label{ssec-primitive} Let $F$ be a number field.
It is convenient to introduce the following definition:
\begin{defn} \label{defn-induced-L}
An $L$-parameter $\varphi:W_F' \to {}^L\GL_{nF}$ is \textbf{$K$-induced} if there is a nontrivial field extension $K/F$ of finite degree and an irreducible $L$-parameter $\phi:W_K' \to {}^L\GL_{nK}$ such that $\varphi \cong \mathrm{Ind}_{W_K'}^{W_F'}\phi$.  If $E/F$ is a nontrivial field extension, then an \textbf{$E$-primitive} $L$-parameter $\varphi$ is an irreducible $L$-parameter such that $\varphi$ is not $K$-induced for all subfields $E \geq K > F$.
\end{defn}
We denote by
\begin{align}
\Phi^{\mathrm{prim}}_n(E/F):=\{ \textrm{Equiv.~classes of $E$-primitive }L\textrm{-parameters }\varphi:W_F' \to {}^L\GL_{nF}\}.
\end{align}

Let $k$ be a global or local field and let $K/k$ be an \'etale $k$-algebra.  Let $\bar{k}$ be a choice of algebraic closure of $k$.  Write $K=\oplus_iK_i$ where the $K_i$ are finite extension fields of $k$.
Let
$$
\mathrm{Ind}_{K}^k:{}^L\mathrm{R}_{K/k}\GL_{nk} \to {}^L\GL_{n[K:k]k}
$$
be the $L$-map that embeds $({}^L\mathrm{R}_{K/k}\GL_{nk})^{\circ}=\times_i\GL_n(\CC)^{\mathrm{Hom}_k(K_i,\bar{k})}$ diagonally and sends $W_{k}'$ to the the appropriate group of permutation matrices\footnote{$L$-maps are defined in \cite[\S 15.1]{Borel}.}.  We recall that $L$-parameters $\phi:W_K' \to {}^L\GL_{nK}$ can be identified canonically with $L$-parameters $\phi:W_k' \to {}^L\mathrm{Res}_{K/k}\GL_{nk}$
\cite[Proposition 8.4]{Borel}; under  this identification $\mathrm{Ind}_K^k(\phi)=\oplus_i\mathrm{Ind}_{W'_{K_i}}^{W_k'}(\phi)$.  Using the local Langlands correspondence, for any irreducible admissible representation $\Pi_v$ of $\GL_n(E \otimes_FF_v)$ we can associate an irreducible admissible representation $\pi_v$ of $\GL_n(F_v)$ by stipulating that if $\phi:W_{F_v}' \to {}^L\mathrm{R}_{E/F}\GL_n$ is the $L$-parameter attached to $\Pi_v$ then $\mathrm{Ind}_E^F \circ \phi$ is the $L$-parameter attached to $\pi_v$.  If this is the case then we write
$$
\pi_v \cong \mathrm{Ind}_E^F(\Pi_v).
$$

\begin{defn}
An automorphic representation $\pi$ of $\GL_{n}(\A_F)$ is \textbf{$K$-automorphically induced} if there is a nontrivial finite extension field $K/F$ and an automorphic representation $\Pi$ of $\GL_n(\A_K)=\mathrm{Res}_{K/F}\GL_n(\A_F)$
such that
$\pi_v \cong \mathrm{Ind}_K^F(\Pi_v)$ for almost all places $v$ of $F$.  If $E/F$ is a nontrivial field extension then an \textbf{$E$-primitive} automorphic representation of $\GL_n(\A_F)$ is a cuspidal automorphic representation of $\GL_{n}(\A_F)$ that is not $K$-induced for all subfields $E \geq K >F$.
\end{defn}
For field extensions $E/F$ let
\begin{align}
\Pi^{\mathrm{prim}}_n(E/ F):=\{\textrm{isom.~classes of $E$-primitive automorphic reps.~of }\GL_{n}(\A_F)\}.
\end{align}

\subsection{Restriction of parameters}
\label{ssec-rest-param}
In \S \ref{ssec-desc} we discussed descent of parameters along a Galois extension $E/F$; the main result being that if $H^2(\Gal(E/F),\CC^{\times})=0$ then $\Gal(E/F)$-invariant irreducible parameters descend.  In this subsection we explore certain converse statements involving restrictions of parameters.  The main result is Proposition \ref{prop-bij-EF'}.  The statement parallel to Proposition \ref{prop-bij-EF'} in the context of automorphic representations is Conjecture \ref{conj-1}, the conjecture that appeared in the statement of Theorem \ref{main-thm-1}.

Let $E \geq K \geq F$ be a subfield.  For the remainder of this section, to ease notation we will often write $K$ where more properly we should write $W_K'$, e.g.
$$
\varphi|_{K}:=\varphi|_{W_K'}.
$$

We begin with the following lemma:

\begin{lem} \label{lem-restriction} Let $E/F$ be a Galois extension of number fields such that $H^2(\Gal(E/F),\CC^{\times})=0$.  Let $\varphi:W_F' \to {}^L\GL_{nF}$ be an irreducible $L$-parameter.
Either there is a subfield $E \geq K \geq F$ and an irreducible $L$-parameter $\phi:W_K' \to {}^L\GL_{mK}$ such that $\varphi \cong \mathrm{Ind}_{K}^{F}\phi$  or there is an $L$-parameter $\varphi_1:W_F' \to {}^L\GL_{mF}$ with $\varphi_1|_{E}$ irreducible and a finite-dimensional irreducible representation $\rho:\Gal(E/F) \to \GL_{d}(\CC)$ such that
$$
\varphi \cong \rho \otimes \varphi_1.
$$
Here we view $\rho$ as an $L$-parameter via the quotient map $W_{F}' \to W_{E}'/W_F'=\Gal(E/F)$.
\end{lem}
We note in particular that in the notation of the lemma one has $m[K:F]=n$ in the former case and  $md=n$ in the latter.  The extreme cases $m=1$ and $d=1$ occur.

For our use in the proof of this lemma and later, we record the following:
\begin{lem} \label{lem-basic} Let $H \leq G$ be groups with $H$ normal in $G$ and $[G:H]< \infty$. Moreover let $\varphi:G \to \mathrm{Aut}(V)$ be a finite-dimensional complex representation that is irreducible upon restriction to $H$.  Then
$$
\mathrm{Ind}_{H}^G(1) \otimes \varphi \cong \mathrm{Ind}_{H}^G(\varphi|_H)
$$
and $\rho \otimes \varphi$ is irreducible for any irreducible representation $\rho$ of $G/H$.
\end{lem}
\begin{proof}
The first statement is \cite[\S 3.3, Example 5]{SerreFG}.

As a representation of $G$, one has $\mathrm{Ind}_H^G(1) \cong \oplus_{i=1}^n \rho_i^{\oplus \mathrm{deg}(\rho_i)}$, where the sum is over a set of representatives for the irreducible representations of $G/H$.  Thus to prove the second statement of the lemma it suffices to show that
$$
\mathrm{dim}_{\CC}\mathrm{Hom}_{G}(\mathrm{Ind}_{H}^G(1) \otimes \varphi,\mathrm{Ind}_{H}^G(1) \otimes \varphi)=\sum_{i}\mathrm{deg}(\rho_i)^2.
$$
By the first assertion of the lemma and Frobenius reciprocity we have
\begin{align*}
\mathrm{Hom}_{G}(\mathrm{Ind}_{H}^G(1) \otimes \varphi,\mathrm{Ind}_{H}^G(1) \otimes \varphi)
&\cong\mathrm{Hom}_{G}(\mathrm{Ind}_{H}^G(1) \otimes \varphi,\mathrm{Ind}_{H}^G(\varphi|_H))\\&\cong\mathrm{Hom}_{H}(\oplus_{i=1}^n \varphi|_H^{\mathrm{deg}(\rho_i)^2},\varphi|_H)
\end{align*}
which has dimension $\sum_{i=1}^n \mathrm{deg}(\rho_i)^2$.
\end{proof}

\begin{proof}[Proof of Lemma \ref{lem-restriction}]
Assume that there does not exist a subfield $E \geq K \geq F$ and an irreducible $L$-parameter $\phi:W_K' \to {}^L\GL_{mK}$ such that $\varphi \cong \mathrm{Ind}_{K}^{F}(\phi)$.  Then, by \cite[\S 8.1, Proposition 24]{SerreFG} the restriction $\varphi|_{E}$ is isomorphic to a direct sum of some number of copies of a fixed
irreducible $L$-parameter $\phi_1:W_E' \to {}^L\GL_{mE}$.  Since $\varphi|_E^{\tau} \cong \varphi|_E$ for $\tau \in W_F'$ (trivially) it follows in particular that $\phi_1$ is $\Gal(E/F)$-invariant and therefore descends to a parameter $\varphi_1:W_F' \to {}^L\GL_{mF}$ by Lemma \ref{lem-bc-param}.  By Lemma \ref{lem-basic} one has
$$
\mathrm{Ind}_{E}^{F}(1) \otimes \varphi_1 \cong \mathrm{Ind}_{E}^{F}(\phi_1).
$$
Applying Frobenius reciprocity we see that
$$
0 \neq \mathrm{Hom}_{E}(\varphi|_{E},\phi_1)=\mathrm{Hom}_{F}(\varphi,\mathrm{Ind}_{E}^{F} (\phi_1))=\mathrm{Hom}_{F}(\varphi,\mathrm{Ind}_{E}^{F}(1) \otimes \varphi_1)
$$
which, in view of Lemma \ref{lem-basic}, completes the proof of the lemma.
\end{proof}
As an example, we have the following corollary:

\begin{cor} \label{cor-restriction} Under the assumptions of Lemma \ref{lem-restriction}, if $\Gal(E/F)$ is the universal perfect central extension of a finite simple nonabelian group $G$,  $n=2$ and $\varphi|_{W'_E}$ is reducible, then $G=A_5$.
\end{cor}

\begin{proof}
By Lemma \ref{lem-restriction}, if $\varphi|_{E}$ is reducible, then either there is a quadratic extension $K/F$ contained in $E$ such that $\varphi  \cong \mathrm{Ind}_{K}^{F}\varphi_1$ for some parameter
$\varphi_1:W_K' \to {}^L\GL_{1K}$ or one has a nontrivial representation $\Gal(E/F) \to \GL_2(\CC)$ (there are no nontrivial one-dimensional representations of $\Gal(E/F)$ since $\Gal(E/F)$ is perfect).  In the former case the extension $K/F$ would correspond to an index $2$ subgroup $H \leq \Gal(E/F)$, which would a fortiori be normal.  Thus we would have $\Gal(E/F)/H \cong \ZZ/2$ contradicting the assumption that $\Gal(E/F)$ is perfect.  Hence we must be in the latter case.  The nontrivial representation $\Gal(E/F) \to \GL_2(\CC)$ induces a nontrivial projective representation $G \to \mathrm{PGL}_2(\CC)$ since $\Gal(E/F)$ is perfect.  By a well-known theorem of Klein, if $G$ is a finite simple nonabelian group and $G \to \mathrm{PGL}_2(\CC)$ is a nontrivial projective representation, then $G\cong A_5$.
\end{proof}

In view of Lemma \ref{lem-restriction}, for each $n$ there are two natural cases to consider, namely the case where there is a nontrivial representation $\Gal(E/F) \to \GL_m(\CC)$ for some $m|n$ and the case where there is no nontrivial representation $\Gal(E/F) \to \GL_m(\CC)$ for any $m|n$.  We will deal with the former case under the additional assumption that $n=2$ in \S \ref{ssec-icosa-gp} below.
 In the latter case one obtains a complete description of the fibers and image of base change on primitive parameters as follows:

\begin{prop} \label{prop-bij-EF} Let $E/F$ be a Galois extension of number fields such that $\Gal(E/F)$ is the universal perfect central extension of a finite simple nonabelian group.  \begin{enumerate}
\item If $\varphi_1,\varphi_2 :W_{F} \to {}^L\GL_{nF}$ are $L$-parameters such that $\varphi_1|_{E}$ and $\varphi_2|_{E}$ are irreducible and isomorphic, then $\varphi_1 \cong \varphi_2$.

\item Assume that for all divisors $m|n$ there are no nontrivial irreducible representations $\Gal(E/F) \to \GL_{m}(\CC)$.  Under this assumption, restriction of parameters induces a bijection
 \begin{align*}
b_{E/F}:\Phi_n^{\mathrm{prim}}(E/F) &\tilde{\lto} \Phi_n^0(E)^{\Gal(E/F)}\\
\varphi &\longmapsto \varphi|_{E}.
 \end{align*}
\end{enumerate}
\end{prop}

\begin{proof}
We first check (1).  Suppose that $\varphi_1,\varphi_2:W_{F}' \to {}^L\GL_{nF'}$ are two irreducible parameters with isomorphic irreducible restrictions to $W_{E}'$.
Then by Frobenius reciprocity and Lemma \ref{lem-basic} we have
\begin{align*}
0 \neq \mathrm{Hom}_E(\varphi_1|_E,\varphi_2|_E)&=\mathrm{Hom}_F(\mathrm{Ind}_{E}^F(\varphi_1|_E),\varphi_2)\\&=\oplus_i \mathrm{Hom}_F(\rho_i \otimes \varphi_1,\varphi_2)^{\oplus \mathrm{deg}(\rho_i)}
\end{align*}
where the sum is over a set of representatives for the irreducible representations of $\Gal(E/F)$.
By Lemma \ref{lem-basic}, $\rho_i \otimes \varphi_1$ is irreducible for all $i$, so by considering degrees we must have $\rho_i \otimes \varphi_1 \cong \varphi_2$ where $\rho_i$ is an abelian character of $\Gal(E/F)$.
 Since $\Gal(E/F)$ is perfect, this $\rho_i$ is necessarily trivial.

Moving on to (2), we note that the restriction map from $L$-parameters of $W_F'$ to $L$-parameters of $W_E'$ obviously has image in the set of $\Gal(E/F)$-invariant parameters and under the addition assumption in (2) it has image in the set of irreducible parameters by Lemma \ref{lem-restriction}.  In other words we have a well-defined
map
$$
b_{E/F}:\Phi_n^{\mathrm{prim}}(E/F) \lto \Phi_n^{0}(E)^{\Gal(E/F)}.
$$
It is injective by (1) and surjective by Lemma \ref{lem-bc-param}, which completes the proof of the proposition.
\end{proof}

To set up trace identities it is convenient to work with automorphic representations attached to a subfield $F' \leq E$.  In view of this we prove the following modification of Proposition \ref{prop-bij-EF}:

\begin{prop} \label{prop-bij-EF'}  Let $E/F$ be a Galois extension of number fields such that $\Gal(E/F)$ is the universal perfect central extension of a finite simple nonabelian group.  Assume that for all $m|n$ there are no nontrivial irreducible representations $\Gal(E/F) \to \GL_{m}(\CC)$.  If $E \geq F' \geq F$ is a subfield  then the restriction map induces an injection
\begin{align} \label{restr-map}
b_{F'/F}:\Phi_n^{\mathrm{prim}}(E/F) & \lto \Phi_n^{\mathrm{prim}}(E/F') \\
\varphi &\longmapsto \varphi|_{{F'}}. \nonumber
\end{align}
If $\phi':W_{F'}' \to {}^L\GL_{nF'}$ is an $L$-parameter such that $\phi'|_{W_E'}$ is irreducible and $\Gal(E/F)$-invariant then there is a unique character $\chi' \in \Gal(E/F')^{\wedge}$ such that $\phi' \otimes \chi'$ is in the image of the restriction map \eqref{restr-map}.  If $\Gal(E/F')$ is solvable of order coprime to $n$ then for any irreducible $L$-parameter $\phi':W_{F'}' \to {}^L\GL_{nF'}$ the restriction $\phi'|_{W_E'}$ is again irreducible.
\end{prop}

\begin{proof} Note that $\varphi|_{E}=(\varphi|_{F'})|_{E}$.  Thus part (2) of
Proposition \ref{prop-bij-EF} implies that restriction of $L$-parameters from $W_F'$ to $W_{F'}'$ maps primitive $L$-parameters to primitive $L$-parameters, so \eqref{restr-map} is well-defined.  Parts (1) and (2) of Proposition \ref{prop-bij-EF} imply that \eqref{restr-map} is injective.

Now suppose that $\phi':W_{F'} \to {}^L\GL_{nF'}$ is an $L$-parameter such that $\phi'|_{E}$ is irreducible and $\Gal(E/F)$-invariant.  By Lemma \ref{lem-bc-param} the restriction $\phi'|_{{E}}$ descends to an irreducible parameter $\varphi:W_F' \to {}^L\GL_{nF}$.  By Frobenius reciprocity and Lemma \ref{lem-basic} we have
\begin{align}
\mathrm{Hom}_{E}(\varphi|_{E},\phi'|_{E})=
\mathrm{Hom}_{{F'}}(\mathrm{Ind}_{E}^{F'}(\varphi|_{E'}),\phi')
=\oplus_{i}\mathrm{Hom}_{{F'}}(\rho_i \otimes \varphi|_{{F'}},\phi')^{\mathrm{deg}(\rho_i)}
\end{align}
where the sum is over a set of representatives for the irreducible representations of $\Gal(E/F')$.
The first space is one dimensional and hence so is the last.  By Lemma \ref{lem-basic} $\rho_i \otimes \varphi|_{F'}$ is irreducible for all $i$, so by considering dimensions we see that
$\rho_i \otimes \varphi|_{F'} \cong \phi'$ for some character $\rho_i$ of $\Gal(E/F')$.
This proves the second claim of the proposition.

We are left with the final assertion of the proposition.  Since $\Gal(E/F')$ is solvable there is a chain of subfields $F'=E_0 \leq  \cdots \leq E_n=E$ such that $E_j/E_{j-1}$ is cyclic of prime degree.  Using this fact the final assertion follows from Lemma \ref{lem-restriction}.
\end{proof}

Motivated by Proposition \ref{prop-bij-EF'}, we make the following conjecture, which is an elaboration of a case of Langlands functoriality:

\begin{conj} \label{conj-1}
Let $E/F$ be a Galois extension of number fields and let $n$ be an integer such that
\begin{itemize}
\item $\Gal(E/F)$ is the universal perfect central extension of a finite simple nonabelian group, and
\item For every divisor $m|n$ there are no nontrivial irreducible representations $\Gal(E/F) \to \GL_{m}(\CC)$.
\end{itemize}
Let $E \geq F' \geq F$ be a subfield.
Every $E$-primitive automorphic representation $\pi$ of $\GL_n(\A_F)$ admits a unique base change $\pi_{F'}$ to $\GL_n(\A_{F'})$ and a unique base change to $\GL_n(\A_E)$, the first of which is an $E$-primitive automorphic representation. Thus base change induces an injection
\begin{align*}
b_{E/F'}:\Pi_n^{\mathrm{prim}}(E/F) &\lto \Pi_n^{\mathrm{prim}}(E/F')\\
\pi &\longmapsto \pi_{F'}
\end{align*}

If $\pi'$ is a cuspidal automorphic representation of $\GL_n(\A_{F'})$ such that its base change $\pi'_E$ to $\GL_n(\A_E)$ is cuspidal and $\Gal(E/F)$-invariant then $\pi'_E$ descends to an automorphic representation of $\GL_n(\A_F)$.
\end{conj}

We also require a conjecture which can be addressed using endoscopic techniques, is discussed at length in \cite{Rajan3}, and is a theorem when $n=2$ \cite[Theorems 1 and 2]{Rajan3} or $\Gal(E/F')$ is cyclic \cite[Chapter 3, Theorems 4.2 and 5.1]{AC}:

\begin{conj} \label{conj-solv} Let $E/F'$ be a solvable extension of number fields and let $\Pi$ be a cuspidal automorphic representation of $\GL_n(\A_E)$.  If $\Pi$ is $\Gal(E/F')$-invariant, then there is a $\Gal(E/F')$-invariant character $\chi \in (E^{\times} \backslash \A_E^{\times})^{\wedge}$ such that $\Pi \otimes \chi$ descends to $\GL_n(\A_{F'})$.  If $H^2(\Gal(E/F'),\CC^{\times})=0$, then $\chi$ can be taken to be the trivial character.  Conversely, if $\pi'_1$, $\pi_2'$ are cuspidal automorphic representations of $\GL_n(\A_{F'})$ that both base change to a cuspidal automorphic representation $\Pi$ of $\GL_n(\A_E)$, then there is a unique $\chi \in \Gal(E/F')^{\wedge}$ such that $\pi_1' \cong \pi_2' \otimes \chi$.
\end{conj}

\subsection{The icosahedral group} \label{ssec-icosa-gp}

Assume that $n=2$ and that
$$
\Gal(E/F) \cong \mathrm{SL}_2(\ZZ/5) \cong \widetilde{A}_5,
$$
the universal perfect central extension of $A_5$.  We fix such an isomorphism for the remainder of this section and view it as an identification: $\Gal(E/F)=\widetilde{A}_5$.  In this subsection we describe the image and fibers of the base change map on $L$-parameters in this setting.  This description is used as motivation for Conjecture \ref{conj-2} below, the conjecture used in the statement of Theorem \ref{main-thm-2} above.

As remarked below Theorem \ref{main-thm-1}, if $n=2$ the case where $\Gal(E/F)$ is the universal perfect central extension of $A_5$ is the only
case in which the hypotheses of Proposition \ref{prop-bij-EF} do not hold.  Moreover, the conclusion of Proposition \ref{prop-bij-EF} does not hold.  Indeed, any irreducible $2$-dimensional representation of $\SL_2(\ZZ/5)$ induces an irreducible $L$-parameter $\varphi:W_{F}' \to {}^L\GL_{2F}$ such that $\varphi|_{E}$ is the direct sum of two copies of the trivial representation.  For our purposes it is more important to find an analogue of Proposition \ref{prop-bij-EF'}.

The facts from group theory that we require in this subsection are collected in \S \ref{appendix}.  Fix an injection $A_4 \hookrightarrow A_5$ and let $\widetilde{A}_4$ denote the inverse image of $A_4$ under the projection map $\widetilde{A}_5 \to A_5$.  It is a double cover of $A_4$.

\begin{lem} \label{lem-gen} Let $\tau \in \Gal(E/F)$ be of order $5$.  Then
$\langle \tau ,\widetilde{A}_4\rangle=\Gal(E/F)$.
\end{lem}
\begin{proof}
By Lagrange's theorem for any element $\tau \in \Gal(E/F)$ of order $5$ the group $\langle \tau,\widetilde{A}_4 \rangle$ has order divisible by $(5)(24)=120$.
\end{proof}

 Our analogue of Proposition \ref{prop-bij-EF'} is the following proposition:
\begin{prop} \label{prop-A5-EF}
Assume that $F'=E^{\widetilde{A}_4}$  and $\tau \in \Gal(E/F)$ is of order $5$. In this case restriction of parameters induces a map
\begin{align} \label{restr-map2}
b_{F'/F}:\Phi_2^0(F) &\lto \Phi_2^0(F')\\
\varphi &\longmapsto \varphi|_{{F'}}. \nonumber
\end{align}
If $\phi':W_{F'}' \to {}^L\GL_{2F'}$ is an $L$-parameter such that $\phi'|_{E}$ is irreducible and $\Gal(E/F)$-invariant
then $\phi' \otimes \chi'$ is in the image of the restriction map \eqref{restr-map2} for a unique $\chi' \in \Gal(E/F')^{\wedge}$.

If $\phi'|_{E}$ is reducible and $\mathrm{Hom}_{E}(\phi'|_{E},\phi'|_{E}^{\tau}) \neq 0$ then $\phi'|_E$ is the restriction of a parameter $\phi:W_F' \to {}^L\GL_{2F}$. There are exactly two nonisomorphic irreducible $\phi_1,\phi_2 :W_{F}' \to {}^L\GL_{2F}$ such that $\phi|_{E} \cong \phi_{1}|_{E} \cong \phi_{2}|_{E}$.
\end{prop}

\begin{proof}
We first check that an irreducible $L$-parameter $\varphi$ as above restricts to an irreducible $L$-parameter on $W_{F'}'$.    Since $\SL_2(\ZZ/5)$ is perfect, there are no subgroups of $\mathrm{SL}_2(\ZZ/5)$ of index $2$.
Since $\varphi$ has degree $2$ it follows from Lemma \ref{lem-restriction} that $\varphi$ is not induced,
and hence $\varphi|_{E}$ is either irreducible or
$$
\varphi \cong \chi \otimes \rho
$$
for some character $\chi:W_F' \to {}^L\GL_{1F}$ and some irreducible representation
$$
\rho:\Gal(E/F) \to \GL_2(\CC).
$$
In the former case $\varphi|_{F'}$ is also irreducible, and hence we are done.  Suppose on the other hand that $\varphi \cong \chi \otimes \rho$.
Notice that any irreducible two-dimensional representation of $\Gal(E/F)$ is necessarily
faithful. Indeed, the only normal subgroup of
$\Gal(E/F)$ is its center and if such a representation
was trivial on the center it would descend to a
representation of $A_5$, a group that has no irreducible
two-dimensional representations.  Since $\Gal(E/F')$ is nonabelian
$\rho(\Gal(E/F'))$ is nonabelian and it follows that $\rho|_{{F'}}$ is
irreducible and hence so is $\varphi|_{F'}$.

The second statement of the proposition is proved by the same argument as the analogous statement in Proposition \ref{prop-bij-EF'}.

For the last assertion assume that  $\phi':W_{F'}' \to {}^L\GL_{2F'}$ is an $L$-parameter such that $\phi'|_{E}$ is reducible.   It follows from Lemma \ref{lem-restriction} that there is an element $\sigma_0 \in \Gal(E/F')$ of order dividing $2$ and a character $\chi_0:W_E' \to {}^L\GL_{1E}$ such that $\phi'|_{E} \cong \chi_0 \oplus \chi_0^{\sigma_0}$.
Since $\phi'|_E \cong \phi'|_E^{\sigma}$ for all $\sigma \in \Gal(E/F)$, the group $W_{F'}' / W_E'$ acts by conjugation on these two factors and this action defines a homomorphism $\Gal(E/F')\cong W_{F'}'/W_E' \to \ZZ/2$.   Now $\widetilde{A}_4$ has no subgroup of index two, so this implies that homomorphism $\widetilde{A}_4 \to \ZZ/2$ just considered is trivial and hence the action of $W_{F'}'/W_E'$ on the pair $\{\chi_0,\chi_0^{\sigma_0}\}$ is trivial.  It follows in particular that $\chi_0 \cong \chi_0^{\sigma_0}$ and additionally $\chi_0$ is isomorphic to all of its $\Gal(E/F')$-conjugates.  If additionally
$\mathrm{Hom}_{E}(\phi'|_{E},\phi'|_{E}^{\tau}) \neq 0$ then $\chi_0$ is fixed under $\Gal(E/F')$ and $\tau$ and hence it is isomorphic to all of its $\Gal(E/F)$-conjugates by Lemma \ref{lem-gen}.
Thus $\chi_0$ descends to a character $\chi:W_F' \to {}^L\GL_{1F}$ by Lemma \ref{lem-bc-param}.

Let $\rho_2$ be a rank two irreducible representation of $\Gal(E/F)$ with character $\theta_2$
in the notation of \S \ref{appendix} and let $\langle \xi \rangle =\Gal(\QQ(\sqrt{5})/\QQ)$.  Then
$$
\rho_2 \otimes \chi \quad \textrm{ and } \quad \xi \circ \rho_2 \otimes \chi
$$
are two nonisomorphic $L$-parameters from $W_F'$ with restriction to $W_E'$ isomorphic to $\phi'|_E$.

If $\phi:W_F' \to {}^L\GL_{2F}$ is any $L$-parameter with $\phi|_E \cong \phi'|_{E}$, then
\begin{align} \label{1-frob}
2=\mathrm{dim}(\mathrm{Hom}_{E}(\chi|_{E},\phi|_{E}))=
\mathrm{dim}(\mathrm{Hom}_{{F}}(\mathrm{Ind}_{{E}}^{{F}}(\chi|_E),\phi)).
\end{align}
Now by Lemma \ref{lem-basic}
$$
\mathrm{Ind}_{{E}}^{{F}}(\chi|_{{E}})\cong \mathrm{Ind}_{{E}}^{{F}}(1)\otimes \chi.
$$
This combined with \eqref{1-frob} implies that
$$
\phi \cong \rho_2 \otimes \chi \textrm{ or }\phi \cong \xi \circ \rho_2 \otimes \chi.
$$
\end{proof}

Motivated by Proposition \ref{prop-A5-EF} and Proposition \ref{prop-bij-EF}
we propose the following conjecture.  It is the (conjectural) translation of Proposition \ref{prop-A5-EF} and part (1) of Proposition \ref{prop-bij-EF} into
a statement on automorphic representations.

\begin{conj} \label{conj-2}
In the setting of Proposition \ref{prop-A5-EF} above each cuspidal automorphic representation $\pi$ of $\GL_2(\A_F)$ admits a unique cuspidal base change to $\GL_2(\A_{F'})$ and a unique base change to an isobaric automorphic representation of $\GL_2(\A_E)$.
If $\pi'$ is a cuspidal automorphic representation of $\GL_2(\A_{F'})$ such that $\pi'_E$ is cuspidal and $\mathrm{Hom}_I(\pi_E',\pi'^{\tau}_E) \neq 0$, then there is a unique cuspidal automorphic representation $\pi$ of $\GL_2(\A_F)$ that has $\pi'_E$ as a base change.  If $\pi'_E$ is not cuspidal and $\mathrm{Hom}_I(\pi'_E,\pi'^{\tau}_E) \neq 0$ then there are precisely two isomorphism classes of cuspidal automorphic representations of $\GL_2(\A_F)$ that base change to $\pi'_E$.
\end{conj}

\begin{rem} In understanding the analogy between Proposition \ref{prop-A5-EF} and Conjecture \ref{conj-2} it is helpful to recall that if $\pi'_E$ is cuspidal and $\mathrm{Hom}_I(\pi'_E,\pi'^{\tau}_E)\neq 0$ then $\pi'_E$ is isomorphic to all of its twists under elements of $\langle \Gal(E/F'),\tau\rangle=\Gal(E/F)$.
\end{rem}

\subsection{Motivating conjectures \ref{conj-32} and \ref{conj-33}} \label{ssec-artin-conj}

In this section we prove some lemmas on restriction of $L$-parameters along subfields of an $\widetilde{A}_5$-extension and then state the conjectures (namely conjectures \ref{conj-32} and \ref{conj-33}) that are the translation of these statements to the context of automorphic representations.  These conjectures are used in the statement of Theorem \ref{main-thm-3} above.

As above, we identify $\Gal(E/F) =\widetilde{A}_5$.  Fix an embedding $\ZZ/2 \times \ZZ/2 \hookrightarrow  A_5$ and let $Q \hookrightarrow \widetilde{A}_5$ be its inverse image under the quotient map $\widetilde{A}_5 \to A_5$; it is isomorphic to the quaternion group.

\begin{lem} \label{lem-A5-EF} Let $F'=E^Q$.
For all quasi-characters $\chi_0:W_{E}' \to {}^L\GL_{1E}$ invariant under $\Gal(E/F')$ there is an irreducible parameter $\varphi':W_{F'}' \to {}^L\GL_{2F'}$ such that $\varphi'|_E \cong \chi_0 \oplus \chi_0$.  The parameter $\varphi'$ is unique up to isomorphism.

Let $\varphi:W_F' \to {}^L\GL_{2F}$ be an irreducible $L$-parameter such that $\varphi|_E\cong \chi_0 \oplus \chi_0$ where $\chi_0:W_E' \to {}^L\GL_{1E}$ is $\Gal(E/F)$-invariant.  Then $\varphi|_{F'}$ is irreducible, and there are precisely two distinct equivalence classes of $L$-parameters in $\Phi^0_2(F)$ that restrict to $\varphi|_{F'}$.
Conversely, if $\varphi': W_{F'}' \to {}^L\GL_{2F'}$ is an irreducible parameter such that $\varphi'|_E \cong \chi_0 \oplus \chi_0$ for some quasi-character $\chi_0:W_{E}' \to {}^L\GL_{1E}$ invariant under $\Gal(E/F)$, then $\varphi'$ extends to an $L$-parameter on $W_F'$.
\end{lem}

\begin{proof} One can twist by $\chi_0^{-1}$ and its extension to $W_F'$ to reduce the lemma to the case where $\chi_0$ is trivial (recall that a $\Gal(E/F')$ (resp.~$\Gal(E/F)$)-invariant quasi-character descends by Lemma \ref{lem-bc-param} and the fact that both of these groups have trivial Schur multiplier).  In this case the lemma follows immediately from the character tables included in \S \ref{appendix} below (see Lemma \ref{lem-Q} in particular).
\end{proof}

The following is the conjectural translation of this statement (via Langlands functoriality) into the language of automorphic representations:

\begin{conj} \label{conj-32}
Let $F'=E^Q$.
 Let $\pi$ be a cuspidal automorphic representation of $\GL_2(\A_F)$ with base change $\chi_0 \boxplus \chi_0$ to an isobaric automorphic representation of $\GL_2(\A_E)$.  Then $\pi$ admits a base change $\pi_{F'}$ to $\GL_2(\A_{F'})$ that is cuspidal.  There are precisely two distinct isomorphism classes of cuspidal automorphic representations of $\GL_2(\A_{F})$ that base change to $\pi_{F'}$.  Conversely, if $\pi'$ is a cuspidal automorphic representation of $\GL_2(\A_{F'})$ such that $\pi'_E \cong \chi_0 \boxplus \chi_0$ where $\chi_0$ is $\Gal(E/F)$-invariant, then $\pi'$ descends to a cuspidal automorphic representation of $\GL_2(\A_{F})$.
\end{conj}

The situation for $n=3$ is similar:
\begin{lem} \label{lem-A5-EF3}
Let $F'=E^{\widetilde{A}_4}$.  Let $\chi_0:W_{E}' \to {}^L\GL_{1E}$ be a character invariant under $\Gal(E/F')$.  There is an irreducible parameter $\varphi':W_{F'}' \to {}^L\GL_{3F'}$ such that $\varphi'|_{W_E'} \cong \chi_0^{\oplus 3}$, unique up to isomorphism.

Let $\varphi:W_F' \to {}^L\GL_{3F}$ be an irreducible $L$-parameter such that $\varphi|_E\cong \chi_0^{\oplus 3}$ where
$$
\chi_0:W_E' \to {}^L\GL_{1E}
$$ is $\Gal(E/F)$-invariant.  Then $\varphi|_{F'}$ is irreducible, and there are precisely two inequivalent isomorphism classes of $L$-parameters in $\Phi^0_3(F)$ that restrict to the isomorphism class of $\varphi|_{F'}$.  Conversely, if $\varphi': W_{F'}' \to {}^L\GL_{3F'}$ is an irreducible parameter such that $\varphi'|_E \cong \chi_0^{\oplus 3}$ for some quasi-character $\chi_0:W_{E}' \to {}^L\GL_{1F}$ invariant under $\Gal(E/F)$, then $\varphi'$ extends to an $L$-parameter on $W_F'$.
\end{lem}

\begin{proof}One can twist by $\chi_0^{-1}$ and its extension to $W_F'$ to reduce the lemma to the case where $\chi_0$ is trivial (recall that a $\Gal(E/F')$ (resp.~$\Gal(E/F)$)-invariant quasi-character descends by Lemma \ref{lem-bc-param} and the fact that both of these groups have trivial Schur multiplier).  In this case the lemma follows immediately from the character tables included in \S \ref{appendix} (see Lemma \ref{lem-tetra-reps} in particular).
\end{proof}

The corresponding conjecture is the following:

\begin{conj} \label{conj-33}
Let $F'=E^{\widetilde{A}_4}$.   Let $\pi$ be a cuspidal automorphic representation of $\GL_3(\A_F)$ with base change $\chi_0^{\boxplus 3}$ to an isobaric automorphic representation of $\GL_3(\A_E)$.  Then $\pi$ admits a base change $\pi_{F'}$ to $\GL_3(\A_{F'})$ that is cuspidal.  There are precisely two nonisomorphic cuspidal automorphic representations of $\GL_3(\A_{F})$ that base change to $\pi_{F'}$.  Conversely, if $\pi'$ is a cuspidal automorphic representation of $\GL_3(\A_{F'})$ such that $\pi'_E \cong \chi_0^{\boxplus 3}$ where $\chi_0$ is $\Gal(E/F)$ invariant, then $\pi'$ descends to a cuspidal automorphic representation of $\GL_3(\A_{F})$.
\end{conj}

\subsection{Appendix:  The representations of some binary groups} \label{appendix}

In \S \ref{ssec-icosa-gp} and \S \ref{ssec-artin-conj} we studied the problem of base change along an extension $E/F$ where $\Gal(E/F)$  was isomorphic to the binary icosahedral group, that is, the universal perfect central extension $\widetilde{A}_5$ of the alternating group $A_5$ on $5$ letters.  Fix an embedding $A_4 \hookrightarrow A_5$, and let $\widetilde{A}_4$ be the inverse image of $A_4$ under the quotient map $\widetilde{A}_5 \to A_5$.  Similarly fix an embedding $\ZZ/2 \times \ZZ/2 \hookrightarrow A_5$ and let $Q$ be the inverse of $\ZZ/2 \times \ZZ/2$ under the quotient map $\widetilde{A}_5 \to A_5$.  In \S \ref{ssec-icosa-gp} and \S \eqref{ssec-artin-conj} we required various properties of the representations of $\widetilde{A}_5$, $\widetilde{A}_4$, and $Q$.  We collect these properties in this subsection for ease of reference.

We now write down the character table of $\widetilde{A}_5$.
For $n \in \{1,2,3,4,6\}$ let $C_n$ be the unique conjugacy class of $\widetilde{A}_5$ consisting of the elements of order $n$.  Let $C_5$ and $C_5'$ be the two conjugacy classes of elements of order $5$, and if $g \in C_5$ (resp.~$C_5'$) let $C_{10}$ (resp.~$C_{10}'$) be the conjugacy class of $-g$ (viewed as a matrix in $\SL_2(\ZZ/5) \cong \widetilde{A}_5$).  The degree of an irreducible representation is given by its subscript.  We let $u,v$ be the distinct roots of the polynomial $x^2-x-1$.  The following character table is in \cite[\S 7]{Buhler} (see \cite[Proof of Lemma 5.1]{KimIcos} for corrections).

\begin{center}
\begin{tabular}{ l | c |c |c | c | c | c | c| c |c |}
 & $C_1$ & $C_2$  & $C_4$ & $C_3$  & $C_6$ & $C_5$ & $C_{10}$ & $C_{5}'$ & $C_{10}'$  \\
\hline
$1$ & $1$ & $1$ & $1$ & $1$ & $1$ & $1$ & $1$ & $1$ & $1$\\
$\theta_3$ & $3$ & $3$ & $-1$ & $0$ & $0$ & $u$ & $u$ & $v$ & $v$ \\
$\theta_3'$ & $3$ & $3$ & $-1$ & $0$ & $0$ & $v$ & $v$ & $u$ & $u$ \\
$\theta_4$ & $4$ & $4$ & $0$ & $1$ & $1$ & $-1$ & $-1$ & $-1$ & $-1$\\
$\theta_5$ & $5$ &$5$ & $1$ & $-1$ & $-1$ & 0 & 0 & 0 & 0\\
$\theta_2$ & $2$ & $-2$ & $0$ & $-1$ & $1$ & $u-1$ & $1-u$ & $v-1$ & $1-v$ \\
$\theta_2'$ & $2$ & $-2$ & $0$ & $-1$ & $1$ & $v-1$ & $1-v$ & $u-1$ & $1-u$ \\
$\theta_4'$ & $4$ & $-4$ & $0$ & $1$ & $-1$ & $-1$ & $1$ & $-1$ & $1$ \\
$\theta_6$ & $6$ & $-6$ & $0$ & $0$ & $0$ & $1$ & $-1$ & $1$ & $-1$
\end{tabular}
\end{center}

Let $\chi$ be a nontrivial character of $\widetilde{A}_4$.  It is of order $3$, as $\widetilde{A}_4^{\mathrm{ab}} \cong \ZZ/3$.  Using the character table above, one proves the following lemma \cite[Lemmas 5.1-5.3]{KimIcos}
\begin{lem} \label{lem-icosa-reps}Let $\langle \xi \rangle =\Gal(\QQ(\sqrt{5})/\QQ)$.  The following is a complete list of irreducible characters of $\widetilde{A}_5$:
\begin{enumerate}
\item trivial
\item $\theta_2$, $\xi \circ \theta_2$ ($2$-dimensional)
\item $\mathrm{Sym}^2(\theta_2)$, $\mathrm{Sym}^2(\xi \circ \theta_2)$ ($3$-dimensional)
\item $\mathrm{Sym}^3(\theta_2)=\mathrm{Sym}^3(\xi \circ \theta_2)$, $\theta_2 \otimes \xi \circ \theta$ ($4$-dimensional)
\item $\mathrm{Ind}_{\widetilde{A}_4}^{\widetilde{A}_5}(\chi) = \mathrm{Sym}^4(\theta_2) = \mathrm{Sym}^4(\xi \circ \theta_2)$ ($5$-dimensional)
\item $\mathrm{Sym}^2(\theta_2) \otimes \xi \circ \theta_2 =\theta_2 \otimes \mathrm{Sym}^2(\xi \circ \theta_2)=\mathrm{Sym}^5(\theta_2)$ ($6$-dimensional)
\end{enumerate}
There two characters of degree $2,3,4$ given above are not equivalent.
\end{lem}

\begin{rem}
The fact that $\mathrm{Sym}^4(\theta_2) = \mathrm{Ind}_{\widetilde{A}_4}^{\widetilde{A}_5}(\chi)$ was observed by D.~Ramakrishnan (see \cite{KimIcos}).  We point this out because it turns out to be an important fact for the arguments of \S \ref{ssec-trace-to-func2} below.
\end{rem}

Next we discuss the representations of $\widetilde{A}_4$.  Write $t=(123)$ and let $\bar{C}_{t^i}$ be the conjugacy classes of $t^i$ for $i \in \{1,2\}$ in $A_4$, respectively.  The inverse image of $\bar{C}_{t^i}$ is a union of two conjugacy classes  $C_{t^i},C_{t^i}'$ for each $i \in \{1,2\}$.  We assume that for $c \in C_{t^i}$ one has $|c|=3$ and for $c' \in C_{t^i}'$ one has $|c'|=6$.
Write $C_2$ for the conjugacy class of elements of order $2$ and $C_4$ for the conjugacy class of order $4$.  One has the following character table:
\begin{center}
\begin{tabular}{ l | c |c |c | c | c | c | c }
 & $C_1$ & $C_2$ & $C_4$  & $C_t$ & $C_t'$& $C_{t^2}'$ & $C_{t^2}'$ \\
\hline
$1$ & $1$ & $1$ & $1$ & $1$ & $1$ & $1$ & $1$\\
$\psi_1$ & $1$ & $1$ & $1$ & $e^{2 \pi i/3}$ & $e^{2 \pi i/3}$ & $e^{4 \pi i/3}$ & $e^{4 \pi i/3}$   \\
$\psi_1^2$ & $1$& $1$ & $1$ & $e^{4 \pi i/3}$ & $e^{4 \pi i/3}$ & $e^{2 \pi i/3}$ & $e^{2 \pi i/3}$  \\
$\psi_3$ & $3$ & $3$ & $-1$ & $0$ & $0$ & 0 & $0$\\
$\psi_2$ & $2$ & $-2$ & $0$ & $-1$ & $1$ & $-1$ & $1$\\
$\psi_2 \psi_1$ &$2$ &$-2$ &$0$& $-e^{2 \pi i/3}$& $e^{2 \pi i/3}$& $-e^{4\pi i/3}$& $e^{4 \pi i/3}$\\
$\psi_2\psi_1^2$ &$2$ & $-2$&$0$& $-e^{4 \pi i/3}$ & $e^{4 \pi i/3}$ & $-e^{2 \pi i/3}$ & $e^{2 \pi i/3}$
\end{tabular}
\end{center}
We make a few comments on the computation of this table.  First, the characters that are lifts of characters of $A_4$ are computed in \cite[\S 5.7]{SerreFG}.  Second, we note that $\psi_2:=\theta_2|_{\widetilde{A}_4}$ is irreducible.  Indeed,  the only normal subgroup of $\widetilde{A}_5$ is the center and $\theta_2$ is not the restriction of a character of $A_5$ since there are no rank two characters of $A_5$.  Thus any representation with character $\theta_2$ is faithful.  Since $\theta_2$ is of degree $2$, if the representation with character $\theta_2|_{\widetilde{A}_4}$ were reducible, it would provide an isomorphism from $\widetilde{A}_4$ to an abelian group.  Since $\widetilde{A}_3$ is nonabelian, this shows that $\theta_2|_{\widetilde{A}_4}$ is irreducible.  Its character values therefore follow from the character table for $\widetilde{A}_5$ above.  The fact that the characters $\psi_2\psi_1^i$ are distinct for distinct $i \in \{1,2,3\}$ follows by considering determinants.  Using the fact that
that the sum of the squares of the degrees of the irreducible characters must equal the order of the group we see that the table is complete.

\begin{lem} \label{lem-tetra-reps} Let $\langle \xi \rangle=\Gal(\QQ(\sqrt{5})/\QQ)$.
One has
$$
\theta_2|_{\widetilde{A}_4}=\xi \circ \theta_2|_{\widetilde{A}_4}=\psi_2.
$$
Moreover
$$
\theta_3|_{\widetilde{A}_4}=\psi_3.
$$
\end{lem}
\begin{proof}  This follows immediately from the character tables above.
\end{proof}

Finally we record the character table for the quaternion group $Q$.
We present the group as
$$
Q= \langle i, j : i^4=1,\, \,i^2=j^2,\, \,i^{-1}ji=j^{-1} \rangle.
$$
Denoting by $C_x$ the conjugacy class of an element $x \in Q$, one has the following character table \cite[\S 19.1]{DF}:

\begin{center}
\begin{tabular}{ l | c |c |c | c | c | }
 & $C_1$ & $C_{-1}$ & $C_{i}$ & $C_{j}$ & $C_{ij}$ \\
\hline
$1$ & $1$ & $1$ & $1$ & $1$ & $1$  \\
$\Theta_1$ & $1$ & $1$ & $-1$ & $1$ & $-1$  \\
$\Theta_1'$ & $1$ & $1$ & $1$ & $-1$ & $-1$ \\
$\Theta_1\Theta_1'$ & $1$ & $1$ & $-1$ & $-1$ & $1$\\
$\Theta_2$ & $2$ & $-2$ & $0$ & $0$ & $0$
\end{tabular}
\end{center}
We note that as before the subscript indicates the degree of the representation.

By examining the character tables of $\widetilde{A}_5$ and $Q$ one immediately deduces the following lemma:
\begin{lem} \label{lem-Q} Let $\langle \xi \rangle=\Gal(\QQ(\sqrt{5})/\QQ)$.
One has $\theta_2|_{Q}=\xi \circ \theta_2|_Q=\Theta_2$.   \qed
\end{lem}

\section{Proofs of the main theorems} \label{sec-proofs}

In this section we prove the theorems stated in the introduction.

\subsection{Preparation}

The propositions of this subsection will be used in the proof of our main theorems in \S \ref{ssec-func-to-trace} and \S \ref{ssec-trace-to-func} below.

\begin{prop} \label{prop-solv} Let $E/F'$ be a Galois extension with $\Gal(E/F') \cong \widetilde{A}_4$ and let $\pi'$ be a cuspidal automorphic representation of $\GL_2(\A_{F'})$.  There are precisely $|\Gal(E/F')^{\mathrm{ab}}|$ non-isomorphic cuspidal automorphic representations of $\GL_2(\A_{F'})$ that have $\pi'_E$ as a base change.
\end{prop}

\begin{prop} \label{prop-solv3}
Let $E/F'$ be a Galois extension with $\Gal(E/F') \cong \widetilde{A}_4$
 and let $\pi'$ be a cuspidal automorphic representation of $\GL_3(\A_{F'})$.  If $\pi'_E\cong \chi_0^{\boxplus 3}$ where $\chi_0$ is a quasi-character invariant under $\Gal(E/F')$, then there is a unique cuspidal automorphic representations of $\GL_2(\A_{F'})$ that has $\pi'_E$ as a base change.  It is of $\rho_3$-type, where $\rho_3: W_{F'}' \to {}^L\GL_{3F'}$ is a representation trivial on $W_E'$ that has character equal to the unique degree three irreducible character of $\widetilde{A}_4$.
\end{prop}

These propositions correspond to the first (and easiest) assertions on $L$-parameters in lemmas \ref{lem-A5-EF} and \ref{lem-A5-EF3}, respectively.  They will be proven in a moment after some preparation. Let $P \leq \GL_{n}$ be a parabolic subgroup and let $P=MN$ be its Levi decomposition.  Suppose that $\Pi_M$ is a cuspidal automorphic representation of $M(E) A_{\GL_{nE}} \backslash M(\A_E)$ and that
$$
\Pi=\mathrm{Ind}_{M(\A_E)}^{\GL_{n}(\A_E)}(\Pi_M)
$$
is an (irreducible) automorphic representation of $\GL_n(E) A_{\GL_{nE}} \backslash \GL_n(\A_E)$. Here $\Pi_M$ is extended to a representation of $P(\A_E)$ by letting the action of $N(\A_F)$ be trivial.  We note that $\Pi$ is irreducible and unitary \cite[Chapter 3, \S 4]{AC}.
Write $M=\prod_{i} \GL_{n_i}$ for some set of integers $n_i \geq 1$ and $\Pi_M=\otimes_i\Pi_i$ where the $\Pi_i$ are cuspidal automorphic representations of $\GL_{n_i}(F) \backslash \GL_{n_i}(\A_F)$.

\begin{lem} \label{lem-const-term}
Suppose that $E/F'$ is a Galois extension of number fields and  $\Pi^{\sigma} \cong \Pi$ for all $\sigma \in \Gal(E/F')$.  Then $\{\Pi_i\}=\{\Pi_i^{\sigma}\}$ for all $\sigma \in \Gal(E/F')$.
\end{lem}

\begin{proof}
Since $\Pi$ is induced from cuspidal we use the theory of Eisenstein series to view $\Pi$ as a subrepresentation (not just subquotient) of $L^2(\GL_n(E) A_{\GL_{nE}} \backslash \GL_n(\A_E))$.  Let $V_{\Pi} \leq L^2(\GL_n(E) A_{\GL_{nE}} \backslash \GL_n(\A_E))$ be the subspace of $\Pi$-isotypic automorphic forms.  Consider the constant term
$$
\phi_P(m):=\int_{N(E) \backslash N(\A_E)} \phi(nm)dn.
$$
It is an automorphic form on $M(\A_E)$ \cite[Lemma 4]{LanglNotion}.
 There is a natural action of $\Gal(E/F')$ on $L^2( M(E)A_{\GL_{2E}} \backslash M(\A_E))$.  By the normal basis theorem one has $d(n^{\sigma})=dn$ for all $\sigma \in \Gal(E/F')$,  and hence the map
\begin{align*}
V_{\Pi} &\lto L^2(A_{\GL_{2E}} M(E) \backslash M(\A_E))\\
\phi &\longmapsto  \phi_P
\end{align*}
is $\Gal(E/F')$-invariant.  Using the theory of Eisenstein series, specifically \cite[Propositions II.1.7 and IV.1.9]{MW}, it follows that that $\Gal(E/F')$ preserves the set of representations $\Pi_{1M}$ of $M(F) A_{\GL_{nE}} \backslash M(\A_E)$ such that $\Pi$ is a constituent of $\mathrm{Ind}_{M(\A_E)}^{\GL_{n}(\A_E)}(\Pi_{1M})$.  Here, as before, we are extending $\Pi_{1M}$ to a representation of $P(\A_E)$ by letting $N(\A_E)$ act trivially.  To make this statement easier for the reader to check, we note that our assumptions imply that $\Pi$ is not in the discrete spectrum, so no residues of Eisenstein series come into play (see \cite{MW2} for the classification of the discrete non-cuspidal spectrum of $\GL_n$).
By the  results contained in \cite[(4.3)]{JSII} on isobaric automorphic representations, the lemma follows.
\end{proof}

We now prove Proposition \ref{prop-solv}:

\begin{proof}[Proof of Proposition \ref{prop-solv}] Recall that $H^2(\Gal(E/F'),\CC^{\times})=H^2(\widetilde{A}_4,\CC^{\times})=0$.  Thus if $\pi'_E$ is cuspidal then the proposition is \cite[Theorem 2]{Rajan3}.

In the remainder of the proof we will constantly use facts on cyclic prime degree base change established in \cite{Langlands}.  A convenient list of the basic properties (in a more general setting) is given in \cite[Chapter 3, Theorems 4.2 and 5.1]{AC}.

  Assume that $\pi'_E$ is not cuspidal.  By the theory of prime degree base change we must then have  $\pi'_E \cong \chi_0 \boxplus \chi_0^{\sigma_0}$ for some quasi-character $\chi_0:E^{\times} \backslash \A_E^{\times} \to \CC^{\times}$ and some $\sigma_0 \in \Gal(E/F')$.  Therefore we can apply Lemma \ref{lem-const-term} to see that $\widetilde{A}_4$ permutes the two-element set $\{\chi_0,\chi_0^{\sigma_0}\}$.  Since $\Gal(E/F') \cong \widetilde{A}_4$ has no subgroup of index two one has $\chi_0^{\sigma}=\chi_0=\chi_0$ for all $\sigma \in \Gal(E/F')$.

Since $\chi_0$ is $\Gal(E/F')$-invariant and $H^2(\Gal(E/F'),\CC^{\times})=0$, Lemma \ref{lem-bc-param} implies that $\chi_0$ extends to a quasi-character $\chi'$ of  $F'^{\times} \backslash \A_{F'}^{\times}$.  Thus, upon replacing $\pi'$ by $\pi' \otimes \chi'^{-1}$ if necessary, we see that to complete the proof of the proposition it suffices to show that there are $|\Gal(E/F)^{\mathrm{ab}}|$ distinct isomorphism classes of cuspidal automorphic representations $\pi'$ of $\GL_{2}(\A_{F'})$ such that $\pi'_E \cong 1 \boxplus 1$.

We now look more closely at the structure of $\widetilde{A}_4$.
Let $V = \ZZ/2 \times \ZZ/2$ denote the Klein $4$ group and fix an embedding $V \hookrightarrow A_4$.  The inverse image $Q$ of $V$ under the covering map $\widetilde{A}_4 \to A_4$ is isomorphic to the quaternion group; it is a nonabelian group of order $8$.  The subgroup $Q \leq \widetilde{A}_4$ is normal and the quotient $\widetilde{A}_4 \to \widetilde{A}_4/Q \cong \ZZ/3$ induces an isomorphism
$$
\widetilde{A}_4^{\mathrm{ab}} \tilde{\lto} \ZZ/3.
$$
Let $\mu$ be a nontrivial character of $F'^{\times} \backslash \A_{F'}^{\times}$ trivial on $\N_{E/F'}\A_E^{\times}$.  Then since $\widetilde{A}_4^{\mathrm{ab}} \tilde{\lto}\ZZ/3$ we have $\mu^3=1$.  The three cuspidal automorphic representations $\pi', \pi' \otimes \mu$ and $\pi' \otimes \mu^2$ are all nonisomorphic (as can be seen by examining central characters) and all have the property that their base changes to $E$ are isomorphic to $1 \boxplus 1$.  Therefore our task is to show that there are no other isomorphism classes of cuspidal automorphic representations of $\GL_2(\A_{F'})$ that base change to $1 \boxplus 1$.  We note that $(\pi'\otimes \mu^i)_{E^Q}$ is independent of $i$ and is cuspidal by prime degree base change.  Therefore it suffices to show that there is at most one cuspidal automorphic representation $\pi_0$ of $\GL_2(\A_{E^{Q}})$ whose base change to $\GL_2(\A_E)$ is $1 \boxplus 1$.

Let $\pi_0$ be a cuspidal automorphic representation of $\GL_2(\A_{E^Q})$ whose base change to $\GL_2(\A_E)$ is $1 \boxplus 1$.  Choose a chain of subfields $E >E_1>E_2>E^Q$.  We denote by $\chi_1 \in \Gal(E/E_2)^{\wedge}$ a character that restricts nontrivially to $\Gal(E/E_1)$.  The theory of prime degree base change implies that $\pi_{0E_1}$ cannot be cuspidal since $1$ is invariant under $\Gal(E/E_1)$. Hence $\pi_{0E_1}$ must be isomorphic to one of
\begin{align}
1 \boxplus 1, \quad 1 \boxplus \chi_1|_{\A_{E_1}^{\times}},\quad  \textrm{ or } \quad \chi_1|_{\A_{E_1}^{\times}} \boxplus \chi_1|_{\A_{E_1}^{\times}}.
  \end{align}
  Thus applying the theory of prime degree base change again we see that $\pi_{E_2}$ cannot be cuspidal.
Now by assumption $\pi_0$ is cuspidal, and since $\pi_{0E_2}$ is not cuspidal $\pi_0$ is $E_2$-induced.  In particular, $\pi_0=\pi(\phi)$ for an irreducible $L$-parameter $\phi:W_{E^Q}' \to {}^L\GL_{2E_Q}$ (compare \cite[\S 2 C)]{Langlands}).  Note that $\phi$ is necessarily trivial on $W_E'$, and hence
can be identified with a two-dimensional irreducible representation of $\Gal(E/E^Q)$.  There is just one isomorphism class of such representations by the character table for $Q$ recorded in \S \ref{appendix}.  It follows that $\pi_0$ is the unique cuspidal automorphic representation of $\GL_2(\A_{E^Q})$ whose base change to $\GL_{2}(\A_E)$ is $1 \boxplus 1$.  As mentioned above, this implies the proposition.
\end{proof}

As a corollary of the proof, we have the following:

\begin{cor} \label{cor-rho-type} Let $E/F'$ be a Galois extension with $\Gal(E/F') \cong \widetilde{A}_4$, and let $\pi'$ be a cuspidal automorphic representation of $\GL_2(\A_{F'})$.  If $\pi'_E$ is not cuspidal, then $\pi'$ is of $\rho$-type for some $L$-parameter $\rho$ trivial on $W_E'$.
\end{cor}

\begin{proof} The proof of Proposition \ref{prop-solv} implies that there is a character $\chi'$ of $F'^{\times} \backslash \A_{F'}^{\times}$ such that $(\pi' \otimes \chi'^{-1})|_{E} \cong 1 \boxplus 1$, so it suffices to treat the case where $\pi'|_E \cong 1 \boxplus 1$.
By the argument in the proof of proposition \ref{prop-solv} and using the notation therein we have that $\pi'|_{E^Q}=\pi(\phi)$, where $\phi:W_{E^Q}' \to {}^L\GL_{2E^Q}$ is the unique irreducible $L$-parameter trivial on $W_E'$.  Twisting $\pi'$ by an abelian character of $\widetilde{A}_4$ if necessary, we can and do assume that the central character of $\pi'$ is trivial.  Thus we can apply \cite[\S 3]{Langlands} to conclude that
$\pi'=\pi(\phi')$ for some $L$-parameter $\phi':W_{F'}' \to{}^L\GL_{2F'}$.
\end{proof}

We now prove Proposition \ref{prop-solv3}:

\begin{proof}[Proof of Proposition \ref{prop-solv3}] The quasi-character $\chi_0$ descends to a quasi-character $\chi':F'^{\times} \backslash \A_{F'}^{\times} \to \CC^{\times}$ by Lemma \ref{lem-bc-param} and the fact that $H^2(\Gal(E/F'),\CC^{\times})=H^2(\widetilde{A}_4,\CC^{\times})=0$.  Replacing $\pi'$ by $\pi' \otimes \chi'^{-1}$ if necessary, we can and do assume that $\pi'_E\cong 1^{\boxplus 3}$.

In the remainder of the proof we will constantly use facts on cyclic prime degree base change established in \cite{AC}.  A convenient list of the basic properties is given in \cite[Chapter 3, Theorems 4.2 and 5.1]{AC}.

Let $Q \hookrightarrow \widetilde{A}_4$ be as in the proof of Proposition \ref{prop-solv}.
By the theory of cyclic prime-degree base change
$$
\pi'_{E^Q} =\chi_1 \boxplus \chi_2 \boxplus \chi_3
$$
for some characters $\chi_i:F'^{\times} \backslash \A_{F'}^{\times} \to \CC^{\times}$.  Thus, by \cite[Chapter 3, Theorem 6.2]{AC} and its proof, since $\pi'$ is cuspidal we conclude that
$$
\pi'\cong\mathrm{Ind}_{E^Q}^{F'}(\chi_1)
$$
and hence is of $\rho$-type for some irreducible degree three $L$-parameter $\rho:W_{F'}' \to {}^L\GL_{nF'}$ trivial on $W_E'$.  By the character table of $\widetilde{A}_4$ recorded in \S \ref{appendix}, we conclude that $\rho \cong \rho_3$ for $\rho_3$ as in the proposition.

\end{proof}

We also require the following linear independence statement:
\begin{lem} \label{lem-lin-ind} Let $M \leq \GL_n$ be the maximal torus of diagonal matrices.
Let $v$ be a place of $F$.  Suppose that there is a countable set $\mathcal{X}$ of quasi-characters of $M(F_v)$ and that the set $\mathcal{X}$ is stable under the natural action of $W(M,\GL_n)$.  Suppose moreover that  $\{a(\chi_v)\}_{\chi_v \in \mathcal{X}}$ is a set of complex numbers such that for all $f_v \in C_c^{\infty}(M(F_v))^{W(M,\GL_n)}$ one has
$$
\sum_{\chi_v \in \mathcal{X}} a(\chi_v)\mathrm{tr}(\chi_v)(f_v)=0
$$
where the sum is absolutely convergent.  Then
$$
\sum_{W \in W(M,\GL_n)}a(\chi_v \circ W)=0
$$
for each $\chi_v \in \mathcal{X}$.
\end{lem}

\begin{proof}
 By assumption
\begin{align*}
0&=\sum_{\chi_v \in \mathcal{X}} \sum_{W \in W(M,\GL_n)}a(\chi_v)\mathrm{tr}(\chi_v)(f_v \circ W)\\
&=\sum_{\chi_v \in \mathcal{X}}\sum_{W \in W(M,\GL_n)} a(\chi_v)\mathrm{tr}(\chi_v \circ W^{-1})(f_v)\\
&=\sum_{\chi_v \in \mathcal{X}}\mathrm{tr}(\chi_v)(f_v)\sum_{W \in W(M,\GL_n)} a(\chi_v \circ W)
\end{align*}
for all  $f_v \in C_c^{\infty}(M(F_v))$. The result now follows from generalized linear independence of characters (see \cite[Lemma 6.1]{LabLan} and \cite[Lemma 16.l.1]{JacquetLanglands}).
\end{proof}

\subsection{Functoriality implies the trace identities} \label{ssec-func-to-trace}
In this subsection we prove theorems \ref{main-thm-1}, \ref{main-thm-2}, and \ref{main-thm-3}, namely that cases of Langlands functoriality explicated in conjectures \ref{conj-1} and \ref{conj-solv} in the first case,  Conjecture \ref{conj-2} in the second case, and conjectures \ref{conj-32} and \ref{conj-33} in the third case imply the stated trace identities.
By Corollary \ref{cor-aut-trace} the sum
\begin{align*}
\sum_{\pi'} \mathrm{tr}(\pi')(h^1b_{E/F'}(\Sigma_{\phi}^{S_0}(X)))
\end{align*}
is equal to $o(X)$ plus
\begin{align*}
\sum_{\pi'} \mathrm{tr}(\pi')(h^1)
\mathrm{Res}_{s=1}\left( \widetilde{\phi}(s)X^sL(s,(\pi'_E \times \pi'^{\tau \vee}_E)^{S_0})\right).
\end{align*}
Here the sum is over a set of equivalence classes of automorphic representations of $A_{\GL_{nF'}} \backslash \GL_n(\A_{F'})$.  Specifically, for \eqref{11} of Theorem \ref{main-thm-1} we take it to be over $E$-primitive representations,
for \eqref{A21} of Theorem \ref{main-thm-2} and \eqref{31} of Theorem \ref{main-thm-3} we take it to be over all cuspidal representations, and for \eqref{B21} of Theorem \ref{main-thm-2} we take it to be over cuspidal representations not of $\rho$-type for any two-dimensional representation $\rho:W_{F'}' \to {}^L\GL_{2F'}$ trivial on $W_E'$.
The only nonzero contributions to this sum occur when $L(s,(\pi'_E \times \pi'^{\tau\vee}_E)^{S_0})$ has a pole, which implies that $\mathrm{Hom}_I(\pi'_E,\pi'^{\tau}_E) \neq 0$ (see \eqref{ord-pole}).
In this case if $\pi'_E$ is cuspidal it is then invariant under $\langle \Gal(E/F'),\tau \rangle=\Gal(E/F)$ and the pole is simple \eqref{ord-pole}.
 In view of conjectures \ref{conj-1} and \ref{conj-2}, in the setting of theorems \ref{main-thm-1} and \ref{main-thm-2} this implies that if $L(s,(\pi'_E \times \pi'^{\tau\vee}_E)^{S_0})$ has a pole then $\pi'_E$ descends to a cuspidal representation $\pi$ of $F$, whether or not $\pi'_E$ is cuspidal.  On the other hand, the only nonzero contributions to the quantity \eqref{31} in Theorem \ref{main-thm-3} come from representations where $\dim \mathrm{Hom}_I(\pi_E',\pi'^{\tau}_E)=n^2$, and this is the case if and only if $\pi'_E \cong \chi_0^{\boxplus n}$ where $\chi_0:E^{\times} \backslash \A_E^{\times} \to \CC^{\times}$ is a quasi-character invariant under $\Gal(E/F) =\langle \Gal(E/F'),\tau \rangle$.  In these cases $\pi'_E$  descends to a cuspidal representation of $\GL_n(\A_F)$ by conjectures \ref{conj-32} and \ref{conj-33}.

Assume for the moment that we are in the setting of Theorem \ref{main-thm-1}.  In this case by Conjecture \ref{conj-solv} there are precisely $|\Gal(E/F')^{\mathrm{ab}}|$ inequivalent cuspidal representations of $A_{\GL_{nF'}} \backslash \GL_n(\A_{F'})$ that base change to $\pi'_E$, since $\pi'_E$ is cuspidal by the theory of prime degree base change \cite[Chapter 3, Theorems 4.2 and 5.1]{AC}.  With this in mind, the definition of transfer (see \S \ref{ssec-transfers} and Lemma \ref{lem-unr-transf}) completes the proof of the claimed trace identity. We only remark that the absolute convergence of the two sums follows from Corollary \ref{cor-aut-trace} and the fact that $L(s,(\pi'_E \times \pi'^{\tau \vee}_E)^{S_0})$ has a pole of order $\mathrm{Hom}_I(\pi'_E,\pi'^{\tau}_E)$ (see \eqref{ord-pole}).

Now assume that we are in the setting of Theorem \ref{main-thm-2}.  In this case $\pi'_E$ may not be cuspidal, but by Proposition \ref{prop-solv} there are still exactly $|\Gal(E/F')^{\mathrm{ab}}|$ non-isomorphic cuspidal automorphic representations of $\GL_2(\A_{F'})$ that base change to $\pi'_E$.  With this in mind, the claimed trace identity follows as before.

The proof of Theorem \ref{main-thm-3} is essentially the same.  We only point out the most
significant difference, namely that we are claiming that one can consider arbitrary Hecke functions on $C_c^{\infty}(\GL_n(F'_{S'_1})//\GL_n(\OO_{F'S'_1}))$ instead of just those that are base changes of Hecke functions in $C_c^{\infty}(\GL_n(E_{S_{10}})//\GL_n(\OO_{ES_{10}}))$.  The reason this is possible is that for each $\Gal(E/F')$-invariant quasi-character $\chi_0:E^{\times} \backslash \A_E^{\times} \to \CC^{\times}$ there is a unique cuspidal automorphic representation of $\GL_n(\A_{F'})$ such that
$\pi'_E \cong \chi_0^{\boxplus n}$ by propositions \ref{prop-solv} and \ref{prop-solv3}.

\qed

\subsection{The trace identity implies functoriality: first two cases} \label{ssec-trace-to-func}

In this subsection we prove theorems \ref{main-thm-1-conv} and \ref{main-thm-2-conv}, namely that the trace identities of Theorem \ref{main-thm-1}, and \ref{main-thm-2} imply the corresponding cases of functoriality
under the assumption of a supply of transfers (more precisely, under Conjecture \ref{conj-transf}).
By assumption, for all $h$ unramified outside of $S'$ with transfer $\Phi$ unramified outside of $S$ one has an identity
\begin{align} \label{id1}
&\lim_{X \to \infty}|\Gal(E/F')^{\mathrm{ab}}|^{-1}X^{-1}
\sum'_{\pi'} \mathrm{tr}(\pi')(h^1b_{E/F'}(\Sigma^{S_0}_{\phi}(X))
\\&=
\lim_{X \to \infty} X^{-1}\sum'_{\pi} \mathrm{tr}(\pi)(
\Phi^1b_{E/F}(\Sigma^{S_0}_{\phi}(X))). \nonumber
\end{align}
Here for the proof of Theorem \ref{main-thm-1-conv}, the sums are over a set of representatives for the equivalence classes of $E$-primitive automorphic representations of $A_{\GL_{nF'}} \backslash \GL_n(\A_{F'})$ and $A_{\GL_{nF}} \backslash \GL_n(\A_F)$, respectively.  For the proof of Theorem \ref{main-thm-2-conv}, the sums are over a set of representatives for the equivalence classes of cuspidal automorphic representations of $A_{\GL_{nF'}} \backslash \GL_n(\A_{F'})$ and $A_{\GL_{nF}} \backslash \GL_n(\A_F)$, respectively, that are not of $\rho$-type for $\rho$ trivial on $W_E'$.

We start by refining \eqref{id1}.  Notice that each representation $\pi'$ appearing in \eqref{id1} above admits a base change $\pi'_E$ to $\GL_n(\A_E)$ by a series of cyclic base changes.  We claim that $\pi'_E$ is cuspidal.  In Theorem \ref{main-thm-1-conv} we have assumed that $\pi'$ is $E$-primitive.  Hence, by the theory of cyclic base change, $\pi'_E$ must be cuspidal \cite[Chapter 3, Theorem 4.2 and Theorem 5.1]{AC}.  In Theorem \ref{main-thm-2-conv} we have assumed that $\pi'$ is not of $\rho$-type for any $\rho$ trivial on $W_E'$.  Thus $\pi'_E$ is cuspidal by Corollary \ref{cor-rho-type}.

Now applying Corollary \ref{cor-aut-trace} and \eqref{ord-pole} we see that the top line of \eqref{id1} is equal to
\begin{align*}
|\Gal(E/F')^{\mathrm{ab}}|^{-1}
\sum'_{\pi':\pi'_E \cong \pi'^{\tau}_E} \mathrm{tr}(\pi')(h^1)\widetilde{\phi}(1)\mathrm{Res}_{s=1}L(s,(\pi'_E \times \pi'^{\tau \vee}_E)^{S_0}).
\end{align*}
Note that the given residue is nonzero and that this sum is absolutely convergent by Corollary \ref{cor-aut-trace}.  At this point we assume that the function $\Phi_S$ is chosen so that at finite places $v \in S$ where $\Phi_v \not \in C_c^{\infty}(\GL_n(F_v)//\GL_n(\OO_{F_v}))$ the function $\Phi_v$ is of positive type (this is possible by Conjecture \ref{conj-transf}).  Under this assumption we claim that the second line of \eqref{id1} is absolutely convergent.  Indeed, the $L$-function $L(s,(\pi_E \times \pi_E^{\vee})^{S_0})$ of the admissible representation $\pi_E^{S_0} \times \pi_E^{\vee S_0}$ is defined and convergent in some half plane \cite[Theorem 13.2]{Borel}, \cite{LanglProb}, and its Dirichlet series coefficients are positive \cite[Lemma a]{HR}.  Thus the smoothed partial sums
$\mathrm{tr}(\pi)(b_{E/F}(\Sigma^{S_0}_{\phi}(X)))$ have positive coefficients.  The fact that the second line of \eqref{id1} converges absolutely follows.

Now we have refined \eqref{id1} to an identity of absolutely convergent sums
\begin{align} \label{id2}
&|\Gal(E/F')^{\mathrm{ab}}|^{-1}
\sum'_{\pi':\pi'_E \cong \pi'^{\tau}_E} \mathrm{tr}(\pi')(h^1)\widetilde{\phi}(1)\mathrm{Res}_{s=1}L(s,(\pi'_E \times \pi'^{\tau \vee}_E)^{S_0})
\\&=
\sum'_{\pi} \mathrm{tr}(\pi)(\Phi^1)\lim_{X \to \infty} X^{-1}\mathrm{tr}(\pi)(b_{E/F}(\Sigma^{S_0}_{\phi}(X))). \nonumber
\end{align}
where the residues in the top line are nonzero.  Before starting the proof in earnest, we wish to refine \eqref{id2} yet again to an identity where only representations of a given infinity type are involved.  Let $\Psi=\otimes_{w|\infty}\Psi_w \in \otimes_{w |\infty}C_c^{\infty}(M(E_{w}))^{W(M,\GL_n)}$, where $M \leq \GL_n$ is the standard maximal torus of diagonal matrices.  For an irreducible unitary generic admissible  representation $\Pi_{\infty}$ of $\GL_n(E_{\infty})$ (resp.~$\pi_{\infty}$ of $\GL_n(F_{\infty})$) write $\chi_{\Pi_{\infty}}: M(E_{\infty}) \to \CC^{\times}$ (resp.~$\chi_{\pi_{\infty}}: M(F_{\infty}) \to \CC^{\times}$) for a choice of quasi-character whose unitary induction to $\GL_n(E_{\infty})$ (resp.~$\GL_n(F_{\infty})$) is $\Pi_{\infty}$ (resp.~$\pi_{\infty}$).  Here we are using our assumption that $F$ is totally complex.  The quasi-characters $\chi_{\Pi_{\infty}w}$ and  $\chi_{\pi_{\infty}v}$ for infinite places $w$ of $E$ and $v$ of $F$ are determined by $\Pi_{w}$ and $\pi_{v}$, respectively, up to the action of $W(M,\GL_n)$.  Moreover, they determine $\Pi_w$ and $\pi_v$, respectively.

We note that by an application of the descent arguments proving Lemma \ref{lem-archi-transf} the identity \eqref{id2} implies
\begin{align} \label{id3}
&|\Gal(E/F')^{\mathrm{ab}}|^{-1}
\sum'_{\pi':\pi'_E \cong \pi'^{\tau}_E} \mathrm{tr}(\chi_{\pi'_E})(\Psi)\mathrm{tr}(h^{\infty})\widetilde{\phi}(1)\mathrm{Res}_{s=1}L(s,(\pi'_E \times \pi'^{\tau \vee}_E)^{S_0})\\
&=
\sum'_{\pi} \mathrm{tr}(\chi_{\pi'})(\otimes_{v |\infty} (*_{w|v}\Psi_w))\mathrm{tr}(\pi)(\Phi^{\infty})\lim_{X \to \infty} X^{-1}\mathrm{tr}(\pi)(b_{E/F}(\Sigma^{S_0}_{\phi}(X))) \nonumber
\end{align}
where the $*$ denotes convolution in $M(F_{\infty})$ (note we are implicitly choosing isomorphisms $M(E \otimes_F F_{v}) \cong \times_{w|v} M(F_{v})$ for each $v|\infty$ to make sense of this).  Let $\chi_0:M(F_{\infty}) \to \CC^{\times}$ be a given quasi-character.  By Lemma \ref{lem-lin-ind} the identity \eqref{id3} can be refined to
\begin{align} \label{id3.5}
&|\Gal(E/F')^{\mathrm{ab}}|^{-1}
\sum'_{\substack{\pi':\chi_{\pi'_{E\infty}w}=\chi_{0Ew}^{W}\\
\textrm{ for some }W \in W(M,\GL_n)\\ \textrm{for all }w|\infty}} \mathrm{tr}(\chi_{\pi'_E})(\Psi)\mathrm{tr}(h^{\infty})\widetilde{\phi}(1)\mathrm{Res}_{s=1}L(s,(\pi'_E \times \pi'^{\tau \vee}_E)^{S_0})\\
&=
\sum'_{\substack{\pi:\chi_{\pi_{\infty}v}=\chi_{0v}^W\\
 \textrm{ for some }W \in W(M,\GL_n)\\\textrm{for all }v|\infty}} \mathrm{tr}(\chi_{\pi'})(\otimes_{v |\infty} (*_{w|v}\Psi_w))\mathrm{tr}(\pi)(\Phi^{\infty})\lim_{X \to \infty} X^{-1}\mathrm{tr}(\pi)(b_{E/F}(\Sigma^{S_0}_{\phi}(X))).\nonumber
\end{align}

Now by descent \eqref{id3.5} implies the identity
\begin{align} \label{id4}
&|\Gal(E/F')^{\mathrm{ab}}|^{-1}
\sum'_{\substack{\pi':\pi'_E \cong \pi'^{\tau}_E\\\pi'_E \cong \pi_{0\infty E}}} \mathrm{tr}(\pi')(h^1)\widetilde{\phi}(1)\mathrm{Res}_{s=1}L(s,(\pi'_E \times \pi'^{\tau \vee}_E)^{S_0})
\\&=
\sum'_{\pi:\pi_{\infty} \cong \pi_{0\infty}} \mathrm{tr}(\pi)(\Phi_S^1)\lim_{X \to \infty} X^{-1}\mathrm{tr}(\pi)(b_{E/F}(\Sigma^{S_0}_{\phi}(X))) \nonumber
\end{align}
for all irreducible admissible generic unitary representations $\pi_{0\infty}$ of $\GL_n(F_{\infty})$; here, as before, for finite $v$ the functions $\Phi_v$ are assumed to be of positive type when they are ramified (i.e.~not in $C_c^{\infty}(\GL_n(F_v)//\GL_n(\OO_{F_v}))$).  Note in particular that for any given $\Phi_S$ and $h_S$ the sums in \eqref{id4} are finite.

We now start to work with \eqref{id4}.  First consider descent of primitive representations.  Suppose that $\Pi$ is a $\Gal(E/F)$-invariant primitive automorphic representation of $A_{\GL_{nE}} \backslash \GL_n(\A_{E})$.  Then by Conjecture \ref{conj-solv} $\Pi$ descends to a representation $\pi'$ of $A_{\GL_{nF'}} \backslash \GL_n(\A_{F'})$.  Here in the $n=2$ case we are using the fact that
$H^2(\Gal(E/F'),\CC^{\times})=H^2(\widetilde{A}_4,\CC^{\times})=0$.
The existence of a primitive automorphic representation $\pi$ of $A_{\GL_{nF}} \backslash \GL_n(\A_F)$ that is a weak descent of $\Pi$ now follows from \eqref{id4} and a standard argument using the transfer of unramified functions (Lemma \ref{lem-unr-transf}).

In more detail, assume that $\Pi$ and $E/F$ are unramified outside of $S$.
Then choosing $h^{S'}=b_{E/F'}(f^{S_0})$ and $\Phi^{S}=b_{E/F}(f^{S_0})$ for $f^{S_0} \in C_c^{\infty}(\GL_n(\A_{E}^{S_0})//\GL_n(\widehat{\OO}_{E}^{S_0}))$ (which are transfers of each other by Lemma \ref{lem-unr-transf}) the identity
\eqref{id4} implies an identity of the form
\begin{align*}
&\sum'_{\substack{\pi':\pi'_E \cong \pi'^{\tau}_E\\\pi'_E \cong \pi_{0\infty E}}} a(\pi')\mathrm{tr}(\pi')(b_{E/F'}(f^{S_{0}}))
=
\sum'_{\pi:\pi_{\infty} \cong \pi_{0\infty}} c(\pi)\mathrm{tr}(\pi)(b_{E/F}(f^{S_{0}}))\nonumber
\end{align*}
for some $a(\pi') \in \RR_{>0},c(\pi) \in \RR_{ \geq 0}$ (here we are using the fact that we assumed the functions $\Phi_v$ to be of positive type if they are ramified). Applying linear independence of characters, this implies a refined identity
\begin{align*}
&\sum a(\pi')\mathrm{tr}(\pi')(b_{E/F'}(f^S))=
\sum c(\pi)\mathrm{tr}(\pi)(b_{E/F}(f^S))
\end{align*}
where the sum on top (resp.~bottom) is over cuspidal automorphic representations $\pi'$ (resp.~$\pi$) such that the character $\mathrm{tr}(\pi' \circ b_{E/F'})$ (resp.~$\mathrm{tr}(\pi \circ b_{E/F})$) of $C_c^{\infty}(\GL_n(\A_E^{S})//\GL_n(\OO_{E}^{S}))$ is equal to $\mathrm{tr}(\Pi)$.  Thus any of the representations $\pi$ on the right is a weak descent of $\Pi$, and there must be some representation on the right because the sum on the left is not identically zero as a function of $f^{S_0}$.

We also note that the base change is compatible at places $v$ where $\pi$ is an abelian twist of the Steinberg representation by Lemma \ref{lem-EP}.
This proves the statements on descent of cuspidal automorphic representations contained in theorems \ref{main-thm-1-conv} and \ref{main-thm-2-conv}.

Now assume that $\pi$ is an $E$-primitive automorphic representation of $A_{\GL_{nF}} \backslash \GL_n(\A_F)$, and if $n=2$ assume that $\pi$ is not of $\rho$-type for any $\rho:W_F' \to {}^L\GL_{2F}$ trivial on $W_E'$.  By assumption we have that
\begin{align} \label{claim-45}
\lim_{X \to \infty}X^{-1}\mathrm{tr}(\pi)(b_{E/F}(\Sigma^{S_0}_{\phi}(X))) \neq 0.
\end{align}

Let $\pi'$ be a cuspidal automorphic representation of $A_{\GL_{nF'}} \backslash \GL_n(\A_{F'})$ that is not of $\rho'$-type for any $\rho':W_{F'}' \to {}^L\GL_{2F'}$ trivial on $W_E'$.
By Lemma \ref{lem-unr-transf} one has
$$
\mathrm{tr}(\pi'_{v'})(b_{E/F'}(f_w))=\pi'_{Ew}(f_w)
$$
whenever $w$ is a finite place of $E$ dividing $v'$ and $f_w \in C_c^{\infty}(\GL_n(E_w)//\GL_n(\OO_{E_w}))$.
Thus by \eqref{claim-45}, the existence of a weak base change of $\pi$ to $A_{\GL_{nE}} \backslash \GL_n(\A_E)$ follows as before.  This completes the proof of Theorem \ref{main-thm-1-conv} and Theorem \ref{main-thm-2-conv}. \qed

\subsection{Artin representations: Theorem \ref{main-thm-3-conv}} \label{ssec-trace-to-func2}

Let $E/F$ be a Galois extension such that $\Gal(E/F)\cong \widetilde{A}_5$.  We assume that $F$ is totally complex. As above, we fix embeddings $A_4 \hookrightarrow A_5$ and $\ZZ/2 \times \ZZ/2 \hookrightarrow A_4 \hookrightarrow A_5$ and let $\widetilde{A}_4,Q \leq \widetilde{A}_5$ denote the pull-backs of these groups under the quotient map $\widetilde{A}_5 \to A_5$.  Throughout this subsection we assume the hypotheses of Theorem \ref{main-thm-3-conv}.
We fix throughout this subsection a representation $\rho_2:W_F' \to {}^L\GL_{2F}$ trivial on $W_E'$ that has character $\theta_2$ in the notation of \S \ref{appendix}.  There is exactly one other nonisomorphic irreducible degree-two character of $\Gal(E/F)$, namely $\xi \circ \theta_2$ where $\xi \in \Gal(\QQ(\sqrt{5})/\QQ)$.

In this subsection we prove Theorem \ref{main-thm-3-conv}, which asserts that the trace identities of Theorem \ref{main-thm-3} imply that $\rho_2 \oplus \xi \circ \rho_2$ has an associated isobaric automorphic representation.   We note at the outset that the argument is modeled on a well-known argument of Langlands in the tetrahedral case \cite[\S 3]{Langlands}.

The trace identities of Theorem \ref{main-thm-3} involve two different fields that were both denoted by $F'$; it is now necessary to distinguish between them.  We let
$$
F':=E^{\widetilde{A}_4} \leq K :=E^Q.
$$
We require the following lemma:
\begin{lem} \label{lem-Q-autom} There is a cuspidal automorphic representation
$\pi'$ of $\GL_2(\A_{F'})$ and a cuspidal automorphic representation $\sigma$ of $\GL_2(\A_K)$ such that $\pi'=\pi(\rho_2|_{F'})$ and $\sigma=\pi(\rho_2|_K)=\pi'_K$.
\end{lem}

\begin{proof}
One has an automorphic representation $\pi'$ such that  $\pi'=\pi(\rho_2|_{F'})$ by Langlands' work
 \cite[\S 3]{Langlands}; see also \cite[\S 6]{GerLab}. By its construction  $\pi'_K$ is isomorphic to
$\sigma:=\pi(\rho_2|_K)$.
\end{proof}

Choose $\sigma$ and $\pi'$ as in the lemma.
Assuming the trace identities of Theorem \ref{main-thm-3} in the $n=2$ case there are precisely two distinct isomorphism classes of cuspidal automorphic representations represented by, say, $\pi_1,\pi_2$, such that $\pi_{iK} \cong \sigma$.  Using our assumption that $F$ is totally complex this can be proven by  arguments analogous to those used in \S \ref{ssec-trace-to-func}; we only note that
$$
\lim_{X \to \infty} (\frac{d^{3}}{ds^3}(\widetilde{\phi}(s)X^s))^{-1}\mathrm{tr}(\sigma)(b_{E/F'}(\Sigma^{S_0}(X)) \neq 0
$$
since $\mathrm{dim}(\mathrm{Hom}_I(\pi'_E,\pi'^{\tau}_E)) =4$ by construction of $\pi$ (compare Proposition \ref{Perron-prop}).  We emphasize that the trace identity of Theorem \ref{main-thm-3} tells us that $\sigma$ is the unique weak base change of $\pi_i$, which is stronger than the statement that $\sigma_E$ is the unique weak base change of $\pi_i$.  We note in particular that using the transfers supplied in \S \ref{ssec-transfers} we have that the base changes are compatible at finite places $v$ that are unramified in $E/F$ and at all infinite places (which are complex by assumption).  Moreover the $\pi_i$ are unramified outside of the set of places where $E/F$ is ramified.

One expects that upon reindexing if necessary one has
\begin{align*}
\pi_{1} &\stackrel{?}{\cong} \pi(\rho_{2})\\
\pi_{2} &\stackrel{?}{\cong} \pi(\xi \circ \rho_{2}).
\end{align*}
We do not know how to prove this, but we will prove something close to it, namely Corollary \ref{cor-isob} below.

Consider $\mathrm{Sym}^2(\pi')$ and $\mathrm{Sym}^2(\sigma)$; the first is a cuspidal automorphic representation of $\GL_3(\A_{F'})$ by \cite[Theorem 9.3]{GJ} and the second is an isobaric (noncuspidal) automorphic
representation of $\GL_3(\A_K)$ \cite[Remark 9.9]{GJ}.

\begin{lem}  \label{lem-sym}
For $i \in \{1,2\}$ one has
$$
\mathrm{Sym}^2(\pi_i)_K \cong \mathrm{Sym}^2(\sigma)
$$
and
$$
\mathrm{Sym}^2(\pi_i)_{F'} \cong \mathrm{Sym}^2(\pi').
$$
\end{lem}

\begin{proof}
Since $\pi_{iK} \cong \pi'$, it is easy to see that
$\mathrm{Sym}^2(\pi_{i})_{Kv_K} \cong \mathrm{Sym}^2(\sigma)_{v_K}$
for all places $v_K$ of $K$ that are finite and such that $K/F$ and $\sigma_i$ are unramified.
The first statement then follows from strong multiplicity one for isobaric automorphic representations \cite[Theorem 4.4]{JSII}.

Since the $\pi_i$ were defined to be weak descents of $\sigma$, they are in particular weak descents of the isobaric representation $1 \boxplus 1$ of $\GL_2(\A_E)$.  Thus
$$
\lim_{X \to \infty} \left(\frac{d^{8}}{ds^8}(\widetilde{\phi}(s)X^s)\big|_{s=1}\right)^{-1}\mathrm{tr}(\mathrm{Sym}^2(\pi_i))(b_{E/F}(\Sigma^{S_0}(X))) \neq 0
$$
since $\mathrm{tr}(\mathrm{Sym}^2(\pi_i))(b_{E/F}(\Sigma^{S_0}(X))$ is a smoothed partial sum of the Dirichlet series $\zeta_E^{S_0}(s)^9$.
Applying the trace identities of Theorem \ref{main-thm-3} we conclude that $\mathrm{Sym}^2(\pi_{i})$ admits a weak base change $\mathrm{Sym}^2(\pi_i)_{F'}$ to $F'$.  Now $\mathrm{Sym}^2(\pi_i)_{F'}$ and $\mathrm{Sym}^2(\pi')$ both base change to $\mathrm{Sym}^2(\sigma)$.  Since $\mathrm{Sym}^2(\sigma)$ is not cuspidal, this implies that $\mathrm{Sym}^2(\pi') \cong \mathrm{Sym}^2(\pi_i)_{F'}$ \cite[Chapter 3, Theorems 4.2 and 5.1]{AC}.
\end{proof}

For convenience, let $S$ be the set of finite places where $E/F$ is ramified.
Thus the base change from $\pi_i$ to $\sigma$ is compatible outside of $S$ and the
the base changes from $\mathrm{Sym}^2(\pi_i)$ to $\mathrm{Sym}^2(\pi')$ and $\mathrm{Sym}^2(\sigma)$ are all compatible outside of $S$

\begin{lem} \label{lem-F'}
For $i \in \{1,2\}$ the cuspidal automorphic representation $\pi'$ is a weak base change of $\pi_i$:
$$
\pi_{iF'} \cong \pi'.
$$
The base change is compatible for $v \not \in S$.
\end{lem}

\begin{proof}  Fix $i \in \{1,2\}$.  We will verify that the local base change $\pi_{iF'v'}$ is isomorphic to $\pi'_{v'}$ for all places  $v'$ of $F'$ not dividing places in $S$; this will complete the proof of the lemma (notice that the local base change is well-defined even though we do not yet know that $\pi_{iF'}$ exists as an automorphic representation).
Notice that $\pi_{iK} \cong \pi'_K \cong \sigma$ by construction of $\pi_i$.  Thus if $v'$ is a place of $F'$ split in $K/F'$ and not lying above a place of $S$,  then $\pi_{iF'v'} \cong \pi'_v$.

Suppose that $v'$ is a place of $F'$ that is nonsplit in $K/F'$ and not lying above a place of $S$.  Then there is a unique place $v_K|v'$ and $[K_{v_K}:F'_{v'}]=3$.  Notice that $\pi_{i}$ and $\pi'$ have trivial central character by construction.  Thus their Langlands classes are of the form
\begin{align*}
A(\pi_{iF'v'})=\begin{pmatrix} a\zeta & \\ & a^{-1} \zeta^{-1}\end{pmatrix}\quad \textrm{ and }\quad A(\pi_{v'}')=\begin{pmatrix} a &\\ & a^{-1}\end{pmatrix}
\end{align*}
for some $a \in \CC^{\times}$ and some third root of unity $\zeta$.  By Lemma \ref{lem-sym} we have that
$$
\mathrm{Sym}^2(A(\pi_{iF'v'}))=\begin{pmatrix}a^2\zeta^2 & & \\ & 1 & \\&&a^{-2} \zeta^{-2} \end{pmatrix}
$$
is conjugate under $\GL_3(\CC)$ to
$$
\mathrm{Sym}^2(A(\pi'_{v'}))=\begin{pmatrix} a^2 & & \\ & 1 & \\ & & a^{-2} \end{pmatrix}.
$$
Thus $\{a^2,a^{-2}\}=\{a^2\zeta^2,a^{-2}\zeta^{-2}\}$.  If $a^2=a^2\zeta^2$ then $\zeta=1$, proving that $\pi_{iF'v'} \cong \pi'_{v'}$.  If on the other hand $a^2=a^{-2}\zeta^{-2}$ and $\zeta \neq 1$, then
$$
a^4=\zeta^{-2}
$$
and the matrix $\mathrm{Sym}^2(A(\pi'_{v'}))$ has order $6$.  On the other hand, $\mathrm{Sym}^2(A(\pi'_{v'}))$ is the image of a Frobenius element of $\Gal(E/F')$ under the Galois representation corresponding to $\mathrm{Sym}^2(\pi')$.  This Galois representation is the symmetric square of a representation of $\widetilde{A}_4$ with trivial determinant, and hence factors through $A_4$. As $A_4$ has no elements of order $6$, we arrive at a contradiction, proving that $\zeta=1$.  Hence
$\pi_{iF'v'} \cong \pi'_{v'}$.

\end{proof}

Let $\chi \in \widetilde{A}_4^{\wedge}$ be a nontrivial (abelian) character.  Then for all places $v$ of $F$ one has an admissible representation $\mathrm{Ind}_{F'}^{F}(\chi)_v$.  It is equal to
$\otimes_{v'|v} \mathrm{Ind}_{F'_{v'}}^{F_v}(\chi_{v'})$.  Note that one does not know a priori whether or not $\mathrm{Ind}_{F'}^F(\chi)$ is automorphic; proving this in a special case is the subject matter of \cite{KimIcos}. By class field theory we can also view $\mathrm{Ind}_{F'}^F(\chi)$ as an $L$-parameter
$$
\mathrm{Ind}_{F'}^{F}(\chi):W_{F}' \lto {}^L\GL_{5F'}
$$
and with this viewpoint in mind we prove the following lemma:

\begin{lem} \label{lem-ind-autom}
For each $i \in \{1,2\}$ the $L$-parameter $\mathrm{Ind}_{F'}^F(\chi)$ is associated to $\mathrm{Sym}^4(\pi_i)$.  More precisely, $\mathrm{Ind}_{F'}^F(\chi)_v=\pi(\mathrm{Ind}_{F'}^F(\chi)_{v}) \cong \mathrm{Sym}^4(\pi_{iv})$ for all $v \not \in S$.
\end{lem}

We note that $\mathrm{Sym}^4(\pi_i)$ is an automorphic representation of $\GL_5(\A_F)$ by work of Kim \cite{Kim} and Kim-Shahidi \cite[Theorem 3.3.7]{KScusp}.

\begin{proof}
At the level of admissible representations for $v \not \in S$ one has
\begin{align} \label{first-frob}
\mathrm{Sym}^4(\pi_i)^{\vee}_v \otimes \mathrm{Ind}_{F'}^F(\chi)_v \cong \otimes_{v'|v}\mathrm{Ind}_{F'_{v'}}^{F_v}(\mathrm{Sym}^4(\pi'_{v'})^{\vee} \otimes \chi_{v'})
\end{align}
by Frobenius reciprocity.
On the other hand $\pi'^{\vee}=\pi(\rho_{2}|_{F'}^{\vee})$ and
$$
\mathrm{Sym}^4(\rho_2)^{\vee} \cong \mathrm{Sym}^4(\rho_2^{\vee}) \cong \mathrm{Sym}^4(\rho_2) \cong \mathrm{Ind}_{F'}^F(\chi) \cong \mathrm{Ind}_{F'}^F(\chi)^{\vee}
$$
at the level of Galois representations (see Lemma \ref{lem-icosa-reps}).  Thus the right hand side of \eqref{first-frob} is isomorphic to
$$
\otimes_{v'|v}\mathrm{Ind}_{F'_{v'}}^{F_v}(\mathrm{Ind}_{F'_{v'}}^{F_v}(\chi_{v'})^{\vee}|_{F'_{v'}} \otimes \chi_{v'}) \cong \otimes_{v'|v}\mathrm{Ind}_{F'_{v'}}^{F_v}(\chi_{v'})^{\vee} \otimes \mathrm{Ind}_{F'_{v'}}^{F_v}(\chi_{v'})
$$
and we conclude that
\begin{align} \label{sym-isom}
\mathrm{Sym}^4(\pi_i)^{\vee}_v \otimes \mathrm{Ind}_{F'}^F(\chi)_v \cong \otimes_{v'|v}\left(\mathrm{Ind}_{F'_{v'}}^{F_v}(\chi_{v'})^{\vee} \otimes \mathrm{Ind}_{F'_{v'}}^{F_v}(\chi_{v'})\right).
\end{align}

Now if $A$ and $B$ are two square invertible diagonal matrices of rank $n$, the eigenvalues of $A$ can be recovered from knowledge of the eigenvalues of $A \otimes B$ and the eigenvalues of $B$.  With this remark in hand, we see that \eqref{sym-isom} implies that
$$
\mathrm{Sym}^4(\pi_i)_v \cong \mathrm{Ind}_{F'}^F(\chi)_v
$$
for all $v \not \in S$.
\end{proof}

With this preparation in place, we make a step towards proving that $\rho_2$ and $\pi_1$
are associated:

\begin{lem} \label{lem-places}
Let $v \not \in S$.  One has $\pi_{1v} \cong \pi(\rho_{2v})$ or $\pi_{1v} \cong \pi(\xi \circ \rho_{2v})$.

\end{lem}

\begin{proof}  For infinite places we use our running assumption that $F$ is totally complex together with Lemma \ref{lem-F'}.  This allows one to deduce the lemma in this case.  Assume now that $v$ is finite and choose $v'|v$.  By Lemma \ref{lem-F'}, up to conjugation, the Langlands class of $\pi_{1v}$ and the
Frobenius eigenvalue of $\rho_2(\mathrm{Frob}_v)$ satisfy
\begin{align*}
A(\pi_{1v})=\begin{pmatrix} a\zeta & \\ & a^{-1} \zeta^{-1}\end{pmatrix}\quad \textrm{ and } \quad \rho_2(\mathrm{Frob}_v)=\begin{pmatrix} a &\\ & a^{-1}\end{pmatrix}
\end{align*}
where $\zeta$ is a $[F'_{v'}:F_v]$-root of unity.  Thus if there is a place $v'|v$ such that $[F'_{v'}:F_v]=1$, then we are done.  Since $[F':F]=5$, there are two other cases to consider, namely where there is a single $v'|v$ of relative degree $5$ and when there are two places $v'_2,v'_3|v$, one of them of relative degree $2$ and the other of relative degree $3$.

By Lemma \ref{lem-ind-autom} the two matrices
\begin{align} \label{are-conj}
\begin{pmatrix} a^4\zeta^4 & & & & \\
& a^2\zeta^2 & & &\\ & & 1 & &  \\ & & & a^{-2} \zeta^{-2} & \\
& & & & a^{-4}\zeta^{-4}\end{pmatrix} \quad \textrm{ and }\quad \begin{pmatrix} a^4 & & & & \\
& a^2 & & &\\ & & 1 & &  \\ & & & a^{-2} & \\
& & & & a^{-4}\end{pmatrix}
\end{align}
are conjugate.  We will use this fact and a  case-by-case argument to prove the lemma.  Assume $[F'_{v'}:F_{v}]=5$.  In this case $a+a^{-1}$ is $\pm \frac{\sqrt{5}-1}{2}$ or $\pm \frac{-\sqrt{5}-1}{2}$ by the character table of $\widetilde{A}_5$ in \S \ref{appendix} above, which implies that $a=\pm\nu$ for a primitive fifth root of unity $\nu$.  We conclude from the conjugacy of the two
matrices \eqref{are-conj} that $\zeta \neq \nu^{-1}$.  On the other hand, if $\zeta$ is any other fifth root of unity then the matrix $A(\pi_{1v})$ is conjugate to either $\rho_2(\mathrm{Frob}_v)$ or $\xi \circ \rho_2(\mathrm{Frob}_v)$, where as above $\xi$ is the generator of $\Gal(\QQ(\sqrt{5})/\QQ)$.  Thus the lemma follows in this case.

Assume now that $[F'_{v'}:F_v]=3$; this is the last case we must check.  By consulting the character table of $\widetilde{A}_5$ in \S \ref{appendix} we see that $a+a^{-1}=\pm 1$ which implies $a$ is a primitive $6$th root of unity or a primitive $3$rd root of unity.  By the conjugacy of the matrices \eqref{are-conj} we conclude that $\zeta \neq \pm a^{-1}$.  Thus if $a$ is a primitive $3$rd root of unity the matrices
$$
\begin{pmatrix} a \zeta & \\ & a^{-1} \zeta^{-1} \end{pmatrix}, \, \, \, \begin{pmatrix} a & \\ & a^{-1} \end{pmatrix}
$$
are either equal (if $\zeta=1$) or conjugate (if $\zeta \neq a^{-1}$ is a nontrivial $3$rd root of unity).  Now suppose that $a$ is a primitive $6$th root of unity; by replacing $a$ by $a^{-1}$ if necessary we may assume that $a=e^{2 \pi i/6}$.  In this case the right matrix in \eqref{are-conj} has eigenvalues $e^{\pm 2 \pi i/3},1$ (the first two with multiplicity two and the last with multiplicity one).  Since $\zeta \neq \pm a^{-1}$, we must have $\zeta=1$ or $\zeta=e^{-2\pi i /3}$.  In the former case $A(\pi_{1v})$ and $\rho_2(\mathrm{Frob}_v)$ are equal and in the latter case they are conjugate.
\end{proof}

Another way of stating the lemma that appears more ``global'' is the following corollary:

\begin{cor}  \label{cor-isob} One has
$$
\pi_1 \boxplus \pi_2 \cong \pi(\rho_2 \oplus \xi \circ \rho_2).
$$\
\end{cor}

This is precisely Theorem \ref{main-thm-3-conv}.

\begin{proof}
This follows from Lemma \ref{lem-places} and \cite[Proposition 4.5]{HennCyc}.  To apply \cite[Proposition 4.5]{HennCyc} one uses the fact that the isobaric sum $\pi_1 \boxplus \pi_2$ is necessarily locally generic (see \cite[\S 0.2]{Bernstein} for  the nonarchimedian case, which is all we use).
\end{proof}

Finally, we prove Corollary \ref{cor-artin-cases}:
\begin{proof}[Proof of Corollary \ref{cor-artin-cases}]
In the notation above, Corollary \ref{cor-isob} implies the following isomorphisms at the level of admissible representations:
\begin{align*}
\mathrm{Sym}^3(\pi_1) &\cong \pi(\mathrm{Sym}^3(\rho_2))\\
\pi_1 \boxtimes \pi_2 &\cong \pi(\rho_2 \otimes \xi \circ \rho_2)\\
\mathrm{Sym}^4(\pi_1) &\cong \pi(\mathrm{Sym}^4(\rho_2))\\
\mathrm{Sym}^2(\pi_1) \boxtimes \pi_2 &\cong \pi(\mathrm{Sym}^2(\rho_2) \otimes \xi \circ \rho_2).
\end{align*}
Notice that any irreducible representation of $\Gal(E/F)$ of dimension greater than $3$ is on this list by Lemma \ref{lem-icosa-reps}.  Therefore to complete the proof of the corollary it suffices to recall that all of the representations on the left are known to be automorphic.
  More precisely, the $\mathrm{Sym}^3$ lift was treated by work of Kim and Shahidi \cite[Theorem B]{KiSh}.  The Rankin product $\pi_1 \boxtimes \pi_2$ is automorphic by work of Ramakrishnan \cite[Theorem M]{RRS}.  The fact that the symmetric fourth is automorphic follows from \cite[Theorem 3.3.7]{KScusp} (see also \cite[Theorem 4.2]{KimIcos}).  Finally, for the last case,
one can invoke \cite[Theorem A]{KiSh} and  \cite[Proposition 4.1]{SW}.
\end{proof}

\section{Some group theory} \label{sec-groups}

In this section we explain why two group-theoretic assumptions we have made in theorems \ref{main-thm-1} and \ref{main-thm-1-conv} are essentially no loss of generality.

\subsection{Comments on universal perfect central extensions}

\label{ssec-upce}

The underlying goal of this paper is to study the functorial transfer conjecturally attached to the map of $L$-groups
$$
b_{E/F}:{}^L\GL_{nF} \lto {}^L\mathrm{Res}_{E/F}\GL_{nE}
$$
for Galois extensions $E/F$.  We explain how ``in principle'' this can be reduced to the study of Galois extensions $E/F$ where $\Gal(E/F)$ is the universal perfect central extension of a finite simple group.  Given a Galois extension $E/F$, we can find a chain of subextensions $E_0=E \geq E_1 \geq \cdots  \geq E_m=F$ such that $E_i/E_{i+1}$ is Galois with simple Galois group.  Using this, one can in principle reduce the study of arbitrary Galois extensions to the study of extensions with simple Galois group\footnote{Of course, this reduction will be subtle; see \cite{LR} and \cite{Rajan3} for the solvable case.}.  If the extension is cyclic, then we can apply the body of work culminating in the book of Arthur and Clozel \cite{AC}.  We therefore consider the case where $\Gal(E/F)$ is a finite simple nonabelian group.

Assume for the moment that $\Gal(E/F)$ is a finite simple nonabelian group.  There exists an extension $L/E$ such that $L/F$ is Galois,
\begin{align} \label{star}
\begin{CD}
 1 @>>> \Gal(L/E) @>>>\Gal(L/F) @>>> \Gal(E/F)@>>>1
\end{CD}
\end{align}
is a central extension and
\begin{align} \label{abundant}
\Gal(L/L \cap EF^{\mathrm{ab}}) \cong H^2(\Gal(E/F),\CC^{\times})^{\wedge},
\end{align}
 where $F^{\mathrm{ab}}$ is the maximal abelian extension of $F$ (in some algebraic closure) \cite[Theorem 5]{Miyake} (in fact Miyake's theorem is valid for an arbitrary Galois extension $E/F$)\footnote{Here the $\wedge$ denotes the dual, so $H^2(\Gal(E/F),\CC^{\times})^{\wedge} \cong H^2(\Gal(E/F),\CC^{\times})$ since $H^2(\Gal(E/F),\CC^{\times})$ is finite abelian.}.
Such an extension $L$ is called an abundant finite central extension in loc.~cit.  Choose an abelian extension $F'/F$ such that
$$
L \cap EF^{\mathrm{ab}}=EF'.
$$
We claim that $\Gal(L/F')$ is the universal perfect central extension of a finite simple group.  Indeed, the central extension \eqref{star} restricts to induce
 a central extension
$$
\begin{CD}
1 @>>> \Gal(L/EF') @>>>\Gal(L/F') @>>> \Gal(E/F)@>>>1
\end{CD}
$$
Moreover, $L \cap EF^{\mathrm{ab}}=EF'$ implies $L \cap F^{\mathrm{ab}}=F'$ since $\Gal(E/F)$ is a simple nonabelian group and therefore $\Gal(L/F')$ is perfect.  By \eqref{abundant}, we conclude that $\Gal(L/F')$ is the universal perfect central extension of the finite simple group $\Gal(E/F)$ \cite[Proposition 4.228]{Gorenstein}.

We observe that if we understand the functorial lifting conjecturally defined by $b_{L/F'}$, then we can ``in principle'' use abelian base change to understand the functorial lifting conjecturally defined by $b_{E/F}$.  Thus assuming that $\Gal(E/F)$ is the universal perfect central extension of a finite simple group from the outset is essentially no loss of generality.

\subsection{Generating $\Gal(E/F)$} \label{gen-gal}

In the statement of Theorems \ref{main-thm-1} and \ref{main-thm-1-conv}, we required that $\Gal(E/F)=\langle \tau, \Gal(E/F') \rangle$ for some subfield $E \geq F' \geq F$ with $E/F'$ solvable and some element $\tau$.  We also placed restrictions on the order of $\Gal(E/F')$.  The theorems we recall in this subsection indicate that these restrictions are little or no loss of generality, and also demonstrate that one has a great deal of freedom in choosing generators of universal perfect central extensions of finite simple groups.  To state some results, recall that a finite group $G$ is quasi-simple if $G/Z_G$ is a nonabelian simple group and $G$ is perfect.  Thus universal perfect central extensions of simple nonabelian groups are quasi-simple.

\begin{thm}[Guralnick and Kantor] \label{thm-GK} Let $G$ be a quasi-simple group.  Let $x \in G$ that is not in the center $Z_G$ of $G$.  Then there is an element $g \in G$ such that $\langle x,g \rangle=G$.
\end{thm}

\begin{proof}
Let $\bar{x}$ be the image of $x$ in $G/Z_G$.  Then there exists a $\bar{g} \in G/Z_G$ such that
$\langle \bar{x}, \bar{g}\rangle=G/Z_G$ by \cite[Corollary]{GurKant}.
We simply let $g \in G$ be any element mapping to $\bar{g}$.
\end{proof}

For applications to base change and descent of automorphic representations of $\GL_2$, preliminary investigation indicates that the primes $2$ and $3$ are troublesome.  With this in mind, the following theorem might be useful (see \cite[Corollary 8.3]{GurM}):

\begin{thm}[Guralnick and Malle] \label{thm-good} Let $G$ be a quasi-simple group.   Then $G$ can be generated by two elements of order prime to $6$. \qed
\end{thm}

\section*{Acknowledgments}

The author would like to thank A.~Adem for information on finite simple groups, R.~Guralnick and G.~Malle for including a proof of Theorem \ref{thm-good} in \cite{GurM} and R.~Guralnick for correcting mistakes in \S \ref{gen-gal}. H.~Hahn,
R.~Langlands, S.~Morel, Ng\^o B.~C., P.~Sarnak, and N.~Templier deserve thanks for many useful conversations.
The author is also grateful for the encouragement of R.~Langlands, Ng\^o B.~C., D.~Ramakrishnan, and P.~Sarnak.
R.~Langlands deserves additional thanks in particular for encouraging the author to record the results in this paper.

\quash{ He also thanks J.~C.~for everything.  }


\end{document}